\documentclass[12pt]{article}

\usepackage{amsmath}
\usepackage{amssymb}
\usepackage{amsthm}
\usepackage{amsfonts}
\usepackage{amscd}
\usepackage{mathtools}
\usepackage[dvipsnames,svgnames,table]{xcolor}
\usepackage{graphicx}
\usepackage[T1]{fontenc}
\usepackage{lmodern}
\usepackage[nottoc]{tocbibind}
\usepackage{enumitem}
\usepackage[font=it,labelfont=bf]{caption}
\usepackage[numbers,sort&compress]{natbib}
\usepackage{hyperref}
\usepackage{geometry}
\usepackage[medium]{titlesec}
\usepackage{etoolbox}
\usepackage{algorithm}
\usepackage{algorithmicx}
\usepackage{tikz}
\usepackage{verbatim}
\usepackage{cleveref}
\usepackage{titling}
\usepackage{arydshln}

\theoremstyle{plain}
\newtheorem{lemma}{Lemma}
\newtheorem{proposition}[lemma]{Proposition}

\newtheorem{conjecture}[lemma]{Conjecture}
\newtheorem{openproblem}[lemma]{Open problem}

\theoremstyle{definition}
\newtheorem{definition}[lemma]{Definition}
\newtheorem{remark}[lemma]{Remark}
\newtheorem{example}[lemma]{Example}
\newtheorem{notation}[lemma]{Notation}
\theoremstyle{plain}
\newtheorem{maintheorem}{Theorem}
\newtheorem{maincorollary}{Corollary}

\hypersetup{colorlinks,linkcolor=RoyalBlue,citecolor=PineGreen,urlcolor=RoyalBlue}
\makeatletter\patchcmd{\ttlh@hang}{\parindent\z@}{\parindent\z@\leavevmode}{}{}\patchcmd{\ttlh@hang}{\noindent}{}{}{}\makeatother 
\titlespacing*{\section}{0pt}{1mm}{1mm}
\titlespacing*{\subsection}{0pt}{1mm}{1mm}
\titlespacing*{\paragraph}{0pt}{1mm}{1mm}
\newcommand{\myspace}{\setlength{\abovedisplayskip}{2mm}\setlength{\belowdisplayskip}{2mm}}

\newenvironment{Mlist}{\begin{itemize}[topsep=0pt,itemsep=0pt,leftmargin=7mm]}{\end{itemize}}
\newenvironment{Menum}{\begin{enumerate}[topsep=0pt,itemsep=0pt,leftmargin=8mm]}{\end{enumerate}}

\newcommand{\scale}[2]{\text{\scalebox{#1}{$#2$}}}

\newcommand{\PRP}[1]{Proposition~\ref{prp:#1}}
\newcommand{\LEM}[1]{Lemma~\ref{lem:#1}}
\newcommand{\LEMS}[1]{\Cref{#1}}
\newcommand{\THM}[1]{Theorem~\ref{thm:#1}}
\newcommand{\COR}[1]{Corollary~\ref{cor:#1}}
\newcommand{\DEF}[1]{Definition~\ref{def:#1}}
\newcommand{\RMK}[1]{Remark~\ref{rmk:#1}}
\newcommand{\EQN}[1]{\eqref{eqn:#1}}
\newcommand{\SEC}[1]{\textsection\,\ref{sec:#1}}
\newcommand{\APP}[1]{Appendix~\ref{sec:#1}}
\newcommand{\FIG}[1]{Figure~\ref{fig:#1}}
\newcommand{\TAB}[1]{Table~\ref{tab:#1}}
\newcommand{\EXM}[1]{Example~\ref{exm:#1}}
\newcommand{\EXMS}[1]{\Cref{#1}}
\newcommand{\END}{\hfill $\vartriangleleft$}
\newcommand{\ALG}[1]{Algorithm~\ref{alg:#1}}
\newcommand{\NTN}[1]{Notation~\ref{ntn:#1}}

\newcommand{\AXM}[1]{\ref{axm:#1}}
\newcommand{\PPP}[1]{\ref{ppp:#1}}
\newcommand{\df}[1]{\textit{#1}}
\newcommand{\Wlog}{without loss of generality }
\newcommand{\resp}{respectively}
\newcommand{\st}{such that }
\newcommand{\wrt}{with respect to }

\renewcommand{\c}{\colon}
\newcommand{\set}[2]{\{#1 : #2\}}
\newcommand{\dto}{\dasharrow}

\newcommand{\bas}[1]{\langle #1\rangle}
\newcommand{\sub}[1]{\downarrow\{#1\}}
\newcommand{\sng}{\operatorname{sing}}

\newcommand{\id}{\operatorname{id}}
\newcommand{\bp}{\operatorname{bp}}
\renewcommand{\l}{\ell}

\newcommand{\ii}{\mathfrak{i}}
\newcommand{\C}{{\mathbb{C}}}
\newcommand{\R}{{\mathbb{R}}}

\newcommand{\Z}{{\mathbb{Z}}}
\newcommand{\A}{{\mathbb{A}}}
\newcommand{\B}{{\mathbb{B}}}
\renewcommand{\P}{{\mathbb{P}}}

\newcommand{\V}{{\mathbb{V}}}
\newcommand{\E}{{\mathbb{E}}}
\newcommand{\cA}{{\mathcal{A}}}
\newcommand{\cB}{{\mathcal{B}}}
\newcommand{\cC}{{\mathcal{C}}}
\newcommand{\cF}{{\mathcal{F}}}
\newcommand{\cG}{{\mathcal{G}}}
\newcommand{\cH}{{\mathcal{H}}}
\newcommand{\cI}{{\mathcal{I}}}

\newcommand{\cL}{{\mathcal{L}}}
\newcommand{\cM}{{\mathcal{M}}}
\newcommand{\cR}{{\mathcal{R}}}
\newcommand{\cQ}{{\mathcal{Q}}}
\newcommand{\cU}{{\mathcal{U}}}
\newcommand{\cV}{{\mathcal{V}}}
\newcommand{\fF}{{\mathfrak{F}}}
\newcommand{\fH}{{\mathfrak{H}}}
\newcommand{\oH}{{\hat{H}}}
\newcommand{\oF}{{\hat{F}}}
\newcommand{\oM}{{\hat{M}}}
\newcommand{\oN}{{\hat{N}}}

\newcommand{\oG}{{\hat{G}}}

\newcommand{\oO}{{\hat{O}}}
\newcommand{\bhA}{{\widehat{\mathbf{A}}}}
\newcommand{\bA}{\mathbf{A}}
\newcommand{\bB}{\mathbf{B}}
\newcommand{\bI}{\mathbf{I}}
\newcommand{\tx}{{\tilde{x}}}
\newcommand{\ty}{{\tilde{y}}}
\newcommand{\tT}{\widetilde{T}}
\newcommand{\tla}{{\tilde{\lambda}}}
\newcommand{\cc}{\operatorname{c}}
\newcommand{\vv}{\operatorname{v}}
\newcommand{\ee}{\operatorname{e}}
\renewcommand{\tt}{\operatorname{t}}
\newcommand{\mm}{\operatorname{m}}
\newcommand{\gR}{\Upsilon_R}
\newcommand{\gS}{\Upsilon_S}
\newcommand{\e}{\varepsilon}
\newcommand{\start}{\operatorname{start}}
\newcommand{\wend}{\operatorname{end}}
\newcommand{\codim}{\operatorname{codim}}

\usetikzlibrary{calc}
\usetikzlibrary{backgrounds}
\colorlet{colG}{DarkSeaGreen}
\definecolor{colR}{HTML}{CC6677}
\definecolor{colO}{HTML}{DDCC77}
\definecolor{colB}{HTML}{6699CC}
\colorlet{colY}{Gold!90!black}
\colorlet{colGray}{black!10}
\colorlet{cole}{black!35} 
\colorlet{colv}{black!60} 
\colorlet{colvf}{black!10} 
\colorlet{colanchor}{colG} 
\colorlet{colanchorf}{colG!30!white} 
\colorlet{colvm}{colR} 
\colorlet{colvm2}{colB} 
\colorlet{colvmf}{colR!30!white} 
\colorlet{colvmf2}{colB!30!white} 
\colorlet{colvh}{colO!80!Orange} 
\colorlet{colvhf}{colvh!30!white} 
\colorlet{coleh}{colvh} 
\tikzstyle{vertex}=[circle,inner sep=0pt, minimum size=5pt, draw=colv,fill=colvf]
\tikzstyle{anchorvertex}=[vertex, draw=colanchor,fill=colanchorf]
\tikzstyle{mvertex}=[vertex, draw=colvm,fill=colvmf]
\tikzstyle{hvertex}=[vertex, draw=colvh,fill=colvhf]
\tikzstyle{cvertex}=[line width=0.2mm, draw=colv,fill=colvf]
\tikzstyle{cvertexr}=[radius=0.1]
\tikzstyle{anchorcvertex}=[cvertex, draw=colanchor,fill=colanchorf]
\tikzstyle{mcvertex}=[cvertex, draw=colvm,fill=colvmf]
\tikzstyle{hcvertex}=[cvertex, draw=colvh,fill=colvhf]
\tikzstyle{edge}=[line width=1.75pt, draw=cole]
\tikzstyle{edgeh}=[edge, draw=coleh]
\tikzstyle{labelsty}=[font=\scriptsize]
\tikzstyle{traj}=[line width=1pt, draw=colvm,densely dashed]
\tikzstyle{traj2}=[line width=1pt, draw=colvm2,densely dashed]
\newcommand{\pl}{(0,0)}
\newcommand{\pr}{(3,0)}
\newcommand{\PM}{(1.5,1)}
\newcommand{\pL}{(0,1)}
\newcommand{\pR}{(3,1)}
\newcommand{\pM}{(1.5,1.8)}
\newcommand{\pu}{(0,-1)}
\newcommand{\pv}{(3,-1)}

\newcommand{\pA}{(0,2)}
\newcommand{\pB}{(3,2)}
\newcommand{\pe}{(2,0.5)}
\newcommand{\pf}{(1,1)}
\newcommand{\pg}{(1,2)}
\newcommand{\ph}{(2,1.5)}

\newcommand{\pF}{(1.8,2.6)}
\newcommand{\pE}{(2,-0.5)}
\newcommand{\pz}{(3.5,0.5)}
\newcommand{\Aedge}[3] {\draw[edge, draw=colanchor] #1 -- #2;}
\newcommand{\Iedge}[3] {\draw[edge] #1 -- #2;}
\newcommand{\DAedge}[2] {\draw[very thick, draw=colanchor, dashed] #1 -- #2;}
\newcommand{\DIedge}[2] {\draw[very thick, draw=cole, dotted] #1 -- #2;}
\newcommand{\DAvert}[1] {\draw[draw=colanchor, fill=colanchorf, dotted] #1 circle [cvertexr];}
\newcommand{\DMvert}[1] {\draw[draw=colvm, fill=colvmf, dotted] #1 circle [cvertexr];}
\newcommand{\PedgeE}[2] {\draw[edge] #1 -- #2 node[labelsty,midway,right] {$+$};}
\newcommand{\PedgeW}[2] {\draw[edge] #1 -- #2 node[labelsty,midway,left]  {$+$};}
\newcommand{\PedgeN}[2] {\draw[edge] #1 -- #2 node[labelsty,midway,above] {$+$};}

\newcommand{\MedgeE}[2] {\draw[edge] #1 -- #2 node[labelsty,midway,right] {$-$};}
\newcommand{\MedgeW}[2] {\draw[edge] #1 -- #2 node[labelsty,midway,left]  {$-$};}

\newcommand{\MedgeS}[2] {\draw[edge] #1 -- #2 node[labelsty,midway,below] {$-$};}
\newcommand{\AvertN}[2] {\node[anchorvertex,label={[labelsty]above:$#2$}] at #1 {};}
\newcommand{\AvertS}[2] {\node[anchorvertex,label={[labelsty]below:$#2$}] at #1 {};}
\newcommand{\AvertW}[2] {\node[anchorvertex,label={[labelsty]left:$#2$}] at #1 {};}
\newcommand{\AvertE}[2] {\node[anchorvertex,label={[labelsty]right:$#2$}] at #1 {};}
\newcommand{\MvertN}[2] {\node[mvertex,label={[labelsty]above:$#2$}] at #1 {};}
\newcommand{\MvertS}[2] {\node[mvertex,label={[labelsty]below:$#2$}] at #1 {};}

\newcommand{\MvertE}[2] {\node[mvertex,label={[labelsty]right:$#2$}] at #1 {};}
\newcommand{\IvertN}[2] {\node[vertex,label={[labelsty]above:$#2$}] at #1 {};}
\newcommand{\IvertS}[2] {\node[vertex,label={[labelsty]below:$#2$}] at #1 {};}
\newcommand{\IvertW}[2] {\node[vertex,label={[labelsty]left:$#2$}] at #1 {};}
\newcommand{\IvertE}[2] {\node[vertex,label={[labelsty]right:$#2$}] at #1 {};}

\newcommand{\markerN}[2] {\draw[draw=white, fill=white, line width=0.2mm] #1 circle [cvertexr] node[above, white] {$#2$};}
\newcommand{\markerE}[2] {\draw[draw=white, fill=white, line width=0.2mm] #1 circle [cvertexr] node[right, white] {$#2$};}
\newcommand{\Qaa}{(0,0)}
\newcommand{\Qea}{(8,0)}
\newcommand{\Qbb}{(2,2.5)}
\newcommand{\Qkk}{(7,2)}
\newcommand{\Qcc}{(2.5,4)}
\newcommand{\Qjj}{(6,4)}
\newcommand{\Qbc}{(-1,3)}
\newcommand{\Qbd}{(2,6)}
\newcommand{\Qcd}{(6,7)}
\newcommand{\Qbe}{(2,8)}
\newcommand{\Paa}{(0,0)}
\newcommand{\Pba}{(2,0)}
\newcommand{\Pca}{(4,0)}

\newcommand{\Pea}{(8,0)}
\newcommand{\Pab}{(0,2)}
\newcommand{\Pbb}{(2,2)}
\newcommand{\Pcb}{(4,2)}
\newcommand{\Pdb}{(6,2)}

\newcommand{\Pac}{(0,4)}
\newcommand{\Pbc}{(2,4)}
\newcommand{\Pcc}{(4,4)}

\newcommand{\Pad}{(0.8,5)}
\newcommand{\Pbd}{(2,6)}
\newcommand{\Pcd}{(4,6)}
\newcommand{\Pdd}{(6,6)}
\newcommand{\Pae}{(0,8)}
\newcommand{\Pbe}{(2,8)}
\newcommand{\Pce}{(4,8)}
\newcommand{\Pde}{(6,8)}

\newcommand{\Paf}{(0,10)}
\newcommand{\Pbf}{(2,10)}

\newcommand{\Pef}{(8,10)}
\newcommand{\Pmm}{(4,9)}
\newcommand{\Pkk}{(5,3)}

\newcommand{\Pii}{(1,5)}
\newcommand{\Pjj}{(5,5)}
\newcommand{\BUedge}[3]  {\draw[edge] #1 -- #2;}
\newcommand{\BDedge}[2]  {\draw[edge, dotted] #1 -- #2;}
\newcommand{\BVedge}[3]  {\draw[edgeh] #1 -- #2;}
\newcommand{\BEedge}[2]  {\draw[edgeh, dotted] #1 -- #2;}
\newcommand{\BIvertN}[2] {\node[vertex,label={[labelsty]above:$#2$}] at #1 {};}
\newcommand{\BIvertS}[2] {\node[vertex,label={[labelsty]below:$#2$}] at #1 {};}
\newcommand{\BIvertW}[2] {\node[vertex,label={[labelsty]left:$#2$}] at #1 {};}
\newcommand{\BIvertE}[2] {\node[vertex,label={[labelsty]right:$#2$}] at #1 {};}

\newcommand{\BBvertN}[2] {\node[hvertex,label={[labelsty]above:$#2$}] at #1 {};}
\newcommand{\BBvertS}[2] {\node[hvertex,label={[labelsty]below:$#2$}] at #1 {};}
\newcommand{\BBvertW}[2] {\node[hvertex,label={[labelsty]left:$#2$}] at #1 {};}
\newcommand{\BBvertE}[2] {\node[hvertex,label={[labelsty]right:$#2$}] at #1 {};}

\newcommand{\BAvertN}[2] {\node[anchorvertex,label={[labelsty]above:$#2$}] at #1 {};}
\newcommand{\BAvertS}[2] {\node[anchorvertex,label={[labelsty]below:$#2$}] at #1 {};}
\newcommand{\BAvertW}[2] {\node[anchorvertex,label={[labelsty]left:$#2$}] at #1 {};}
\newcommand{\BAvertE}[2] {\node[anchorvertex,label={[labelsty]right:$#2$}] at #1 {};}
\newcommand{\BMvertN}[2] {\node[mvertex,label={[labelsty]above:$#2$}] at #1 {};}
\newcommand{\BMvertS}[2] {\node[mvertex,label={[labelsty]below:$#2$}] at #1 {};}
\newcommand{\BMvertW}[2] {\node[mvertex,label={[labelsty]left:$#2$}] at #1 {};}

\newif\ifshowtikz
\let\oldtikzpicture\tikzpicture
\let\oldendtikzpicture\endtikzpicture
\renewenvironment{tikzpicture}{\ifshowtikz\expandafter\oldtikzpicture\else\comment\fi}{\ifshowtikz\oldendtikzpicture\else\endcomment\fi}
\showtikztrue

\begin{document}
\myspace

\begin{center}
\LARGE
Coupler curves of moving graphs and counting realizations of rigid graphs
\\[5mm]\large
Georg Grasegger, Boulos El Hilany, Niels Lubbes
\\[5mm]\large
\today
\end{center}

\begin{abstract}
A calligraph is a graph that for almost all edge length assignments moves
with one degree of freedom in the plane,
if we fix an edge and consider the vertices as revolute joints.
The trajectory of a distinguished vertex of the calligraph is called its coupler curve.
To each calligraph we uniquely assign a vector consisting of three integers.
This vector bounds the degrees and geometric genera of irreducible components of the coupler curve.
A graph, that up to rotations and translations
admits finitely many, but at least two, realizations into the plane
for almost all edge length assignments, is a union of two calligraphs.
We show that this number of realizations is equal to a certain inner product of the vectors
associated to these two calligraphs.
As an application we obtain an improved algorithm for counting numbers of realizations,
and by counting realizations we characterize invariants of coupler curves.


\textbf{MSC Class:} 52C25, 70B15, 14C20
\end{abstract}

\begingroup
\def\addvspace#1{\vspace{-2mm}}
\tableofcontents
\endgroup

\section{Introduction}
\label{sec:intro}

In this article we count realizations of rigid graphs into the plane
and investigate invariants of the trajectories of vertices of graphs
that move in the plane.
We start with a warm up example inspired by \citep[Figure~5]{rigid-num-borcea}
in order to explain our main result and the state of art.
The quoted definitions in this introduction are meant for building some intuition,
and are made precise in~\SEC{main}.
We conclude this introduction with an overview of the remaining article.

\textbf{Warm up example.} Let us consider the graph~$\cC_3$ in \FIG{intro}
and assign to the edges $\{0,3\}$, $\{1,3\}$ and $\{2,3\}$
some lengths $\lambda_{03}$, $\lambda_{13}$ and $\lambda_{23}$, \resp.
We now consider all possible ways to assign coordinates to each of the vertices \st
vertex~1 is sent to $(0,0)$, vertex~2 is sent to $(1,0)$
and the distance between vertex $i$ and vertex $3$ is equal to $\lambda_{i3}$
for all $i\in\{0,1,2\}$:
\[
(x_i-x_3)^2+(y_i-y_3)^2=\lambda_{i3}^2.
\]
The ``coupler curve'' of $\cC_3$ is defined as the
set of all possible coordinates for vertex~0 and thus
consists of two circles. The center of the second circle is the reflection of
vertex~3 along the line spanned by the edge~$\{1,2\}$.
The marked graphs~$\cC_3$ and $\cL$ in \FIG{intro} are examples of ``calligraphs''.
Informally, when fixing vertices 1 and 2 of a calligraph, we require that the vertex~0
draws a curve.

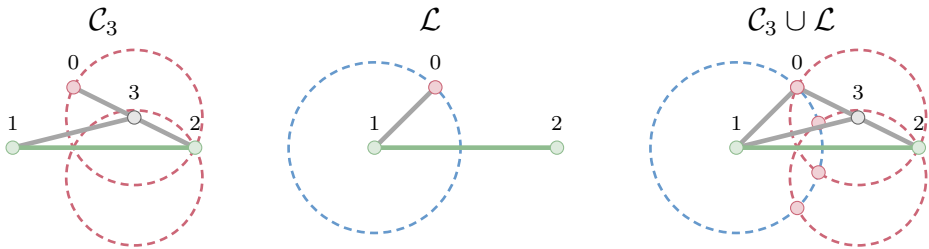
\begin{figure}[!h]
\centering
\setlength{\tabcolsep}{5mm}
\begin{tabular}{ccc}
$\cC_3$ & $\cL$ & $\cC_3\cup\cL$
\\
\begin{tikzpicture}[scale=0.8]
\draw[traj] \pe circle [radius=1.12];
\draw[traj] \pE circle [radius=1.12];
\Iedge \pl \pe;
\Iedge \pr \pe;
\Iedge \pe \pf;
\Aedge \pl \pr;
\AvertN \pl 1;
\AvertN \pr 2;
\IvertN \pe 3;
\MvertN \pf 0;
\markerN{(0,-1.9)}{}
\end{tikzpicture}
&
\begin{tikzpicture}[scale=0.8]
\draw[traj2] \pl circle [radius=1.414];
\Iedge \pl \pf;
\Aedge \pl \pr;
\AvertN \pl 1;
\AvertN \pr 2;
\MvertN \pf 0;
\markerN{(0,-1.9)}{}
\end{tikzpicture}
&
\begin{tikzpicture}[scale=0.8]
\draw[traj] \pe circle [radius=1.12];
\draw[traj] \pE circle [radius=1.12];
\draw[traj2] \pl circle [radius=1.414];
\MvertN{(1,1)}{};
\MvertN{(1,-1)}{};
\MvertN{(1.35,0.41)}{};
\MvertN{(1.35,-0.41)}{};

\Iedge \pl \pe;
\Iedge \pr \pe;
\Iedge \pe \pf;
\Iedge \pl \pf;
\Aedge \pl \pr;
\AvertN \pl 1;
\AvertN \pr 2;
\IvertN \pe 3;
\MvertN \pf 0;

\markerN{(0,-1.9)}{}
\end{tikzpicture}
\end{tabular}
\vspace{-0.7cm}
\caption{The coupler curves of the two calligraphs intersect in four different points.}
\label{fig:intro}
\end{figure}

The union $\cC_3\cup \cL$ of the two calligraphs in \FIG{intro} can
be realized in the plane in four different ways up to translations and rotations.
We see in \FIG{min} how the four realizations of $\cC_3\cup \cL$
are related to the number of intersections between the coupler curves of
$\cC_3$ and~$\cL$.

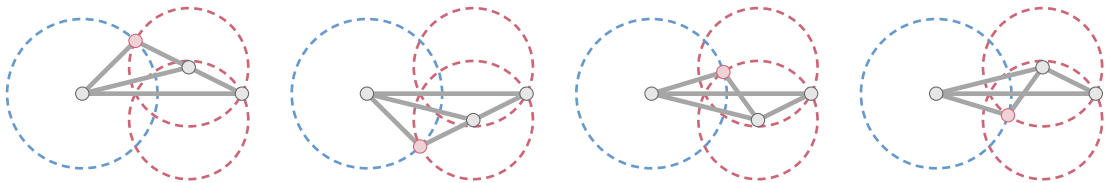
\begin{figure}[!ht]
\centering
\setlength{\tabcolsep}{2.5mm}
\begin{tabular}{cccc}
\begin{tikzpicture}[scale=0.7]
\draw[traj] \pe circle [radius=1.12];
\draw[traj] \pE circle [radius=1.12];
\draw[traj2] \pl circle [radius=1.414];
\Iedge{(0,0)}{(2,0.5)};
\Iedge{(3,0)}{(2,0.5)};
\Iedge{(2,0.5)}{(1,1)};
\Iedge{(0,0)}{(1,1)};
\Iedge{(0,0)}{(3,0)};
\IvertN{(0,0)}{};
\IvertN{(3,0)}{};
\IvertN{(2,0.5)}{};
\MvertN{(1,1)}{};
\markerN{(0,-1.9)}{}
\end{tikzpicture}
&
\begin{tikzpicture}[scale=0.7]
\draw[traj] \pe circle [radius=1.12];
\draw[traj] \pE circle [radius=1.12];
\draw[traj2] \pl circle [radius=1.414];
\Iedge{(0,0)}{(2,-0.5)};
\Iedge{(3,0)}{(2,-0.5)};
\Iedge{(2,-0.5)}{(1,-1)};
\Iedge{(0,0)}{(1,-1)};
\Iedge{(0,0)}{(3,0)};
\IvertN{(0,0)}{};
\IvertN{(3,0)}{};
\IvertN{(2,-0.5)}{};
\MvertN{(1,-1)}{};
\markerN{(0,-1.9)}{}
\end{tikzpicture}
&
\begin{tikzpicture}[scale=0.7]
\draw[traj] \pe circle [radius=1.12];
\draw[traj] \pE circle [radius=1.12];
\draw[traj2] \pl circle [radius=1.414];
\Iedge{(0,0)}{(2,-0.5)};
\Iedge{(3,0)}{(2,-0.5)};
\Iedge{(2,-0.5)}{(1.35,0.41)};
\Iedge{(0,0)}{(1.35,0.41)};
\Iedge{(0,0)}{(3,0)};
\IvertN{(0,0)}{};
\IvertN{(3,0)}{};
\IvertN{(2,-0.5)}{};
\MvertN{(1.35,0.41)}{};
\markerN{(0,-1.9)}{}
\end{tikzpicture}
&
\begin{tikzpicture}[scale=0.7]
\draw[traj] \pe circle [radius=1.12];
\draw[traj] \pE circle [radius=1.12];
\draw[traj2] \pl circle [radius=1.414];
\Iedge{(0,0)}{(2,0.5)};
\Iedge{(3,0)}{(2,0.5)};
\Iedge{(2,0.5)}{(1.35,-0.41)};
\Iedge{(0,0)}{(1.35,-0.41)};
\Iedge{(0,0)}{(3,0)};
\IvertN{(0,0)}{};
\IvertN{(3,0)}{};
\IvertN{(2,0.5)}{};
\MvertN{(1.35,-0.41)}{};
\markerN{(0,-1.9)}{}
\end{tikzpicture}
\end{tabular}
\vspace{-0.8cm}
\caption{The intersections of coupler curves of two calligraphs
are related to the number of realizations of a minimally rigid graph
into the plane.}
\label{fig:min}
\end{figure}

The graph~$\cC_3\cup \cL$ is called ``minimally rigid graph''
and its ``number of realizations''~$\cc(\cC_3\cup \cL)$ is equal to four
for almost all choices of edge length assignments.
This number can be expressed as the number of complex solutions of quadratic equations.
In \FIG{min} all solutions are real, but in general there may be non-real solutions.

\textbf{Main result.} In this article we assign to a calligraph~$\cG$ its ``class''~$[\cG]$,
namely a triple~$(a,b,c)$ of integers that satisfies two axioms.
Axiom~\AXM{1} states that
$[\cL]=(1,1,0)$, $[\cR]=(1,0,1)$ and $[\cC_3]=(2,0,0)$,
where $\cR$ is defined in \FIG{cal}.
Axiom~\AXM{2} essentially states that if
$\cG\cup\cG'$ is a minimally rigid graph for some calligraphs $\cG$ and~$\cG'$,
then its number of realizations equals
$\cc(\cG\cup\cG')=[\cG]\cdot[\cG']=2(a\cdot a'-b\cdot b'-c\cdot c')$ (see \DEF{class}).
For example, we verify that~$\cc(\cC_3\cup\cL)=[\cC_3]\cdot[\cL]=2(2\cdot 1-0\cdot 1-0\cdot 0)=4$.
Our main result is \THM{class}, which asserts that the class of a calligraph exists and is unique.
By its \COR{class} the class satisfies six properties \PPP{1}---\PPP{6} that characterize
invariants of coupler curves.
For example, the coupler curve of $\cC_3$ is of degree~$4$ by \PPP{2} and \PPP{5}.
Property~\PPP{6} is more technical and bounds the geometric genera and degrees of each
of the irreducible components of a coupler curve.
As an application of our methods we obtain an improved algorithm
for computing the number of realizations of minimally rigid graphs
that uses the algorithm in~\cite{rigid-alg} as a fallback algorithm (see \RMK{timing}).

\textbf{State of the art: coupler curves.}
A calligraph is a special case of a linkage with revolute joints
and thus its origins
can be traced back to at least 1785, when James Watt
used a calligraph to convert a rotational motion to an approximate straight-line motion
for his steam engine.
Kempe described in 1876
a method that constructs for any given planar algebraic curve
a linkage that traces out a portion of this curve \cite{kempe}.
Coupler curves of linkages have been studied extensively by engineers
and we refer to \citep[Sections~7 and 15.3.7]{hunt} for more details
and further references.
Recently, a method was provided for constructing a linkage,
with a small number of links and joints,
that has a given rational curve as trajectory~\cite{motion-planar,motion-kempe}.
There exist minimally rigid graphs \st
for a carefully chosen edge length assignment, a vertex in this graph
traces out a trajectory, while fixing some edge into place.
We refer to \cite{yet} for an overview of the classification of such
paradoxically moving graphs.

\textbf{State of the art: number of realizations.}
The investigation of rigid structures can be traced back to
James Clerk Maxwell~\cite{maxwell} and there is currently again a considerable interest in rigidity
theory due to various applications in natural science and engineering \cite{rigid-thorpe}.
Minimally rigid graphs (also known as Laman graphs) have been classified in \cite{rigid-gei}
by Pollaczek-Geiringer and independently in \cite{rigid-laman}.
Bounds on the number of realizations of such graphs have been investigated in
\cite{rigid-stef,Emiris09,rigid-num-borcea,Bartzos_2020,bartzos2021new,rigid-low-bound,rigid-num-jan}.
Algorithms and theory for computing
the precise number of realizations for minimally rigid graphs
have been investigated in \cite{rigid-alg,rigid-jack-2018,RealizationsIsostatic}.

\textbf{Theoretical context.}
The number of realizations of the union of two calligraphs is
equal to the number of complex intersections of their coupler curves
for almost all choices of edge length assignments.
In \FIG{min}, for example, all these intersections are real.
It follows from B\'ezout's theorem that this number
is equal to the product of the degrees of the coupler curves
minus the complex intersections at infinity counted with multiplicities.
If we fix a calligraph and vary its edge length assignments,
then we obtain an ``algebraic series'' of coupler curves in the projective plane.
If all curves in this series meet
the same complex points, then these points are called ``base points''.
We show that the complex intersections at infinity of coupler curves are exactly at such base points.
In this light, our main result is that a base point associated to a calligraph
coincides with a base point associated to either the calligraph $\cC_3$, $\cL$ or~$\cR$.
The degree of a coupler curve of a calligraph and its multiplicities at the base points
remains constant for almost all choices of edge length assignments and is encoded by
the class of a calligraph.
The number of complex intersections of the coupler curves,
minus the number of intersections at infinity, is equal to a certain inner product between the classes
(see Axiom~\AXM{2}).

To prepare the reader, let us consider the base points of calligraphs in a bit more detail.
Almost each curve in the algebraic series of~$\cC_3$
is a union of two circles that meet complex conjugate ``cyclic'' points
at infinity with multiplicity~2.
The algebraic series of~$\cL$ contains all circles that are centered around vertex~1.
These circles meet aside the cyclic base points, also the ``1-centric'' base points.
Similarly, the algebraic series of $\cR$
contains all circles that are centered at vertex~2 and meet aside the
cyclic base points the ``2-centric'' base points.
The 1-centric and 2-centric points are ``infinitely near'' to the cyclic base points.
This terminology is made precise in~\SEC{bp}.
The multiplicities of general curves in the algebraic series of a calligraph at the cyclic, 1-centric and 2-centric
points correspond to the three numbers in the class of this calligraph.
In particular, $[\cC_3]=(2,0,0)$, $[\cL]=(1,1,0)$ and $[\cR]=(1,0,1)$ (see Axiom~\AXM{1}).
We encode the algebraic series of calligraphs as so called divisor classes,
and obtain the two axioms and six properties for classes using standard algebro geometric methods.
The multiplicity at the cyclic points (classically known as \df{cyclicity}) was used in \cite{wunderlich} to
recover the degree of coupler curves for a certain type of linkages.
To our best knowledge the 1-centric and 2-centric base points have not been considered before
and allow us to characterize the coupler curves of any calligraph.

We want to emphasize that in this article we count the number of realizations
over the complex numbers, and thus we obtain in general only an upper bound for the number of \emph{real} realizations
(see \citep[Section~8, page~5]{rigid-jack-2018}).

\textbf{Acknowledgements.}
We thank Josef Schicho for the inspiring and fruitful discussions.
We thank Matteo Gallet for carefully reading the manuscript and giving useful
comments concerning its presentation.
We thank the anonymous reviewers for valuable feedback.
Niels Lubbes was supported by the Austrian Science Fund (FWF): P33003.
Boulos El Hilany was in his first year supported by P33003 as well,
and for the remaining time by the grant DFG EL1092/1-1.
Georg Grasegger was supported by the Austrian Science Fund (FWF): P31888.

\textbf{Overview.}
In \SEC{main} we start with introducing terminology and conclude with
a precise statement of the main result, namely \THM{class}, together with some
conjectures and open problems.

We clarify in \SEC{mainapp} how the main result can be applied to
determine the number of realizations of minimally rigid graphs
and invariants of coupler curves such as degree and geometric genus.
In particular, we discuss the correctness and heuristics of
an algorithm for computing the number of realizations.
In \APP{tree} we give a detailed example illustrating the
recursive structure of this algorithm.

In \SEC{bp} we start by recalling the notion of infinitely near base points
and with a proof for \PRP{pseudo}, which states that
\THM{class} and \COR{class} hold under the assumption
that calligraphs are ``centric'', namely that
the base points of the algebraic series of a calligraph are either cyclic, 1-centric or 2-centric.

In \SEC{op} we show that the procedure for detecting infinitely near base points
can be done directly on a system of quadric equations defined by a calligraph.
This translates into sufficient conditions for the centricity of calligraphs
in terms of symbolic modifications of these equations.

In \SEC{proof} we show that all calligraphs are centric
by using the sufficient conditions from~\SEC{op}
and thereby conclude the proof of \THM{class} and \COR{class}.
We also need \PRP{vert}, but we moved its proof to \APP{lines}
as this part is independent and may distract from the
logical path.

We have made an effort to make this article accessible to a wide audience.
The main results and algorithmic applications do not require many preliminaries,
however, for the proofs we assume familiarity with algebraic geometry.
In order to prepare the reader we list in \APP{nota} the
notation that is used across different sections.

\section{Statements of main result and open problems}
\label{sec:main}

In \SEC{class} we state the definition of the class of a calligraph
and \THM{class}.
In \SEC{coupler} we introduce definitions related to coupler curves
in order to state \COR{class}.
We conclude in \SEC{problems} with some conjectures and open problems.

\subsection{Classes of calligraphs}
\label{sec:class}

Let $\cG$ denote a graph with vertices~$\vv(\cG)\subset \Z_{\geq0}$ and edges~$\ee(\cG)$.
In what follows, all graphs are undirected and simple.
We call a graph $\cG$ \df{marked} if $\{1,2\}\in\ee(\cG)$.
The graphs $\cL$, $\cR$ and $\cC_v$ for $v\in\Z\setminus\{0,1,2\}$ are defined as in \FIG{cal}.
\begin{figure}[!ht]
\centering
\setlength{\tabcolsep}{5mm}
\begin{tabular}{ccc}
\begin{tikzpicture}
\Iedge \pl \PM;
\Aedge \pl \pr;
\AvertN \pl 1;
\AvertN \pr 2;
\MvertN \PM 0;
\end{tikzpicture}
&
\begin{tikzpicture}
\Iedge \pr \PM;
\Aedge \pl \pr;
\AvertN \pl 1;
\AvertN \pr 2;
\MvertN \PM 0;
\end{tikzpicture}
&
\begin{tikzpicture}
\Iedge \pl \pe;
\Iedge \pr \pe;
\Iedge \pe \pf;
\Aedge \pl \pr;
\AvertN \pl 1;
\AvertN \pr 2;
\IvertN \pe v;
\MvertN \pf 0;
\end{tikzpicture}
\\
$\cL$ & $\cR$ & $\cC_v$
\end{tabular}
\caption{Three calligraphs that play a central role.}
\label{fig:cal}
\end{figure}
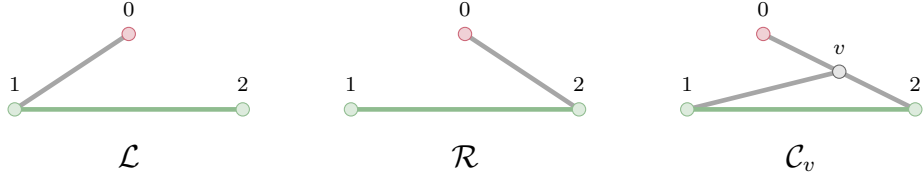

We call $\cG$ a \df{minimally rigid graph} if $|\ee(\cG)|=2|\vv(\cG)|-3$ and
$|\ee(\cH)|\leq 2|\vv(\cH)|-3$ for all subgraphs $\cH\subset \cG$
\st $|\vv(\cH)|>1$.
For example, $\cC_3\cup \cL$ in \FIG{intro} and the triangle~$\cL\cup\cR$ are both minimally rigid graphs.

We call $\cG$ a \df{calligraph} if
$\vv(\cG)\cap \vv(\cC_v)=\{0,1,2\}$, $\ee(\cG)\cap\ee(\cC_v)=\{\{1,2\}\}$
and $\cG\cup \cC_v$ is a minimally rigid graph for all $v\notin\vv(\cG)$.
For example, $\cL$, $\cR$ and $\cC_3$ are all calligraphs.
By counting the vertices and edges we find that
a calligraph is a minimally rigid graph minus one edge.

Suppose that $\text{Prop}\c X\to \{\text{True},~\text{False}\}$ is a proposition about some algebraic set~$X$
\st $\set{x\in X}{\text{Prop}(x)=\text{False}}$ is contained in an algebraic set~$Y$.
If each irreducible component of $Y$ forms a lower dimensional subset of some irreducible component of~$X$,
then we call any element in $X\setminus Y$ \df{general}
and we say that $\text{Prop}(x)=$True \df{for almost all}~$x\in X$.
Informally, we may think of a general element in~$X$ as a random element.
Notice that if $X\subseteq\C^n$, then there exists a
non-zero $n$-variate complex polynomial~$f$ \st $f(y)=0$ for all~$y\in Y$.
See \RMK{generic} for the relation to the notion of ``generic''.

If $\cG$ is a marked graph, then
the \df{set of edge length assignments}~$\Omega_\cG$ is defined as the set of
maps~$\omega\c\ee(\cG)\to \C$ \st $\omega(\{1,2\})=1$.
The \df{set of realizations}~$\Xi_\cG^\omega$, that are compatible
with the edge length assignment~$\omega\in\Omega_\cG$, is defined as the set of maps~$\xi\c\vv(\cG)\to \C^2$ \st
$\xi(1)=(0,0)$, $\xi(2)=(1,0)$ and
\[
(x_i-x_j)^2+(y_i-y_j)^2=\omega(\{i,j\})^2
\quad\text{for all}\quad\{i,j\}\in\ee(\cG),
\]
where $\xi(i)=(x_i,y_i)$ for all $i\in\vv(\cG)$.

It follows from \citep[Theorem~3.6]{rigid-jack-2018} or \citep[Corollary~1.11]{rigid-alg} that
$|\Xi_\cG^\omega|=|\Xi_\cG^{\omega'}|$ for almost all edge length assignments~$\omega,\omega'\in \Omega_\cG$.
If $\omega\in \Omega_\cG$ is general,
then we define the \df{number of realizations} of $\cG$ as
\[
\cc(\cG):=|\Xi_\cG^\omega|\in \Z_{\geq 0}\cup\{\infty\}.
\]
Notice that the definition ``general'' means in this case that the set
$\set{\omega\in \Omega_\cG}{|\Xi_\cG^\omega|\neq\cc(\cG)}$ is contained in a lower dimensional set
of the algebraic set~$\Omega_\cG$.
The number of realizations does not depend on the choice of the marked edge~$\{1,2\}$
and thus is well-defined for graphs that are not marked.
It follows from \cite{rigid-gei} or \cite{rigid-laman}
that $\cG$ is a minimally rigid graph if and only if $\cc(\cG)\in\Z_{>0}$.
For example, if $\cG=\cC_3\cup\cL$, then each realization in~$\Xi_\cG^\omega$
is illustrated in \FIG{min} for some general edge length assignment~$\omega\in \Omega_\cG$.

We call $(\cG,\cG')$ a \df{calligraphic split} for the graph~$\cG\cup \cG'$,
if $\cG$ and $\cG'$ are calligraphs \st
$\ee(\cG)\cap\ee(\cG')=\{\{1,2\}\}$ and  $\vv(\cG)\cap\vv(\cG')=\{0,1,2\}$.
If in addition
$\vv(\cG\cup\cG')-2\geq |\vv(\cG)|\geq  |\vv(\cG')|$,
then $(\cG,\cG')$ is called \df{non-trivial}.
For example, $(\cC_3,\cL)$ as depicted in \FIG{intro} is a calligraphic split of $\cC_3\cup\cL$.
Notice that all minimally rigid graphs admit a calligraphic split of the form
$(\cG,\cL)$ or $(\cG,\cR)$ for some calligraph~$\cG$.

We define the bilinear map $\cdot\c\Z^3\times\Z^3 \to\Z$ as
\[
(a,b,c)\cdot (a',b',c')=2\,(a\,a'-b\,b'-c\,c').
\]

\begin{definition}
\label{def:class}
A \df{class} for calligraphs is a function that assigns to each calligraph~$\cG$
an element $[\cG]$ of~$\Z^3$ \st the following two axioms are fulfilled.
\begin{enumerate}[topsep=0mm,itemsep=0mm,label=A\arabic*.,ref=A\arabic*]
\item\label{axm:1}
$[\cL]=(1,1,0)$, $[\cR]=(1,0,1)$ and $[\cC_v]=(2,0,0)$ for all $v\in \Z_{\geq3}$.
\item\label{axm:2}
If $(\cG,\cG')$ is a calligraphic split, then $\cc(\cG\cup\cG')=[\cG]\cdot[\cG']$. \END
\end{enumerate}
\end{definition}

\begin{maintheorem}
\label{thm:class}
The class for calligraphs exists and is unique.
\end{maintheorem}

The class~$[\cG]$ of a calligraph~$\cG$ is uniquely
determined by the numbers $\cc(\cG\cup\cL)$,
$\cc(\cG\cup\cR)$ and $\cc(\cG\cup\cC_v)$ for some $v\notin\vv(\cG)$ (see forward \ALG{class}).
We see in \COR{class} that the class
reveals key properties about coupler curves.
However, for this we need to introduce in \SEC{coupler} some additional concepts.

\subsection{Invariants of coupler curves}
\label{sec:coupler}

In this article we assume that curves are real algebraic, namely the complex solution sets of polynomial equations
with real coefficients.
If $C\subset \C^2$ is a curve, then we denote by $\deg C$ its \df{degree}.
If $C$ is also irreducible,
then $g(C)$ denotes its \df{geometric genus} (see for example \cite{miranda}).
The \df{singular locus} of~$C$ is denoted by~$\sng C$.
We call $C$ a \df{circle} if $\deg C=2$ and $C\cap \R^2$ is a circle.

The \df{coupler curve} of a calligraph~$\cG$
\wrt edge length assignment $\omega\in\Omega_\cG$ is defined as
follows (``t'' stands for ``trajectory''):
\[
\tt_\omega(\cG):=\set{\xi(0)}{\xi\in \Xi_\cG^\omega}.
\]
We call $(T_1,\ldots,T_n)$ its \df{coupler decomposition}
if $\tt_\omega(\cG)=T_1\cup\cdots \cup T_n$
and $T_i$ is an irreducible curve for all $1\leq i\leq n$.
For example, in \FIG{intro} we see that the coupler curve of $\cC_3$ for some choice of
edge length assignment consist of two circles $T_1$ and~$T_2$.
The \df{coupler multiplicity} of~$\cG$ is defined as
\[
\mm(\cG):=|\set{\xi\in \Xi_\cG^\omega}{\xi(0)=p}|,
\]
where both $\omega\in\Omega_\cG$ and $p\in \tt_\omega(\cG)$ are general.
In other words, $\mm(\cG)$ equals the number of realizations sending
vertices $1$, $2$ and $0$ to $(0,0)$, $(1,0)$ and $p$, \resp.
General coupler curves are indeed curves by \LEM{dim}\ref{lem:dim:d} and
the coupler multiplicity is well-defined by \LEM{mult}.

In \FIG{tm} we consider the real points of the coupler curves and the coupler multiplicities of three
calligraphs $\cL$, $\cM$ and $\cQ$,
for some general choice of edge length assignments
$\widetilde{\omega}\in \Omega_\cL$, $\omega\in \Omega_\cM$ and $\hat{\omega}\in \Omega_\cQ$.
Each coupler curve is a circle centered at~$(0,0)$
and the points $p\in \tt_\omega(\cM)$ and $q\in \tt_{\hat{\omega}}(\cQ)$ are assumed to be general.
We see that $\mm(\cM)=2$, since
there exist two realizations $\xi,\xi'\in\Xi_\cM^{\omega}$ \st $\xi(0)=\xi'(0)=p$,
where $\xi'(3)$ is obtained by flipping $\xi(3)$ along the edge~$\{1,2\}$.
Similarly, $\mm(\cQ)=2$,
as there exist two real realizations $\xi,\xi'\in\Xi_\cQ^{\hat{\omega}}$ \st $\xi(0)=\xi'(0)=q$, where
$\xi'(3)$ is obtained by reflecting~$\xi(3)$ along the line spanned by~$\{0,2\}$.
The thick arcs on the coupler curve of $\cQ$ correspond to the
subset $\set{\xi(0)}{\xi\in \Xi_\cQ^{\hat{\omega}} \text{ and } \xi(0),\xi(3)\in\R^2}$,
namely the real trace of vertex~$0$ when we consider the graph as a moving mechanism.

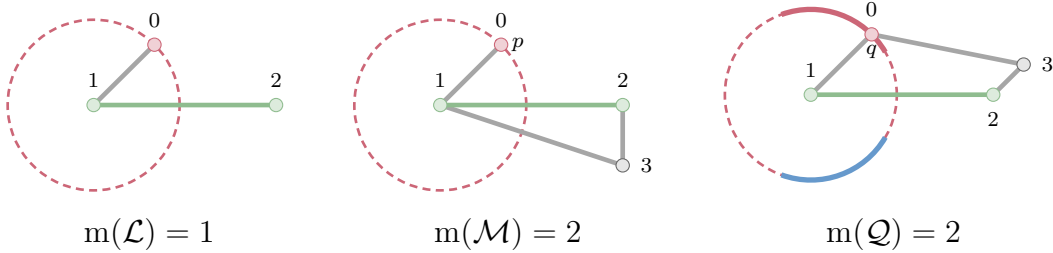
\begin{figure}[!ht]
\centering
\setlength{\tabcolsep}{4mm}
\begin{tabular}{ccc}
\begin{tikzpicture}[scale=0.8]
\draw[traj] \pl circle [radius=1.414];
\Iedge \pl \pf;
\Aedge \pl \pr;
\AvertN \pl 1;
\AvertN \pr 2;
\MvertN \pf 0;
\end{tikzpicture}
&
\begin{tikzpicture}[scale=0.8]
\draw[traj] \pl circle [radius=1.414];
\Iedge \pl \pf;
\Iedge \pl \pv;
\Iedge \pr \pv;
\Aedge \pl \pr;
\AvertN \pl 1;
\AvertN \pr 2;
\MvertN \pf 0;
\IvertE \pv 3;
\node[labelsty,right] at \pf {$p$};
\end{tikzpicture}
&
\begin{tikzpicture}[scale=0.8]
\draw[traj] \pl circle [radius=1.414];
\draw[colvm,line width=2pt] ($({1.414*cos(30)},{1.414*sin(30)})$) arc (30:110:1.414);
\draw[colvm2,line width=2pt] ($({1.414*cos(250)},{1.414*sin(250)})$) arc (250:330:1.414);
\Iedge \pl \pf;
\Iedge \pr \pz;
\Iedge \pf \pz;
\Aedge \pl \pr;
\AvertN \pl 1;
\AvertS \pr 2;
\MvertN \pf 0;
\IvertE \pz 3;
\node[labelsty,below] at \pf {$q$};
\end{tikzpicture}
\\
$\mm(\cL)=1$ & $\mm(\cM)=2$ & $\mm(\cQ)=2$
\end{tabular}
\caption{Three calligraphs $\cL$, $\cM$ and $\cQ$ together with their coupler curves and coupler multiplicities.
For each $r\in \tt_{\hat{\omega}}(\cQ)$ that lies on one of the two thick arcs
there exists $\xi\in\Xi_\cQ^{\hat{\omega}}$ \st $\xi(0)=r$ and $\xi(3)$ is real .
}
\label{fig:tm}
\end{figure}

A connected graph is called
\df{$k$-vertex connected} if it has more than $k$ vertices and remains connected
whenever fewer than $k$ vertices are removed.
We call a calligraph~$\cG$ \df{thin}, if $\cG':=\cG\cup\cL\cup\cR$ is 3-vertex connected
and $\cG'$ minus the edge~$\{v_1,v_2\}$ is minimally rigid for all edges~$\{v_1,v_2\}\in\ee(\cG')$.
For example, the calligraph $\cH$ in \FIG{HH} is thin.
We remark that if $\cG$ is thin, then $\cG\cup\cL\cup\cR$ is ``generic globally rigid''
(see \citep[Corollary~1.7]{rigid-con} or \citep[Theorem~5.1]{rigid-jack-2018})
and we shall see in \COR{class} that this property characterizes coupler multiplicity of~$\cG$.

We call $(\alpha_1,\ldots,\alpha_n)$ a \df{class partition} for a calligraph~$\cG$ if
$\alpha_i=(\alpha_{i0},\alpha_{i1},\alpha_{i2})\in \Z^3$ \st
$\alpha_{i0}\geq \alpha_{i1}\geq0$, $\alpha_{i0}\geq \alpha_{i2}\geq 0$
and
$\alpha_1+\cdots+\alpha_n=\tfrac{1}{m}\cdot[\cG]$ for all $1\leq i\leq n$,
where $m$ denotes the coupler multiplicity of $\cG$.

\begin{maincorollary}
\label{cor:class}
The following six properties are satisfied
for all calligraphs $\cG$ and $\cG'$
and almost all edge length assignments~$\omega\in\Omega_\cG$ and $\omega'\in\Omega_{\cG'}$,
where we use the following notation:
\[
T:=\tt_\omega(\cG),\qquad
m:=\mm(\cG),\qquad
T':=\tt_{\omega'}(\cG'),\qquad
m':=\mm(\cG').
\]
\begin{enumerate}[topsep=0mm,itemsep=0mm,label=P\arabic*.,ref=P\arabic*]
\item\label{ppp:1}
If $[\cG]=(a_0,a_1,a_2)$, then $a_0\geq a_1\geq 0$ and $a_0\geq a_2\geq0$.
\item\label{ppp:2}
If $\cG$ is thin, then $\mm(\cG)=1$.
\item\label{ppp:3}
If $\cG\cup\cL$ is not minimally rigid, then $[\cG]=(m,m,0)$.\\
If $\cG\cup\cR$ is not minimally rigid, then $[\cG]=(m,0,m)$.
\item\label{ppp:4}
If $(\cG,\cG')$ is a calligraphic split, then
$m\cdot m'\cdot |T\cap T'|=[\cG]\cdot [\cG']$.
\item\label{ppp:5}
If $[\cG]=(a_0,a_1,a_2)$, then $\deg T=2a_0/m$.
\item\label{ppp:6}
If $(T_1,\ldots,T_n)$ is a coupler decomposition for $T$,
then there exists a class partition~$(\alpha_1,\ldots,\alpha_n)$ for $\cG$ \st for all $1\leq i\leq n$ we have
\\
$\deg T_i=\alpha_i\cdot(1,0,0)$ and
$g(T_i) \leq \tfrac{1}{2}\cdot\alpha_i\cdot\bigl(\alpha_i-(2,1,1)\bigr)+1-|\sng T_i|$.
\end{enumerate}
\end{maincorollary}

Before we continue let us mention some conjectures and open problems for the reader
to keep in mind as we clarify \THM{class}, \COR{class} and its applications in \SEC{mainapp}.
The proof of our main results are prepared in \SEC{bp}, \SEC{op} and \APP{lines}, and concluded in~\SEC{proof}.

\subsection{Conjectures and open problems}
\label{sec:problems}

The following conjecture states that for almost all edge length assignments
each component of a coupler curve has the same degree and geometric genus.

\begin{conjecture}
\label{con:class}
In \PPP{6} we additionally have for all $1\leq i,j\leq n$:
\[
\deg T_i=\deg T_j
\quad\text{and}\quad
g(T_i)=g(T_j)=\tfrac{1}{2}\cdot\alpha_i\cdot\bigl(\alpha_i-(2,1,1)\bigr)+1-|\sng T_i|.
\]
\end{conjecture}

Note that the equality for the geometric genus follows if
the delta invariant of any isolated singularity of
the coupler curve equals one (see the proof of \PRP{pseudo}).
Even if this conjecture is confirmed, we still do not know how to determine
efficiently the number $|\sng T_1|$ of singularities of an irreducible component~$T_1$
of the coupler curve, and the number $n$ of irreducible components.

\begin{openproblem}
For any given calligraph, determine the number of irreducible components of the coupler curve
and for each such component its degree and geometric genus.
\end{openproblem}

On the other hand,
we refer to \citep[Section~8]{rigid-jack-2018}
for open problems concerning the number of realizations of minimally rigid graphs.

\section{Applications of the main result}
\label{sec:mainapp}

In this section we shall consider two applications of our main result \THM{class}
and its \COR{class}.
\begin{Mlist}
\item
If $(\cG,\cG')$ is a calligraphic split, then
we reduce the computation of~$c(\cG\cup \cG')$
to the computation of the numbers of realizations of six minimally rigid graphs
with at most $\max(|\vv(\cG)|,|\vv(\cG')|)+1$ vertices.
If $(\cG,\cG')$ is non-trivial, then
we confirm experimentally that this reduction
leads to a significant speed up.

\item If $\cG$ is a calligraph with coupler curve~$T$ and coupler decomposition~$(T_1,\ldots, T_n)$,
then we determine $\deg T$ and give non-trivial bounds for $n$, $\deg T_i$ and $g(T_i)$
for all $1\leq i\leq n$.
Furthermore, we determine the number of complex intersections of coupler curves of two calligraphs.
\end{Mlist}
The first application is clarified in \SEC{H} and \APP{tree}.
The second application is clarified in \SEC{couplerH} and \SEC{large}.

\subsection{Counting the number of realizations using classes}
\label{sec:H}

Suppose we would like to determine the number of realizations~$c(\cU)$ of the minimally rigid graph~$\cU$
in~\FIG{UVW}.
We remove the vertex of degree 2 from $\cU$ and obtain $\cV$.
It is straightforward to see that $c(\cU)=2\cdot c(\cV)$.

\begin{figure}[!ht]
\centering
\setlength{\tabcolsep}{5mm}
\begin{tabular}{ccc}
\begin{tikzpicture}[scale=0.8]
\Iedge \pl \pg;
\Iedge \pg \ph;
\Iedge \pg \pf;
\Iedge \ph \pf;
\Iedge \ph \pf;
\Iedge \ph \pr;
\Iedge \pl \pu;
\Iedge \pr \pv;
\Iedge \pu \pv;
\Iedge \pu \pf;
\Iedge \pv \pf;
\Iedge \pl \pA;
\Iedge \pA \pg;
\Iedge \pl \pr;
\IvertW \pl ~;
\IvertE \pr ~;
\IvertN \pg ~;
\IvertN \ph ~;
\IvertS \pf ~;
\IvertW \pu ~;
\IvertE \pv ~;
\IvertN \pA ~;
\end{tikzpicture}
&
\begin{tikzpicture}[scale=0.8]
\Iedge \pl \pg;
\Iedge \pg \ph;
\Iedge \pg \pf;
\Iedge \ph \pf;
\Iedge \ph \pf;
\Iedge \ph \pr;
\Iedge \pl \pu;
\Iedge \pr \pv;
\Iedge \pu \pv;
\Iedge \pu \pf;
\Iedge \pv \pf;
\Iedge \pl \pr;
\IvertW \pl ~;
\IvertE \pr ~;
\IvertN \pg ~;
\IvertN \ph ~;
\IvertS \pf ~;
\IvertW \pu ~;
\IvertE \pv ~;
\end{tikzpicture}
&
\begin{tikzpicture}[scale=0.8]
\Iedge \pl \pg;
\Iedge \pg \ph;
\Iedge \pg \pf;
\Iedge \ph \pf;
\Iedge \ph \pf;
\Iedge \ph \pr;
\Iedge \pl \pu;
\Iedge \pr \pv;
\Iedge \pu \pv;
\Iedge \pu \pf;
\Iedge \pv \pf;
\Aedge \pl \pr;
\AvertW \pl 1;
\AvertE \pr 2;
\IvertN \pg 3;
\IvertN \ph 4;
\MvertS \pf 0;
\IvertW \pu 5;
\IvertE \pv 6;
\end{tikzpicture}
\\
$\cU$ & $\cV$ &
\end{tabular}
\caption{$c(\cU)=2\cdot c(\cV)$ and thus it is sufficient to compute $c(\cV)$.}
\label{fig:UVW}
\end{figure}

Since $\cV$ does not contain any degree 2 vertices,
we are now going to apply Axioms~\AXM{1} and~\AXM{2}.
For this we need to find a non-trivial calligraphic split for $\cV$.
First we choose an edge and a vertex in $\cV$.
After relabeling the vertices
we obtain the rightmost marked graph in \FIG{UVW} \st $\{1,2\}$ and $0$
are the chosen edge and vertex, \resp.
We consider the subgraph of~$\cV$ induced by $\vv(\cV)\setminus\{0,1,2\}$
as is illustrated in \FIG{HH}.
This subgraph consists of two components defined by the edges~$\{3,4\}$ and $\{5,6\}$.
These components define the calligraphs $\cH$ and~$\cI$ which are induced by the vertices $\{3,4\}\cup\{0,1,2\}$ and $\{5,6\}\cup\{0,1,2\}$, \resp.
We verify that $(\cH,\cI)$ is a non-trivial calligraphic split for~$\cV$.

\begin{figure}[!ht]
\centering
\setlength{\tabcolsep}{3mm}
\begin{tabular}{ccc}
\begin{tikzpicture}[scale=0.8]
\DIedge \pl \pg;
\Iedge \pg \ph;
\DIedge \pg \pf;
\DIedge \ph \pf;
\DIedge \ph \pf;
\DIedge \ph \pr;
\DIedge \pl \pu;
\DIedge \pr \pv;
\Iedge \pu \pv;
\DIedge \pu \pf;
\DIedge \pv \pf;
\DAedge \pl \pr;
\DAvert \pl;
\DAvert \pr;
\DMvert \pf;
\IvertN \pg 3;
\IvertN \ph 4;
\IvertW \pu 5;
\IvertE \pv 6;
\end{tikzpicture}
&
\begin{tikzpicture}[scale=0.8]
\Iedge \pl \pg;
\Iedge \pg \ph;
\Iedge \pg \pf;
\Iedge \ph \pf;
\Iedge \ph \pf;
\Iedge \ph \pr;
\Aedge \pl \pr;
\AvertW \pl 1;
\AvertE \pr 2;
\IvertN \pg 3;
\IvertN \ph 4;
\MvertS \pf 0;
\markerE \pv 6;
\end{tikzpicture}
&
\begin{tikzpicture}[scale=0.8]
\Iedge \pl \pu;
\Iedge \pr \pv;
\Iedge \pu \pv;
\Iedge \pu \pf;
\Iedge \pv \pf;
\Aedge \pl \pr;
\AvertW \pl 1;
\AvertE \pr 2;
\MvertN \pf 0;
\IvertW \pu 5;
\IvertE \pv 6;
\end{tikzpicture}
\\
$\cV=\cH\cup\cI$ & $\cH$ & $\cI$
\end{tabular}
\caption{$(\cH,\cI)$ is a non-trivial calligraphic split for $\cV$.}
\label{fig:HH}
\end{figure}

Our next goal is to determine the class of~$[\cH]$.
We first compute the number of realizations for the minimally rigid graphs
in \FIG{unions} using the algorithm described in \citep[Section~5]{rigid-alg}.
It follows from Axioms \AXM{1} and \AXM{2} that $[\cH]\cdot [\cL]=[\cH]\cdot [\cR]=8$ and $[\cH]\cdot [\cC_8]=24$,
where $[\cL]=(1,1,0)$, $[\cR]=(1,0,1)$ and $[\cC_8]=(2,0,0)$.
We solve the resulting system of linear equations, where the three entries of $[\cH]$
are the indeterminates, and obtain
\[
[\cH]=(6,2,2).
\]
In this particular example $\cH$ and $\cI$ are isomorphic as calligraphs, so $[\cH]=[\cI]$.
However, in general we need to compute $c(\cI\cup\cL)$, $c(\cI\cup\cR)$ and $c(\cI\cup\cC_8)$
as well and thus in total the numbers of realizations of six graphs with at most $|\vv(\cH)|+1$ vertices.
We conclude from Axiom \AXM{2} that
\[
c(\cU)=2\cdot c(\cV)=2\cdot c(\cH\cup\cI)=2\cdot [\cH]\cdot [\cI]=112.
\]

\begin{figure}[!ht]
\centering
\setlength{\tabcolsep}{5mm}
\begin{tabular}{ccc}
\begin{tikzpicture}[scale=0.8]
\Iedge \pl \pg;
\Iedge \pg \ph;
\Iedge \pg \pf;
\Iedge \ph \pf;
\Iedge \ph \pf;
\Iedge \ph \pr;
\Iedge \pl \pf;
\Aedge \pl \pr;
\AvertW \pl 1;
\AvertE \pr 2;
\IvertN \pg 3;
\IvertN \ph 4;
\MvertS \pf 0;
\end{tikzpicture}
&
\begin{tikzpicture}[scale=0.8]
\Iedge \pl \pg;
\Iedge \pg \ph;
\Iedge \pg \pf;
\Iedge \ph \pf;
\Iedge \ph \pf;
\Iedge \ph \pr;
\Iedge \pr \pf;
\Aedge \pl \pr;
\AvertW \pl 1;
\AvertE \pr 2;
\IvertN \pg 3;
\IvertN \ph 4;
\MvertS \pf 0;
\end{tikzpicture}
&
\begin{tikzpicture}[scale=0.8]
\Iedge \pl \pg;
\Iedge \pg \ph;
\Iedge \pg \pf;
\Iedge \ph \pf;
\Iedge \ph \pf;
\Iedge \ph \pr;
\Iedge \pl \pe;
\Iedge \pr \pe;
\Iedge \pe \pf;
\Aedge \pl \pr;
\AvertW \pl 1;
\AvertE \pr 2;
\IvertN \pg 3;
\IvertN \ph 4;
\IvertN \pe 8;
\MvertS \pf 0;
\end{tikzpicture}
\\
$\cH\cup\cL$
&
$\cH\cup\cR$
&
$\cH\cup\cC_8$
\end{tabular}
\caption{$c(\cH\cup\cL)=c(\cH\cup\cR)=8$ and $c(\cH\cup\cC_8)=24$ }
\label{fig:unions}
\end{figure}

\begin{remark}[algorithms]
\label{rmk:alg}
Our example works in general and can be translated into \ALG{class} and \ALG{nor}.
We shall refer to the algorithm described in \citep[Section~5]{rigid-alg}
for computing numbers of realizations by ``Algorithm~\cite{rigid-alg}''.
A graph can be determined to be minimally rigid in polynomial time
in the number of vertices~\cite{PebbleGame}.
We find calligraphic splits for a minimally rigid graph~$\cU$
by trying out each pair consisting of an edge and a vertex (there are less than $4\cdot\vv(\cU)^2$ possibilities).
In order to determine whether a calligraphic split~$(\cA,\cB)$ is non-trivial in \ALG{nor},
it is sufficient to verify that $|\vv(\cA)|\geq|\vv(\cB)|\geq 5$.
In this case $\cA,\cB\notin\{\cL,\cR,\cC_v\}$ for all $v\in\Z_{\geq 0}$
and thus \ALG{class} and \ALG{nor} do not end up in an infinite recursion.
Hence, the correctness of \ALG{class} and \ALG{nor} is a straightforward consequence of \AXM{1}, \AXM{2} and \PPP{3}.
In particular, each linear system of equations in \ALG{class} always has a unique solution.
See \APP{tree} for an example of the recursive execution tree of \ALG{nor}.
\END
\end{remark}

\begin{algorithm}[!ht]
\footnotesize
\caption{\texttt{getClass}}
\label{alg:class}
  \vspace*{02mm}$\bullet$ \textbf{Input.} A calligraph~$\cG$.
\\\hspace*{00mm}$\bullet$ \textbf{Output.} Its class~$[\cG]$.
\\\hspace*{00mm}$\bullet$ \textbf{Method.}
\\\hspace*{00mm}\textbf{if}\quad$\cG\cup\cL$ is not minimally rigid\quad\textbf{then}
\\\hspace*{05mm}\textbf{return}\quad$[\cG]:=(m,m,0)$\quad where\quad $m:=\frac{1}{2}\cdot\texttt{getNoR}(\cG\cup\cR)$.
\\\hspace*{00mm}\textbf{if}\quad$\cG\cup\cR$ is not minimally rigid\quad\textbf{then}
\\\hspace*{05mm}\textbf{return}\quad$[\cG]:=(m,0,m)$\quad where\quad $m:=\frac{1}{2}\cdot\texttt{getNoR}(\cG\cup\cL)$.
\\\hspace*{00mm}Determine $[\cG]$ by solving the linear system:
$\Bigl\{\texttt{getNoR}(\cG\cup\cF)=[\cG]\cdot [\cF]:\cF\in\{\cL,\cR,\cC_v\}\Bigr\}$.
\\\hspace*{00mm}\textbf{return} $[\cG]$
\end{algorithm}

\begin{algorithm}[!ht]
\footnotesize
\caption{\texttt{getNoR}}
\label{alg:nor}
  \vspace*{02mm}$\bullet$ \textbf{Input.} A minimally rigid graph~$\cU$.
\\\hspace*{00mm}$\bullet$ \textbf{Output.} The number of realizations~$\cc(\cU)$.
\\\hspace*{00mm}$\bullet$ \textbf{Method.}
\\\hspace*{00mm}\textbf{if}\quad $v\in\vv(\cU)$ has degree 2\quad\textbf{then}
\\\hspace*{05mm}\textbf{return}\quad $2\cdot\texttt{getNoR}(\cV)$
\quad where $\cV$ is obtained from $\cU$ by removing $v$.
\\\hspace*{00mm}Find a non-trivial calligraphic split $(\cA,\cB)$ for $\cU$.
\\\hspace*{00mm}\textbf{if}\quad no such split was found\quad\textbf{then}
\\\hspace*{05mm}\textbf{return}\quad $\cc(\cU)$\quad where $\cc(\cU)$ is computed using Algorithm~\cite{rigid-alg}.
\\\hspace*{00mm}\textbf{return}\quad $\texttt{getClass}(\cA)\cdot\texttt{getClass}(\cB)$
\end{algorithm}

\begin{remark}[experimental results]
\label{rmk:timing}
We call a graph \df{splittable} if it is a minimally rigid graph that admits a non-trivial calligraphic split.
Notice that if an input graph~$\cU$ is splittable, then \ALG{nor} reduces the computation of $\cc(\cU)$
to the computation of the numbers of realizations of six minimally rigid graphs with at most~$|\vv(\cU)|-1$ vertices.
To our best knowledge Algorithm~\cite{rigid-alg} is currently the
fastest symbolic algorithm for counting complex realizations.
We still expect that Algorithm~\cite{rigid-alg} is super exponential in~$|\vv(\cU)|$,
and indeed \FIG{timing} confirms heuristically that \ALG{nor} leads to a significant speed up
for splittable graphs~$\cU$ \st $|\vv(\cU)|\geq 15$.
In \FIG{chart} we charted the relative number of splittable minimally rigid graphs~$\cU$
\st $7\leq |\vv(\cU)|\leq 12$.
We used the \texttt{C++} version of Algorithm~\cite{rigid-alg} as a fallback
for graphs that are either unsplittable or have less than 13 vertices.
We remark that if instead we use the slower \textsc{Mathematica} version of Algorithm~\cite{rigid-alg},
then our \textsc{Mathematica} implementation of \ALG{nor} is also faster for splittable graphs with less than 13 vertices.
The implementations of the algorithms can be found at \cite{CalligraphsWL}.
\END
\end{remark}

\begin{figure}[ht]
\centering
\begin{tikzpicture}[yscale=1.2]
\draw[->] (12,0) -- (12,2.5) node[left,align=left,font=\small] {\%};
\draw[->] (12,0) -- (20,0) node[below right,font=\small] {\#vertices};
\foreach \y [evaluate=\y as \yc using \y/100] in {0,50,100,150,200} \draw (12,\yc) -- +(-0.05,0) node[left,font=\small] {$\y$};
\foreach \y in {0.25,0.75,...,2} \draw (12,\y) -- +(-0.05,0);
\foreach \y in {0.5,1,1.5,2} \draw (12,\y) -- +(-0.1,0);
\foreach \x in {13,14,...,19} \draw (\x,0) -- +(0,-0.005) node[below,font=\small] {$\x$};
\fill[colG] (13-0.25,0) -- (13-0.25,2) -- (13-0.1,2.15)--(13+0.1,1.85) -- (13+0.25,2) -- (13+0.25,0) --cycle;
\draw[black!40!white,dotted] (13-0.3,1) -- (13+0.3,1);
\foreach \x [count=\i from 14] in {1.69394, 0.928441, 0.560054, 0.391343, 0.285858, 0.20246}
{
    \fill[colG] (\i-0.25,0) rectangle (\i+0.25,\x);
    \draw[black!40!white,dotted] (\i-0.3,1) -- (\i+0.3,1);
}
\end{tikzpicture}
\caption{The timing for \ALG{nor} relative to Algorithm~\cite{rigid-alg}
using a set of 100 pseudo random graphs that are splittable.
The $100\%$ line represents the timing for Algorithm~\cite{rigid-alg}.
We see that \ALG{nor} is faster for splittable graphs with at least 15 vertices.
}
\label{fig:timing}
\end{figure}
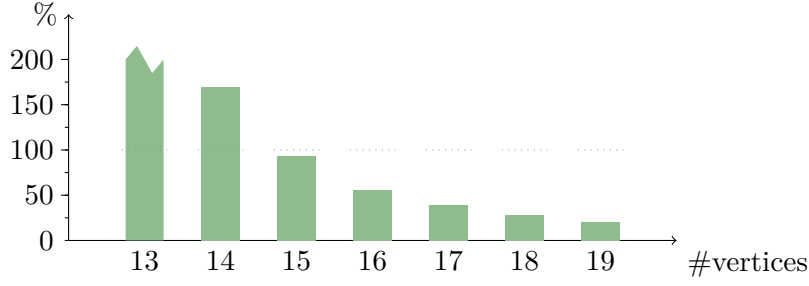

\begin{figure}[!ht]
\centering
\begin{tikzpicture}[yscale=3]
\draw[->] (6,0) -- (6,1.2) node[left,align=left,font=\small] {\%};
\draw[->] (6,0) -- (13,0) node[below right,font=\small] {\#vertices};
\foreach \y [evaluate=\y as \yc using \y/100] in {0,50,100} \draw (6,\yc) -- +(-0.05,0) node[left,font=\small] {$\y$};
\foreach \y in {0.25,0.75} \draw (6,\y) -- +(-0.05,0) node[left,font=\small] {$~$};
\foreach \y in {0.5,1} \draw (6,\y) -- +(-0.1,0);
\foreach \x in {7,8,9,10,11,12} \draw (\x,0) -- +(0,-0.005) node[below,font=\small] {$\x$};
\coordinate (a7) at (0,0.75);
\coordinate (xs) at (0.25,0);
\fill[colG] ($(7,0)-(xs)$) rectangle ($(7,0)+(a7)+(xs)$);
\fill[colGray] ($(7,0)+(a7)-(xs)$) rectangle ($(7,0)+(0,1)+(xs)$);
\coordinate (a8) at (0,0.65625);
\fill[colG] ($(8,0)-(xs)$) rectangle ($(8,0)+(a8)+(xs)$);
\fill[colGray] ($(8,0)+(a8)-(xs)$) rectangle ($(8,0)+(0,1)+(xs)$);
\coordinate (a9) at (0,0.689394);
\fill[colG] ($(9,0)-(xs)$) rectangle ($(9,0)+(a9)+(xs)$);
\fill[colGray] ($(9,0)+(a9)-(xs)$) rectangle ($(9,0)+(0,1)+(xs)$);
\coordinate (a10) at (0,0.659141);
\fill[colG] ($(10,0)-(xs)$) rectangle ($(10,0)+(a10)+(xs)$);
\fill[colGray] ($(10,0)+(a10)-(xs)$) rectangle ($(10,0)+(0,1)+(xs)$);
\coordinate (a11) at (0,0.621291);
\fill[colG] ($(11,0)-(xs)$) rectangle ($(11,0)+(a11)+(xs)$);
\fill[colGray] ($(11,0)+(a11)-(xs)$) rectangle ($(11,0)+(0,1)+(xs)$);
\coordinate (a12) at (0,0.574376);
\fill[colG] ($(12,0)-(xs)$) rectangle ($(12,0)+(a12)+(xs)$);
\fill[colGray] ($(12,0)+(a12)-(xs)$) rectangle ($(12,0)+(0,1)+(xs)$);
\end{tikzpicture}
\caption{Amount of splittable graphs \wrt all minimally rigid graphs
whose vertices are of degree at least 3 (thus each vertex belongs to at least 3 edges).
If the number of vertices is at most~$12$, then at least 50\% of such graphs is splittable.}
\label{fig:chart}
\end{figure}
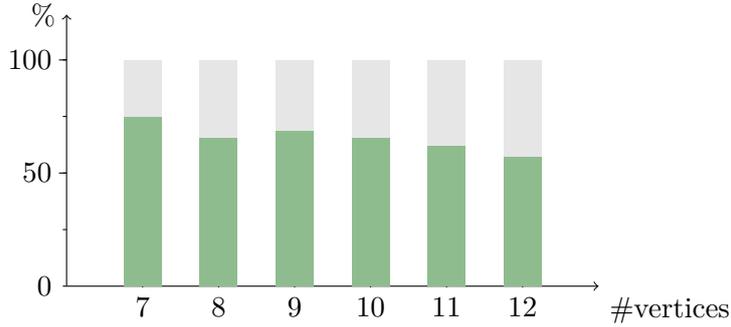

\subsection{Invariants of coupler curves and their intersections}
\label{sec:couplerH}
Suppose that
we would like to determine the degrees and geometric genera of the components
of the coupler curve of the calligraph~$\cH$ defined in~\FIG{HH}.

The equation of the coupler curve of $\cH$
can be determined explicitly and was done in \citep[Section~7.3]{hunt}.
Hence, its components, degrees and intersections with
other coupler curves can be computed in a straightforward manner.
However, for larger calligraphs this quickly becomes infeasible and therefore
we present an alternative method that works in purely combinatorial manner
without the need to perform computations in polynomial rings.
We assume that we can obtain an arbitrary good approximation
of the real part of the coupler curve of a calligraph for some given edge length assignment.
In practice this means that we have a drawing of the coupler curve as for example in~\FIG{tH}.

We apply \ALG{class} and find that~$[\cH]=(6,2,2)$.
In \FIG{tH} we illustrate the coupler curve~$T:=\tt_\omega(\cH)$ for edge length assignment~$\omega\in\Omega_\cH$.
We observe from this approximation that $T$ admits the
coupler decomposition $(T_1,T_2)$ \st $|\sng T_1|,|\sng T_2|\geq 3$
and $\deg T_1,\deg T_2\geq 6$.
As $\cH$ is thin and $\deg T=12$ by \PPP{2} and \PPP{5}, we find that $\deg T_1=\deg T_2=6$.
By \PPP{6} there exists a class partition~$(\alpha_1,\alpha_2)$
that is equal to either
$\bigl((3,0,0),(3,2,2)\bigr)$,
$\bigl((3,2,0),(3,0,2)\bigr)$,
$\bigl((3,1,2),(3,1,0)\bigr)$ or
$\bigl((3,1,1),(3,1,1)\bigr)$.
If $\alpha_1\in\{(3,0,0),(3,2,0),(3,1,2)\}$, then $g(T_1)<0$, which is impossible.
Therefore, $\alpha_1=\alpha_2=(3,1,1)$ so that for $i\in\{1,2\}$ we have
\[
0\leq g(T_i)\leq 4-|\sng T_i|\leq 1.
\]

\begin{figure}[!t]
\centering
\setlength{\tabcolsep}{10mm}
\begin{tabular}{cc}
\begin{tikzpicture}
\node[inner sep=0pt] at (1.6,0) {\includegraphics{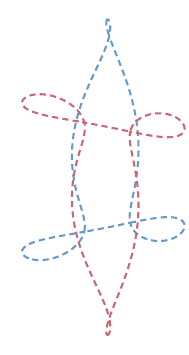}};
\Iedge \pl \pg;
\Iedge \pg \ph;
\Iedge \pg \pf;
\Iedge \ph \pf;
\Iedge \ph \pf;
\Iedge \ph \pr;
\Aedge \pl \pr;
\AvertS \pl 1;
\AvertS \pr 2;
\IvertN \pg 3;
\IvertN \ph 4;
\MvertS \pf 0;
\end{tikzpicture}
&
\begin{tikzpicture}
\node[inner sep=0pt] at (1.6,0) {\includegraphics{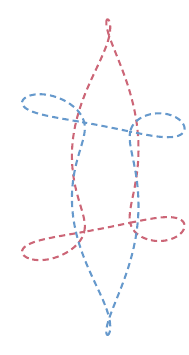}};
\Iedge \pl \pg;
\Iedge \pg \ph;
\Iedge \pg \pF;
\Iedge \ph \pF;
\Iedge \ph \pF;
\Iedge \ph \pr;
\Aedge \pl \pr;
\AvertS \pl 1;
\AvertS \pr 2;
\IvertN \pg 3;
\IvertS \ph 4;
\MvertE \pF 0;
\end{tikzpicture}
\end{tabular}
\caption{The coupler curve~$T$ of~$\cH$ consists
of two components $T_1$ and $T_2$ \st $\deg T_1=\deg T_2=6$, $0\leq g(T_1)\leq 1$ and $0\leq g(T_2)\leq 1$.}
\label{fig:tH}
\end{figure}

We illustrate in \FIG{TC},
for some choice of edge length assignment, the intersections of the coupler curves $T$ and $T'$ of
the calligraphs $\cH$ and $\cC_8$, \resp.
The coupler multiplicities are equal to one in both cases by \PPP{2}
and thus it follows from \AXM{1}, \AXM{2} and \PPP{4} that
\[
|T\cap T'|=[\cH]\cdot [\cC_8]=(6,2,2)\cdot(2,0,0)=24=c(\cH\cup\cC_8).
\]

\begin{figure}[!h]
\centering
\setlength{\tabcolsep}{8mm}
\begin{tabular}{ccc}
\begin{tikzpicture}
\node[inner sep=0pt] at (1.6,0) {\includegraphics{img/trajH.pdf}};
\Iedge \pl \pg;
\Iedge \pg \ph;
\Iedge \pg \pf;
\Iedge \ph \pf;
\Iedge \ph \pf;
\Iedge \ph \pr;
\Aedge \pl \pr;
\AvertS \pl 1;
\AvertS \pr 2;
\IvertN \pg 3;
\IvertN \ph 4;
\MvertS \pf 0;
\end{tikzpicture}
&
\begin{tikzpicture}
\draw[traj2] \pe circle [radius=1.11803];
\draw[traj] \pE circle [radius=1.11803];
\Iedge \pl \pe;
\Iedge \pr \pe;
\Iedge \pe \pf;
\Aedge \pl \pr;
\AvertN \pl 1;
\AvertN \pr 2;
\IvertN \pe 8;
\MvertN \pf 0;
\markerN{(0,-2.67)}{~}
\end{tikzpicture}
&
\begin{tikzpicture}
\node[inner sep=0pt] at (1.6,0) {\includegraphics{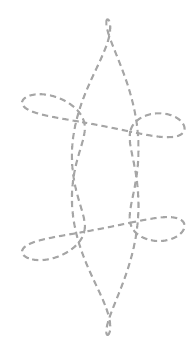}};
\draw[traj2] \pe circle [radius=1.11803];
\draw[traj] \pE circle [radius=1.11803];

\def\cb{colvmf2}
\def\cm{colvmf}
\foreach \p in {(1.33826,-0.401165), (1.47194,1.48547), (2.14618,1.60844), (2.2,-0.6), (1.,1.), (1.17048,1.2496), (1.21884,-0.299866), (2.34376,-0.563876), (3.,1.), (3.09104,0.744205)}
{
	\draw[draw=black, opacity=0.7, fill=\cb] \p   circle [radius=0.1];
}
\foreach \p in {(1.,-1.), (1.17048,-1.2496), (1.21884,0.299866), (2.34376,0.563876), (3.,-1.), (3.09104,-0.744205), (1.33826,0.401165), (1.47194,-1.48547), (2.14618,-1.60844), (2.2,0.6)}
{
	\draw[draw=black, opacity=0.7, fill=\cm] \p   circle [radius=0.1];
}
\end{tikzpicture}
\end{tabular}
\caption{The $20$ real intersections between the coupler curves of
$\cH$ and $\cC_8$ in $\C^2$.
There are $[\cH]\cdot[\cC_8]=24$ complex intersections in total.}
\label{fig:TC}
\end{figure}

We observe that 20 of the complex intersection points in $T\cap T'$ are real.
See \citep[Figure~5]{rigid-num-borcea} for an example where all 24 intersections are real.

\subsection{Invariants of a calligraph with 10 vertices}
\label{sec:large}

Let us consider the calligraph $\cF$ as defined in \FIG{G2G3}.
Suppose that $\omega\in\Omega_{\cF}$ is a general edge length assignment
and let $T=\tt_\omega(\cF)$ be a coupler curve of the calligraph~$\cF$
with coupler decomposition~$(T_1,\ldots,T_n)$.
Our goal is to determine $\deg T$ and a lower bound for~$n$.
Moreover, we determine upper bounds for the degrees and the geometric genera of the irreducible components~$T_1,\ldots,T_n$.

\begin{figure}[!ht]
\centering
\setlength{\tabcolsep}{8mm}
\begin{tabular}{@{}cc@{}}
\begin{tikzpicture}[scale=0.4]
\BUedge \Qbe \Qbd;
\BUedge \Qbe \Qcd;
\BUedge \Qaa \Qbb;
\BUedge \Qaa \Qkk;
\BUedge \Qea \Qbb;
\BUedge \Qea \Qkk;
\BUedge \Qaa \Qbc;
\BUedge \Qbc \Qbd;
\BUedge \Qbc \Qcc;
\BUedge \Qkk \Qjj;
\BUedge \Qbb \Qcc;
\BUedge \Qcc \Qjj;
\BUedge \Qjj \Qcd;
\BUedge \Qcd \Qbd;
\BUedge \Qjj \Qbd;
\Aedge \Qaa \Qea;
\BMvertN \Qbe 0;
\BAvertS \Qaa 1;
\BAvertS \Qea 2;
\BIvertS \Qbb \
\BIvertE \Qkk \
\BIvertN \Qcc \
\BIvertE \Qjj \
\BIvertW \Qbc \
\BIvertW \Qbd \
\BIvertN \Qcd \
\end{tikzpicture}
&
\begin{tikzpicture}[scale=0.4]
\BVedge \Qbe \Qbd;
\BVedge \Qbe \Qcd;
\BVedge \Qaa \Qbb;
\BVedge \Qaa \Qkk;
\BVedge \Qea \Qbb;
\BVedge \Qea \Qkk;
\BUedge \Qaa \Qbc;
\BUedge \Qbc \Qbd;
\BUedge \Qbc \Qcc;
\BUedge \Qkk \Qjj;
\BUedge \Qbb \Qcc;
\BUedge \Qcc \Qjj;
\BVedge \Qjj \Qcd;
\BVedge \Qcd \Qbd;
\BVedge \Qjj \Qbd;
\Aedge \Qaa \Qea;
\BMvertN \Qbe 0;
\BAvertS \Qaa 1;
\BAvertS \Qea 2;
\BIvertS \Qbb 3;
\BIvertE \Qkk 4;
\BIvertN \Qcc 5;
\BIvertE \Qjj 6;
\BIvertW \Qbc 7;
\BIvertW \Qbd 8;
\BIvertN \Qcd 9;
\end{tikzpicture}
\end{tabular}
\caption{The calligraph $\cF$ with
$[\cF]=(272,0,0)$ illustrated with the two largest minimally rigid subgraphs
that contain vertex~$0$ and edge~$\{1,2\}$, \resp.}
\label{fig:G2G3}
\end{figure}
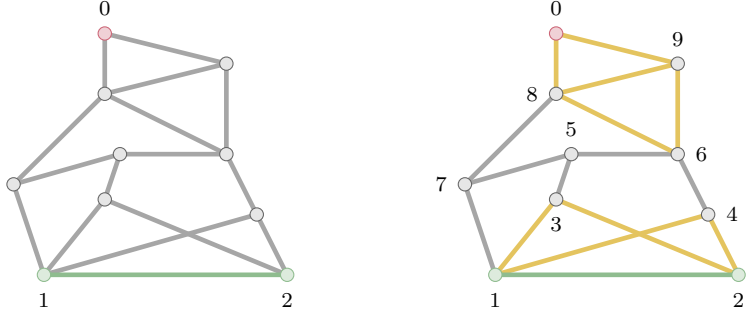

We apply \ALG{class} and find that
$[\cF]=(272,0,0)$, where $272=2^4\cdot 17$.
We verify that $\cF$ is thin and thus $\deg T=2\cdot 272=544$ by \PPP{2} and \PPP{5}.

The largest minimally rigid subgraph $\cF$ that
contains the vertex 0 is induced by the vertices $\{0,6,8,9\}$.
Let $\cF'$ be defined by this graph, but with vertices $6$ and $8$ relabeled to $1$ and $2$, \resp,
so that $\cF'$ is marked.
The largest minimally rigid subgraph of~$\cF''\subset\cF$ that contains the edge~$\{1,2\}$
is induced by $\{1,2,3,4\}$.

It is straightforward to see that it is not possible to continuously move vertex 0 of $\cF$ while fixing
vertices 1, 2, 3 and 4 and arrive at a position where vertex~0 is flipped along the edge $\{8,6\}$
or $\{8,9\}$.
With \df{flipping} a vertex along an edge we mean
that we reflect the vertex along the line that is spanned by this edge,
and in this manner obtain another realization of~$\cF$.
Such an almost everywhere continuous movement, without flipping any vertices along the way,
defines a component of the coupler curve~$T$
(we believe that such a component must be irreducible).
Since $\cc(\cF')=4$ and $\cc(\cF'')=4$ we obtain in this manner~$16$ possibly reducible components
and their union coincides with the coupler curve.
Thus, we established the following lower bound for the number of irreducible components:~$n\geq 16$.

By \PPP{5} the degree of a coupler curve is uniquely determined by the class of a calligraph
and therefore does not depend on the starting position of the vertices before we move the graph.
From this we deduce that the components that are related by flipping a vertex in our example are of equal degree.
It follows that each of the obtained possibly reducible components is of degree~$d$.
Since $\deg T=544=16\cdot d$ we find that $d=34$.
We established an upper bound for the degrees of irreducible components:
$\deg T_i\leq 34$ for all $1\leq i\leq n$.

In order to determine an upper bound for the geometric genera of the irreducible components,
we apply \PPP{6}: there exists a class partition~$(\alpha_1,\dots,\alpha_n)$
\st $\alpha_i=(t,0,0)$ for some $t\leq 17$
so that for all $1\leq i\leq n$ the following holds:
\[
g(T_i)\leq \tfrac{1}{2}\cdot (t,0,0)\cdot (t-2,-1,-1)+1\leq 17\cdot 15+1=256.
\]
In comparison, the geometric genus of a curve of degree 34
is at most $(34-1)(34-2)/2=528$ by the genus formula at~\citep[Exercise~I.7.2b]{har}.

\section{Base points and pseudo classes of calligraphs}
\label{sec:bp}

If we vary the edge length assignments for a given calligraph~$\cG$,
then we obtain a set of coupler curves.
The set of defining polynomials of these coupler curves
can be represented in terms of a single polynomial~$F$ with coefficients in the edge lengths.
After embedding the coupler curves of~$\cG$ into the projective plane,
we observe that all their Zariski closures pass through certain complex points at infinity.
In this section we show how to recover from $F$ such base points and their multiplicities
in terms of substitutions and polynomial quotients.
We conclude this section by proving \PRP{pseudo},
which states that the 3-tuple consisting of multiplicities at three distinguished base points
is under a certain hypothesis a class, and that we can recover
from such classes invariants of coupler curves.

Suppose that $\fF\subset\C[x,y]$ is a subset of polynomials.
In this article, we assume that
$\fF=\set{F(c_1,\ldots,c_n,x,y)}{c_1,\ldots,c_n\in \C}\subset\C[x,y]$
for some polynomial $F\in\C[u_1,\ldots,u_n,x,y]$.
We say that $\fF$ has a \df{base point} at $p:=(x_p,y_p)$ in $\C^2$ if $f(p)=0$ for all $f\in\fF$.
This base point is called \df{$m$-fold} if
for all $f\in \fF$
the lowest-degree monomial term of~$f(x+x_p,y+y_p)$
has degree $\geq m$
and there exists $g\in\fF$ \st
the lowest-degree monomial term of~$g(x+x_p,y+y_p)$
has degree $m$.

\begin{definition}
\label{def:ab}
Let $\alpha_p^m,\beta_p^m\c\C[x,y]\to\C[x,y]$ for $p:=(x_p,y_p)\in\C^2$ and $m\in\Z_{\geq 0}$ be the maps
\[\alpha_p^m(f):=f(x+x_p,yx+y_p)\div x^m\qquad\text{and}\qquad\beta_p^m(f):=f(xy+x_p,y+y_p)\div y^m,\]
where $f\div g$ for polynomials $f,g\in\C[x,y]$ denotes the \df{polynomial quotient}.
\END
\end{definition}

A base point $q:=(x_q,y_q)\in\C^2$
is \df{infinitely near} to an
$m$-fold base point $p$ of $\fF$
if either
\begin{Mlist}
\item $q$ is a base point of $\alpha_p^m(\fF)$ \st $x_q=0$, or
\item $q$ is a base point of $\beta_p^m(\fF)$ \st $x_q=y_q=0$.
\end{Mlist}
Similarly, we define base points that are infinitely near to $q$ and so on.
See forward \RMK{blowup} for an algebro geometric interpretation of
infinitely near base points.

We denote by $\ii$ the \df{imaginary unit}.

\begin{example}
\label{exm:bp}
Let $\fF:=\set{(1-x)^2+y^2-\ell^2x^2 }{\ell\in\C}$ be a subset of
polynomials in $\C[x,y]$.
The base points $p:=(0,\ii)$ and $\overline{p}:=(0,-\ii)$ of $\fF$ are both 1-fold.
For determining whether $p$ admits infinitely near base points we consider
\begin{align*}
\alpha_p^1(\fF)&=\set{-2 + x + 2\ii y + x y^2 - \l^2 x}{\ell\in \C}
\quad\text{and}\quad
\\\beta_p^1(\fF)&=\set{2\ii - 2x + y + x^2y - \l^2 x^2 y}{\ell\in \C}.
\end{align*}
We find that
$q:=(0,-\ii)$
is the 1-fold base point of $\alpha_p^1(\fF)$.
Notice that the complex conjugate $\overline{q}=(0,\ii)$ is the 1-fold base point of
$\alpha_{\overline{p}}^1(\fF)$.
We verify that both $\alpha_q^1\circ\alpha_p^1(\fF)$ and
$\beta_q^1\circ\alpha_p^1(\fF)$ are base point free
and thus we identified all base points of $\fF$.
\END
\end{example}

The \df{projective plane}~$\P^2$ is defined as $\bigl(\C^3\setminus\{(0,0,0)\}\bigr)/\sim$,
where
\[
(z_0:z_1:z_2)\sim (\lambda\,z_0:\lambda\,z_1:\lambda\,z_2)
\qquad\text{for all}\qquad
\lambda\in\C\setminus\{0\}.
\]
If $h\in\C[z_0,z_1,z_2]$ is homogeneous, then we denote by $V(h)\subset \P^2$ its \df{zero set}.
The \df{line at infinity} is defined as $V(z_0)$.
We define $\gamma_0,\gamma_1,\gamma_2\c\C[z_0,z_1,z_2]\to\C[x,y]$ as
\[
\gamma_0(h):=h(1,x,y),\qquad
\gamma_1(h):=h(x,1,y)\qquad\text{and}\qquad
\gamma_2(h):=h(y,x,1).
\]
The embeddings $\gamma^*_0,\gamma^*_1,\gamma^*_2\c\C^2\hookrightarrow\P^2$ are defined as
\[
\gamma^*_0(x,y):=(1:x:y),\qquad
\gamma^*_1(x,y):=(x:1:y)\qquad\text{and}\qquad
\gamma^*_2(x,y):=(y:x:1).
\]
Suppose that $\fH\subset\C[z_0,z_1,z_2]$ is a subset of homogeneous polynomials
of the same degree.
The \df{series associated to}~$\fH$
is defined as the following set of projective curves:
\[
\Gamma_\fH:=\set{V(h)\subset\P^2}{h\in \fH}.
\]
We call $r:=(r_0:r_1:r_2)\in \P^2$ a \df{base point} of the series~$\Gamma_\fH$ if
$r=\gamma_j^*(p)$ for some $p:=(p_x,p_y)$ \st either
\begin{Mlist}
\item $j=0$ and $p$ is a base point of $\gamma_0(\fH)$,
\item $j=1$ and $p$ is a base point of $\gamma_1(\fH)$ \st $p_x=0$, or
\item $j=2$ and $p$ is a base point of $\gamma_2(\fH)$ \st $p_x=p_y=0$.
\end{Mlist}
If $q\in \C^2$ is a base point that is infinitely near to $p$,
then we say that $q$ is also \df{infinitely near to}~$r$.
We call $r\in\P^2$ \df{cyclic} if it is equal to either $(0:1:\ii)$ or $(0:1:-\ii)$.

Suppose that $r\in\P^2$ is a base point of the series~$\Gamma_\fH$
\st $r=\gamma_1^*(p)$, where
$p=(0,\pm\ii)$ is an $m$-fold base point of $\gamma_1(\fH)$,
and $q\in \C^2$ is a base point of $\alpha_p^m\circ\gamma_1(\fH)$.
In this case we call
\begin{Mlist}
\item $r$ a \df{cyclic} base point of~$\Gamma_\fH$,
\item $q$ a \df{1-centric} base point of~$\Gamma_\fH$ if $q=(0,0)$, and
\item $q$ a \df{2-centric} base point of~$\Gamma_\fH$ if $q=(0,\mp\ii)$ \st $y_p=-y_q$.
\end{Mlist}
Notice that 1-centric and 2-centric base points are infinitely near to a cyclic base point.
We call the series $\Gamma_\fH$ \df{centric} if all its base points are either cyclic, 1-centric
or 2-centric, and if almost all curves in the series meet the
line at infinity only at the cyclic points.

\begin{example}
\label{exm:hbp}
Let $\fH:=\set{(z_1-z_0)^2+z_2^2-\ell^2\,z_0^2 }{\ell\in\C}$. Then
\begin{align*}
\gamma_0(\fH)&=\set{(x-1)^2+y^2-\ell^2}{\ell\in\C},\quad\\
\gamma_1(\fH)&=\set{(1-x)^2+y^2-\ell^2\,x^2}{\ell\in\C},\\
\gamma_2(\fH)&=\set{(x-y)^2+1-\ell^2\,y^2}{\ell\in\C}.
\end{align*}
We observe that $\gamma_0(\fH)$ is base point free and defines a pencil of
circles in $\C^2$ that are centered around $(1,0)$.
We deduce from \EXM{bp} that
$\gamma_1(\fH)$ has two complex conjugate base points
whose $\gamma_1^*$-images in $\P^2$ are cyclic,
and that the complex conjugate infinitely near base points are 2-centric.
Finally, we check that $(0,0)$ is not a base point of $\gamma_2(\fH)$.
Thus there are in total four complex base points and each of them
is $1$-fold.
A conic in the series $\Gamma_\fH$ meets the line
at infinity at exactly the two cyclic points,
and thus we established that $\Gamma_\fH$ is centric.
\END
\end{example}

Suppose that $\cG$ is a calligraph. We consider the embedding of its coupler curves into
the projective plane, namely we
let $T_\omega\subset\P^2$ be the Zariski closure of the curve~$\gamma_0^*(\tt_\omega(\cG))$.
Define $\fH\subset\C[z_0,z_1,z_2]$ as the set of square free homogeneous polynomials
\st
\[
\set{V(h)\subset\P^2}{h\in \fH}=\set{T_\omega\subset\P^2}{\omega\in \Omega_\cG}.
\]
We shall see in \LEM{dim}\ref{lem:dim:c} that $\fH$ can be represented similarly as in \EXM{hbp},
namely in terms of a square free homogeneous polynomial in $z_0$, $z_1$ and $z_2$,
whose coefficients are polynomials in the variables for the edge lengths.
The series~$\Gamma_\fH$ associated to~$\fH$ is called the
\df{series associated to the calligraph}~$\cG$.
We call $\cG$ \df{centric} if its associated series~$\Gamma_\fH$ is centric.

\begin{definition}
\label{def:pclass}
The \df{pseudo class} for calligraphs assigns to each calligraph~$\cG$ the tuple
\[
[\cG]:=(m\,a,m\,b,m\,c)\in \Z^3,
\]
where $m$ equals the coupler multiplicity~$\mm(\cG)$
and where the series $\Gamma_\fH$ associated to~$\cG$ has
an $a$-fold cyclic base point,
a $b$-fold 1-centric base point, and a $c$-fold 2-centric base point.
If $\Gamma_\fH$ has no cyclic, 1-centric or 2-centric base points,
then $a=0$, $b=0$ and $c=0$, \resp.
\END
\end{definition}

\begin{remark}
\label{rmk:real}
If a cyclic base point of the series associated to a calligraph has multiplicity~$n$,
then both cyclic base points must have multiplicity~$n$. This statement also holds
if we replace ``cyclic'' with ``1-centric'' or ``2-centric''.
The reason is that for real edge length assignments,
the coupler curves of a calligraph
are real and thus the multiplicities at complex conjugate base points
of the associated series must be equal (see \EXM{bp}).
\END
\end{remark}

\begin{example}
\label{exm:R}
Suppose that $\Gamma_\fH$ is the series associated to the calligraph $\cR$ as
defined in \FIG{cal}.
The coupler multiplicity of $\cR$ is equal to $1$.
The coupler curves of $\cR$ define circles centered around
the point $(1,0)$ and thus $\fH$ is defined as in \EXM{hbp}.
Hence, $\Gamma_\fH$ is centric with two complex conjugate 1-fold cyclic base points,
and two complex conjugate 1-fold 2-centric base points.
We conclude that $\cR$ has pseudo class $(1,0,1)$.
\END
\end{example}

The following proposition shows that
the pseudo class is a class under the assumption that calligraphs are
centric. We show in \SEC{proof} that calligraphs are indeed centric.

\begin{proposition}
\label{prp:pseudo}
If calligraphs are centric, then the following holds:
\begin{Mlist}
\item
The coupler multiplicity and the pseudo class for calligraphs are well-defined.
\item
The pseudo class satisfies the Axioms \AXM{1}, \AXM{2}, and the six Properties \PPP{1}---\PPP{6}.
\item
A class for calligraphs must be the unique pseudo class.
\end{Mlist}
\end{proposition}

For its proof we assume that the reader is familiar with
intersection theory for surfaces (see for instance \citep[Sections V.1 and V.3]{har}).
Before we prove \PRP{pseudo} let us in preparation first
give the algebro geometric interpretation of infinitely near base points in
\RMK{blowup}
and prove \LEMS{lem:dim,lem:trans,lem:mult,lem:glob}.
\LEM{glob} is based on a result about ``global rigidity'' in \citep[Corollary~1.7]{rigid-con}.

\begin{remark}
\label{rmk:blowup}
Our definition for infinitely near base points is motivated as follows.
Let
\[
W_i:=\set{(z_0:z_1:z_2)\in\P^2}{z_i\neq 0}
\]
so that $\P^2=W_0\cup W_1\cup W_2$ and $\gamma_i^*(\C^2)=W_i$.
Suppose that $V(h)\subset\P^2$
is a curve that has multiplicity $m>0$ at some point~$r$ (see \citep[Exercise~I.5.3]{har} for multiplicities of curves at points).
Let us assume that $r\notin W_0$ and $r\in W_1$ so that $r=\gamma_1^*(p)$ for some $p\in\C^2$.
In this case the preimage of $V(h)$ \wrt $\gamma_1^*$
is defined as the zero set~$V(f)\subset \C^2$, where $f:=h\circ\gamma_1^*=h(x,1,y)$ is a polynomial in~$\C[x,y]$.
Notice that $V(f)$ has multiplicity~$m$ at~$p$,
since $\gamma_1^*$ defines an isomorphism~$\C^2\to W_1$.
We translate $p$ to the origin and consider the corresponding
curve
\[
C:=V(f(x+x_p,y+y_p)).
\]
The blowup of $\C^2$ at the origin is defined as follows (see \citep[Section~I.4]{har}):
\[
U:=\set{(x,y;\,t_0:t_1)\in\C^2\times\P^1}{x\,t_1=y\,t_0}.
\]
We find that $U=U_0\cup U_1$ where
\[
U_i:=\set{(x,y;t_0:t_1)\in U}{t_i\neq 0}.
\]
In the chart~$U_0$ we have the relation $y=x\tilde{y}$ with $\tilde{y}:=t_1/t_0$.
Thus the map $\phi\c \C^2\to U_0$ that sends $(x,\tilde{y})$ to $(x,\tilde{y}x;1:\tilde{y})$
is an isomorphism.
Let $\rho\c U_0\to \C^2$ be the projection that sends $(x,y;\,t_0:t_1)$ to $(x,y)$.
We find that the composition $\rho\circ\phi$ sends
$(x,y)$ to $(x,yx)$ and is an isomorphism when $x\neq 0$.
The Zariski closure of the fiber $\rho^{-1}(0,0)$ is called an \df{exceptional curve}
and the Zariski closure of~$\rho^{-1}(C)$ contains this exceptional curve
with multiplicity $m$ (see \citep[Example~I.4.9.1 and Proposition~V.3.6]{har}).
This exceptional curve corresponds via $\phi$ to the line~$V(x)\subset \C^2$.

The $(\rho\circ\phi)$-preimage of $C$ is defined as
the zero set of $f(x+x_p,yx+y_p)=g\circ\rho\circ\phi$ with $g(x,y):=f(x+x_p,y+y_p)$.
After removing the $(x^m)$-component,
this preimage is defined as the zero set of $\alpha^m_p(f)$.
As $\rho\circ\phi$ is an isomorphism outside $V(x)$,
we are only interested in $(x_q,y_q)\in V(\alpha^m_p(f))$ \st $x_q=0$.
In this case $y_q$ corresponds to a tangent direction of $C$ at
the origin and
thus gives a geometric interpretation of an infinitely near point (see \citep[Example~I.4.9.1 and Figure~I.3]{har}).

For the remaining chart $U_1$ we consider $\beta^m_p(f)$
instead of $\alpha^m_p(f)$.
Because $U_1\setminus U_0=\set{(x,y;\,t_0:t_1)\in U}{t_0=0,~t_1=1}$,
it follows that to complete the analysis we only
need to check the multiplicity of $(0,0)$ in $V(\beta^m_p(f))$.
Indeed, the definitions in this section are chosen \st each
point in the overlapping charts~$W_0\cap W_1$, $W_0\cap W_2$, $W_1\cap W_2$ and $U_0\cap U_1$
is only considered once.
\END
\end{remark}

\begin{lemma}
\label{lem:dim}
Suppose that $\cG$ is a calligraph. Let
$\V:=\vv(\cG)\setminus\{1,2\}$,
$\E:=\ee(\cG)\setminus\{\{1,2\}\}$,
$(x_1,y_1):=(0,0)$ and $(x_2,y_2):=(1,0)$.
\begin{Menum}

\item\label{lem:dim:a}
If $V_\cG\subset\C^{2|\V|}\times\C^{|\E|}$ denotes the zero set of
\[
\set{(x_i-x_j)^2+(y_i-y_j)^2-\l_{\{i,j\}}^2}{\{i,j\}\in\E}\subset \C[\l_e,x_i,y_i:e\in\E,i\in\V],
\]
then $\dim V_\cG=1+|\E|$. In particular, if
$\lambda=(\lambda_e)_{e\in\E}\in\C^{|\E|}$ is general and
$H_\lambda\subset\C^{2|\V|}\times\C^{|\E|}$ is the zero set of $\set{\l_e-\lambda_e}{e\in\E}$,
then $\dim(V_\cG\cap H_\lambda)=1$.

\item\label{lem:dim:b}
There exists a square free polynomial $F\in\C[\l_e,x_0,y_0:e\in\E]$ \st
for all edge length assignments~$\omega\in\Omega_\cG$,
the coupler curve~$\tt_\omega(\cG)$ is equal to the zero set of $F(\lambda,x_0,y_0)$,
where $\lambda=(\omega(e))_{e\in\E}$.

\item\label{lem:dim:c}
There exists a square free polynomial $H\in\C[\l_e,z_0,z_1,z_2:e\in\E]$
that is homogeneous in the variables $z_0$, $z_1$ and $z_2$
\st the series $\Gamma_\fH$ associated to the calligraph~$\cG$
is defined as the series associated to the following subset of homogeneous polynomials:
\[
\fH:=\set{H(\lambda,z_0,z_1,z_2)}{\lambda\in\C^{|\E|}}\subset\C[z_0,z_1,z_2].
\]

\item\label{lem:dim:d}
The coupler curve~$\tt_\omega(\cG)$ is 1-dimensional
for almost all~$\omega\in\Omega_\cG$.

\end{Menum}
\end{lemma}

\begin{proof}
\ref{lem:dim:a}
Recall that the calligraph~$\cG$ can be obtained by removing an edge~$\{u,v\}$
from some minimally rigid graph~$\cM$.
We define $V_\cM:=V_\cG\cap M$,
where $M\subset\C^{2|\V|}$ denotes the zero set of
the polynomial $(x_u-x_v)^2+(y_u-y_v)^2-\lambda_{\{u,v\}}^2$ for some general edge length~$\lambda_{\{u,v\}}\in\C$.
By construction, a point in $V_\cM\cap H_\lambda$ corresponds to a realization of~$\cM$ for some general edge lengths
determined by $\lambda=(\lambda_e)_{e\in\E}$ and $\lambda_{\{u,v\}}$.
Recall that for a general choice of edge lengths, a minimally rigid graph
admits at least one, but only finitely many realizations.
Therefore, $V_\cM\cap H_\lambda$ is non-empty and $\dim(V_\cM\cap H_\lambda)=0$.
This implies that $\dim V_\cM=|\E|$.
Moreover, $\dim V_\cG>|\E|$ as $\cG$ is not minimally rigid.
It follows from \citep[Theorem~0.2]{3264} that
\[
\dim V_\cM\geq \dim V_\cG - \codim M= \dim V_\cG - 1.
\]
Hence, $\dim V_\cG=1+|\E|$ and $\dim(V_\cG\cap H_\lambda)=1$ as was to be shown.

\ref{lem:dim:b}
Let $\rho\c\C^{2|\V|}\times\C^{|\E|}\to\C^2\times\C^{|\E|}$ denote the linear projection
that sends
\[
(x_i,y_i,\l_e)_{i\in\V,e\in\E}
\]
to $(x_0,y_0,\l_e)_{e\in\E}$.
We know from \citep[Theorem~3 in Section~3.2]{cox}
that the ideal of~$\rho(V_\cG)$
is obtained by eliminating the variables~$\set{x_i,y_i}{i\in \V\setminus\{0\}}$
from the ideal $\bas{V_\cG}$ of $V_\cG$.
By assertion~\ref{lem:dim:a}, $V_\cG\cap H_\lambda$ is a non-linear curve and thus
its projection to the $(x_0,y_0)$-plane is again a curve.
As a direct consequence of the definitions, this planar curve corresponds to the coupler curve~$\tt_\omega(\cG)$,
where $\lambda=(\omega(e))_{e\in\E}$.
This implies that the elimination ideal $\bas{V_\cG}\cap\C[\l_e,x_0,y_0]$
is generated by a single polynomial.
We define~$F(\l,x_0,y_0)$ to be the square free part of this polynomial.
Notice that the zero set of $F(\lambda,x_0,y_0)$ is equal to~$\tt_\omega(\cG)$, which was to be shown.

\ref{lem:dim:c}
This assertion is a direct consequence of assertion~\ref{lem:dim:b},
since the square free polynomial~$H$ is the homogenization
of~$F$ \wrt the variables $x_0$ and $y_0$.

\ref{lem:dim:d} This assertion is a direct consequence of assertion~\ref{lem:dim:b}.
\end{proof}

\begin{lemma}
\label{lem:trans}
If the calligraphs $\cG$ and $\cG'$ are centric,
then the coupler curves~$\tt_{\omega}(\cG)$ and $\tt_{\omega'}(\cG)$
are curves that
intersect transversally for almost all
edge length assignments~$\omega\in \Omega_{\cG}$ and~$\omega'\in \Omega_{\cG'}$.
\end{lemma}

\begin{proof}
Let $F(\l,z)$ denote the square free polynomial in \LEM{dim}\ref{lem:dim:b},
where we renamed the variables~$(x_0,y_0)$ to $z:=(z_1,z_2)$.
Thus, the coupler curve~$\tt_\omega(\cG)$ is equal to the zero set of~$F(a,z)$,
where $a_e$ equals the edge length~$\omega(e)$ for all~$e\in \E$.
Let $F_i$ for $i\in\{1,2\}$ denote the partial derivative~$\partial_{z_i} F$
\wrt $z_i$.
Analogously, we obtain a polynomial~$G(b,z)$ for $\cG'$ and its partial
derivatives $G_1$ and $G_2$.
Since for general edge length assignments the coupler curves of $\cG$ and $\cG'$
do not have a common component, it follows that $F$ and~$G$ have no common factor.

For all $p\in\C^2$ we consider the set
\begin{multline*}
U_p:=\bigl\{
(\alpha,\beta)\in\C^m\times\C^n
:\\
F(\alpha,p)=G(\beta,p)=
F_1(\alpha,p)\cdot G_2(\beta,p)-F_2(\alpha,p)\cdot G_1(\beta,p)=0
\bigr\},
\end{multline*}
where $m:=|\ee(\cG)|-1$ and $n:=|\ee(\cG')|-1$.

First suppose that both $\cG$ and $\cG'$ admit only a 1-dimensional family of
coupler curves. Thus, the
coupler curves of $\cG$ and $\cG'$ are \Wlog the circles centered at $(0,0)$ and $(1,0)$, \resp.
In this case, the main assertion holds
and thus we may assume in the remainder of the proof that $\cG$ admits
a $d$-dimensional family of coupler curves \st $d\geq 2$.
This implies that $m\geq 2$ and thus $m+n\geq 3$.

We suppose by contradiction that
for general~$\omega\in\Omega_\cG$ and $\omega'\in\Omega_{\cG'}$
the coupler curves~$\tt_{\omega}(\cG)$ and $\tt_{\omega'}(\cG)$
intersect non-transversally at some point $p\in\C^2$.
Algebraically, this means that for all $(a,b)\in\C^m\times\C^n$
there exists $p\in\C^2$ \st $(a,b)\in U_p$.
Indeed, the equations $F(a,p)=0$ and $G(b,p)=0$ indicate that both coupler curves pass
through the point~$p$ and the Jacobian determinant equation $F_1(a,p)\cdot G_2(b,p)-F_2(a,p)\cdot G_1(b,p)=0$
means that the tangent vectors of the coupler curves are either linear dependent at~$p$,
or $p$ is a singular point of at least one of the coupler curves.
Hence, we established that
\[
\dim \bigcup_{p\in \C^2}U_p=m+n.
\]
Let $A_p:=\set{\alpha\in\C^m}{F(\alpha,p)=0}$ and $B_p:=\set{\beta\in\C^n}{G(\beta,p)=0}$.
Since the calligraphs are centric, the
series associated to these calligraphs only admit base points at infinity.
It follows that there does not exists
a point $p\in\C^2$ \st $A_p=\C^m$ and $B_p=\C^n$.
This implies that for all~$p\in \C^2$ we have
\[
\dim U_p\leq \dim A_p\times B_p=m+n-2.
\]

\textbf{Claim 1.} \textit{%
For general $p\in \C^2$, $a\in A_p$ and $b\in B_p$ we have}
\[
F_1(a,p)\cdot G_2(b,p)-F_2(a,p)\cdot G_1(b,p)\neq 0.
\]
Let the coupler curves~$C_\alpha$ and $D_\beta$ denote the zero sets in $\C^2$ of $F(\alpha,z)$ and $G(\beta,z)$, \resp.
We observe that $\dim A_q=\dim A_p$ for almost all~$q\in\C^2$
and thus instead of first fixing a general~$p\in\C^2$ and afterwards a general~$a\in A_p$, we may equivalently
first fix a general $a\in\C^m$ and afterwards a general $p\in C_a$ so that $a$ is general in~$A_p$ as well.
It follows that $p\notin\sng C_a$ as a general point in~$C_a$ is smooth.
The analogous argument shows that $p\notin\sng D_b$.
This implies that $(F_1(a,p),F_2(a,p))\neq (0,0)$ and $(G_1(b,p),G_2(b,p))\neq (0,0)$.
We set $s_\alpha(z):=F_1(\alpha,z)/F_2(\alpha,z)$ and $t_\beta(z):=G_1(\beta,z)/G_2(\beta,z)$.
Now suppose by contradiction that $F_1(a,p)\cdot G_2(b,p)-F_2(a,p)\cdot G_1(b,p)=0$.
In this case, the curves~$C_a$ and $D_b$ must intersect tangentially at $p$.
The slope $s_a(p)=t_b(p)$ of their mutual tangent line at~$p$
does not depend on $a$ or $b$, as $(a,b)\in A_p\times B_p$ is general.
As $p$ is general as well,
we established that $s_\alpha(z)=t_\beta(z)$ for almost all $(z,\alpha,\beta)\in \C^2\times A_p\times B_p$.
Therefore, the rational function~$u(z):=s_\alpha(z)=t_\beta(z)$ only depends on $z\in\C^2$
and the $\alpha$'s and $\beta$'s cancel out.
We established that~$u(z)$
defines a slope field that is continuous for almost all~$z\in \C^2$,
and in particular, $u(z)$ is continuous at~$p$.
By construction, the curve~$C_a$ is an integral curve for the slope field
that passes through~$p$.
By the implicit function theorem we may assume \Wlog that $C_a$
is in a complex analytic neighborhood~$W_p$ around $p$ defined by the graph of the function~$g(z_1)$.
It follows from the Picard-Lindel\"of theorem that $g(z_1)$ is locally the
unique solution to the differential equation~$f'(z_1)=u(z_1,f(z_1))$
with initial condition $f(p_1)=p_2$, where $p=(p_1,p_2)$.
Hence, $C_{\tilde{a}}\cap W_p=C_a\cap W_p$ for all $\tilde{a}\in A_p$.
Now recall that, by assumption, the family of coupler curves
is $d$-dimensional with $d\geq 2$.
Thus, the family of coupler curves
that pass through the point~$p$ is $(d-1)$-dimensional.
We arrived at a contradiction, as the coupler curves containing~$p$
cannot all be equal in the neighborhood~$W_p$.
This concludes the proof of Claim~1.

It follows from Claim~1 that $U_p\neq A_p\times B_p$ for almost all $p\in \C^2$. Notice that
$U_p\subset A_p\times B_p$ is also not a component of maximal dimension, since $(a,b)\in A_p\times B_p$ was assumed general.
Thus, for almost all $p\in \C^2$ the algebraic set~$U_p\subset A_p\times B_p$ is of codimension 1 so that
\[
\dim U_p\leq m+n-3.
\]
We arrived at a contradiction since $\dim \bigcup_{p\in \C^2}U_p\leq m+n-1$.
This concludes the proof as
the coupler curves~$\tt_{\omega}(\cG)$ and $\tt_{\omega'}(\cG)$
must intersect transversally.
\end{proof}

\begin{lemma}
\label{lem:mult}
Suppose that $\cG$ and $\cG'$ are centric calligraphs.
\begin{Mlist}
\item
The coupler multiplicity and pseudo class of $\cG$ are well-defined.
\item
If $(\cG,\cG')$ is a calligraphic split, then
for almost all the edge length assignments~$\omega\in\Omega_\cG$
and~$\omega'\in\Omega_{\cG'}$, we have
\[
\cc(\cG\cup\cG')=m\cdot m'\cdot |T\cap T'|,
\]
where $T:=\tt_\omega(\cG)$, $m:=\mm(\cG)$, $T':=\tt_{\omega'}(\cG')$ and $m':=\mm(\cG')$.
\end{Mlist}
\end{lemma}

\begin{proof}
Let $\E$, $\V$ and
$V_\cG,H_\lambda\subset\C^{2|\V|}\times\C^{|\E|}$ be defined as in \LEM{dim}\ref{lem:dim:a}.
Let the linear projection~$\phi\c V_\cG\dto \phi(V_\cG)\subset \C^2\times\C^{|\E|}$
send $(x_i,y_i,\l_e)_{i\in\V,e\in\E}$ to $(x_0,y_0,\l_e)_{e\in\E}$.
We suppose that $(\omega(e))_{e\in\E}=(\lambda_e)_{e\in\E}$,
where $\omega\in\Omega_\cG$ is general by assumption.
We know from \LEM{dim}\ref{lem:dim:a} that $C_\lambda:=V_\cG\cap H_\lambda$ is a curve
and it follows from \LEM{dim}\ref{lem:dim:b} that
\[
\phi(C_\lambda)=\set{(p,\lambda)\in \C^2\times\C^{|\E|}}{p\in T}.
\]
Since the projection~$\phi(C_\lambda)$ is a curve,
the restricted map $\phi|_{C_\lambda}$ is generically finite.
From this it follows that $\phi$ itself is generically finite.
As a straightforward consequence of the definitions,
we find that for almost all $p\in T$ the following holds:
\[
m=|\phi^{-1}(p,\lambda)|=|\set{\xi\in \Xi_\cG^\omega}{\xi(0)=p}|<\infty.
\]
We established that the coupler multiplicity~$m$ of~$\cG$ is well-defined
as it does not depend on the
general choices of $\omega\in\Omega_\cG$ and $p\in T$.

Suppose that $[\cG]=m\cdot (a,b,c)$ is the pseudo class of~$\cG$.
As a consequence of \LEM{dim}\ref{lem:dim:c},
the multiplicities $(a,b,c)$
do not depend on the general choice of edge length assignments.
We established that the pseudo class is well-defined.

The number of realizations $\cc(\cG\cup\cG')\in\Z_{>0}$ of the minimally rigid graph $\cG\cup\cG'$
does not depend on the general choices of edge length assignments
by \citep[Theorem~3.6]{rigid-jack-2018} (see alternatively \citep[Corollary~1.11]{rigid-alg}).
We know from \LEM{dim}\ref{lem:dim:d} that the coupler curves~$T$ and~$T'$ are indeed curves.
Each point in $T\cap T'$ corresponds to the realization of the vertex~0 in a realization of~$\cG\cup\cG'$
(see Figures~\ref{fig:min} and~\ref{fig:tm}).
Thus, if $T\cap T'=\{p_1,\ldots,p_k\}$, then
\[
\cc(\cG\cup\cG')=\sum_{1\leq i\leq k}
|\set{\xi\in \Xi_\cG^\omega}{\xi(0)=p_i}|
\cdot
|\set{\xi'\in \Xi_{\cG'}^{\omega'}}{\xi'(0)=p_i}|.
\]
This implies that
$\cc(\cG\cup\cG')=m\cdot m'\cdot |T\cap T'|$.
\end{proof}

\begin{lemma}
\label{lem:glob}
If $\cG$ is a thin calligraph,
then for general~$p\in \tt_\omega(\cG)$
and general edge length assignment~$\omega\in\Omega_\cG$
there exists a unique realization $\xi\in \Xi_\cG^\omega$
\st $\xi(0)=p$. In other words, $|\set{\xi\in \Xi_\cG^\omega}{\xi(0)=p}|=1$.
\end{lemma}

\begin{proof}
Let us assign general coordinates to the vertices of $\cG':=\cG\cup\cL\cup\cR$.
We may assume up to rotations, translations and scalings that vertices 1 and 2 have coordinates
$(0,0)$ and $(1,0)$, \resp. Suppose that $\omega'\in \Omega_{\cG'}$ is the corresponding
edge length assignment.
Since $\cG'$ is thin, it follows that $\cG'$ is 3-vertex connected and
``generically redundantly rigid'', so $\cG'$ is ``generically globally rigid''
by \citep[Corollary~1.7]{rigid-con}.
In other words, $\cG'$ admits
up to rotations, translations and reflections only one realization in the plane
that is compatible with the edge length assignment~$\omega'$.
This implies that $|\Xi_{\cG'}^{\omega'}|=2$, since we do not identify reflections.
We now fix $\xi'\in \Xi_{\cG'}$ to be one of the two realizations.
Let the edge length assignment~$\omega\in \Omega_\cG$
and realization~$\xi\in \Xi_\cG^\omega$ be induced by $\omega'$ and $\xi'$, \resp.
By construction, we have that $p:=\xi(0)$ and $\omega$ are general and thus
satisfy our main hypothesis.
The distances $\delta(\xi(0)-\xi(1))$ and $\delta(\xi(0)-\xi(2))$
with $\delta(x,y):=(x^2+y^2)^{1/2}$
are equal to the edge lengths~$\omega'(\{0,1\})$ and $\omega'(\{0,2\})$, \resp.
Hence, $|\Xi_{\cG'}^{\omega'}|=2$ implies that $|\set{\xi\in \Xi_\cG^\omega}{\xi(0)=p}|=1$.
\end{proof}

\newpage
\begin{proof}[Proof of \PRP{pseudo}.]
Suppose that $\cG$ and $\cG'$ are centric calligraphs
and let the edge length assignments $\omega\in\Omega_\cG$ and $\omega'\in\Omega_{\cG'}$ be general.
We follow the shorthand notation of \COR{class},
where $T:=\tt_\omega(\cG)$, $m:=\mm(\cG)$, $T':=\tt_{\omega'}(\cG')$
and $m':=\mm(\cG')$.
We know from \LEM{dim}\ref{lem:dim:d} that the coupler curves $T$ and $T'$ are indeed curves.
Let the pseudo classes of $\cG$ and $\cG'$ be
\[
[\cG]=m\cdot (a,b,c)
\quad\text{and}\quad
[\cG']=m'\cdot (a',b',c').
\]
We know from \LEM{mult} that these coupler multiplicities and pseudo classes are well-defined.
We proceed to show that axioms and properties for pseudo classes of centric calligraphs.

Recall from \EXM{R} that $\cR$ has pseudo class $(1,0,1)$.
Similarly, we find that $[\cL]=(1,1,0)$ and $[\cC_v]=(2,0,0)$ for all $v\in\Z_{\geq 3}$.
Hence, the pseudo class satisfies Axiom \AXM{1}.

If $\cG\cup\cL$ is not minimally rigid, then we observe that the coupler curve~$T$
of $\cG$ is a circle that is centered around vertex~$1$, which implies
that $[\cG]=m\cdot [\cL]=(m,m,0)$. The analogous statement holds
in case $\cG\cup\cR$ is not minimally rigid, and thus \PPP{3} holds true.

Property \PPP{2} follows from \LEM{glob} and the definition of coupler multiplicity.

Let $\Gamma_\cH$ be the associated series of $\cG$.
Since $\Gamma_\cH$ is centric by assumption,
almost all curves in $\Gamma_\cH$ meet the
line at infinity at only the complex conjugate cyclic points with
multiplicity~$a$ and thus it follows from
B\'ezout's theorem that a curve in $\Gamma_\cH$ has
degree $2a$ so that \PPP{5} is fulfilled.

Suppose that $Y$ is the blowup of $\P^2$ at the cyclic points
and the infinitely near 1-centric and 2-centric points.
We remark that after blowing up $\P^2$ at the cyclic
points~$u$ and~$\overline{u}$,
the infinitely near points lie on the
two complex conjugate exceptional curves that contract to $u$ and $\overline{u}$, respectively (see \RMK{blowup}).
Let $v$, $\overline{v}$ and $w$, $\overline{w}$ denote the pairs of 1-centric and 2-centric points, \resp.
Thus the centers of blowup consist of three pairs of complex conjugate points.
Therefore, $Y$ admits a real structure,
namely an antiholomorphic involution~$\sigma_Y\c Y\to Y$ (see \RMK{real}).

In the following we use the algebro geometric concepts of divisor classes
and canonical classes, which should not be confused with classes or pseudo classes
of calligraphs.

Let $\Theta:=\{u,{\overline{u}},v,{\overline{v}},w,{\overline{w}}\}$
and for $p\in \Theta$ let $\e_p$
denote the divisor class of the exceptional curve that is centered at~$p$.
We denote by $\e_0$ the divisor class of the pullback of a general line in $\P^2$.
The N\'eron-Severi lattice of $Y$ is generated by the group lattice
$
N(Y)=\bas{\e_0,\e_p:p\in\Theta}_\Z,
$
where the non-zero intersection products between the generators are
$\e_0^2=-\e_p^2=1$ for $p\in \Theta$ (see \citep[Proposition~V.3.2]{har}).
The real structure $\sigma_Y$ induces a unimodular involution $\sigma_*\c N(Y)\to N(Y)$
\st $\sigma_*(\e_0)=\e_0$ and $\sigma_*(\e_p)=\e_{\overline{p}}$ for $p\in\Theta$
(see \citep[chapter~I]{sil}).
Therefore, the real group lattice is defined as
\[
N_\R(Y)=\bas{\e_0,\e_1,\e_2,\e_3}_\Z
~~\text{where}~~
\e_1:=\e_c+\e_{\overline{c}},\quad
\e_2:=\e_u+\e_{\overline{u}}
~~\text{and}~~
\e_3:=\e_v+\e_{\overline{v}}.
\]
The only nonzero intersection products between the generators are
\[
\e_0^2=1
\quad\text{and}\quad
\e_1^2=\e_2^2=\e_3^2=-2.
\]
The canonical class of $Y$ is as follows
(see \citep[Example~II.8.20.3 and Proposition~V.3.3]{har}):
\[
\kappa:=-3\,\e_0+\e_1+\e_2+\e_3.
\]
If $C\subset\C^2$ is a real curve, then
its class~$[C]\in N_\R(Y)$ is defined as the divisor class of the strict transform of
$\gamma_0^*(C)\subset\P^2$ along the blowup map~$Y\to\P^2$.
Notice that the coupler curve~$T$ is real, if the edge length assignment~$\omega$ is real.
We have $\deg T=2 a=[T]\cdot\e_0$ by~\PPP{5} so that
\[
[T]:=2a\,\e_0-a\,\e_1-b\,\e_2-c\,\e_3
\quad\text{and}\quad
[T']:=2a'\,\e_0-a'\,\e_1-b'\,\e_2-c'\,\e_3.
\]
Let $\tT$ be the strict transform of $\gamma_0^*(T)\subset\P^2$ to $Y$
along the blowup map $Y\to\P^2$.
By construction, $\e_1-\e_2$ and $\e_1-\e_3$ are divisor classes
of the exceptional curves that are contracted to points by the blowup map $Y\to\P^2$.
Since these curves are not components of~$\tT$,
we require that $[T]\cdot (\e_1-\e_2)\geq0$ and $[T]\cdot (\e_1-\e_3)\geq 0$,
which implies that $a\geq b\geq0$ and $a\geq c\geq 0$ so that \PPP{1} holds.

In order to verify \PPP{6}, let us consider the coupler decomposition~$(T_1,\ldots, T_n)$
of~$T$ and let us denote the corresponding irreducible components of the strict transform~$\tT$
to~$Y$ by $(\tT_1,\ldots,\tT_n)$.
We set $\alpha_i:=[T_i]$ for all $1\leq i\leq n$
so that $\alpha_1+\ldots+\alpha_n=[T]=\frac{1}{m}\cdot [\cG]$.
If $\alpha_i=(\alpha_{i0},\alpha_{i1},\alpha_{i2})$, then
$\alpha_{i0}\geq \alpha_{i1}\geq0$ and $\alpha_{i0}\geq \alpha_{i2}\geq 0$
by the same argument used for proving \PPP{1}.
Using the same argument for proving \PPP{5} we find that
\[
\deg T_i=2\,\alpha_{i0}=\alpha_i\cdot(1,0,0).
\]
It follows from the genus formula
(see \citep[Proposition~V.1.5 and Examples~V.3.9.2 and V.3.9.3]{har}) that:
\[
g(T_i)=g(\tT_i)=\tfrac{1}{2}([T_i]^2+[T_i]\cdot \kappa)+1-\sum_{p\in \sng\tT_i}\delta_p(\tT_i),
\]
where the \df{delta invariant} $\delta_p(\tT_i)$ of a singular point $p\in\tT_i$
is at least one so that
\[
|\sng T_i|\leq \sum_{p\in \sng \tT_i}\delta_p(\tT_i).
\]
It is now a straightforward to verify that \PPP{6} is satisfied.

If $(\cG,\cG')$ is a calligraphic split, then
$\cc(\cG\cup\cG')=m\cdot m'\cdot |T\cap T'|$ by \LEM{mult}.
It follows from \LEM{trans} that the curves~$T$ and~$T'$ intersect
transversally and thus we know from \citep[Theorem~V.1.1]{har} that
\[
|T\cap T'|=[T]\cdot [T']=4aa'-2aa'-2bb'-2cc'=\frac{1}{m\cdot m'}\cdot[\cG]\cdot [\cG'].
\]
This concludes the proof for \PPP{4}
and as a consequence Axiom~\AXM{2} holds as well.

We established that all axioms and properties for pseudo classes are true
under the assumption that calligraphs are centric.
It follows from \ALG{class} and the discussion in \RMK{alg},
that if the class~$[\cG]$ of the calligraph~$\cG$ exist,
then by \AXM{1}, \AXM{2} and \PPP{3} it can be uniquely recovered from
the number of realizations $\cc(\cG\cup\cL)$, $\cc(\cG\cup\cR)$ and $\cc(\cG\cup\cC_v)$.
Therefore, the class for calligraphs is unique and
we conclude the proof of \PRP{pseudo}.
\end{proof}

The following lemma is only needed in \SEC{op} and its 8 conditions correspond to the 8 conditions of \PRP{cases}.
However, now is the right time to understand its statement.

\begin{lemma}
\label{lem:centric}
If $\fH\subset\C[z_0,z_1,z_2]$ is a subset of homogeneous polynomials of the same degree with
real coefficients, then its associated series~$\Gamma_\fH$ is centric
if the following 8 conditions hold:
\begin{enumerate}[topsep=0mm,itemsep=0mm,leftmargin=20pt]
\item\label{lem:centric:V} We have $V(f)\cap V(z_0)=\{(0:1:\ii),(0:1:-\ii)\}$ for a general $f\in\fH$.
\item The set $\gamma_0(\fH)$ has no base points.
\item The set $\gamma_1(\fH)$ has an $m$-fold base point at $p:=(0,\ii)$ for some $m>0$.
\item The set $\gamma_2(\fH)$ does not have $(0,0)$ as a base point.
\item If $\alpha_p^m\circ\gamma_1(\fH)$ has an $n$-fold base point~$q$ for some $n>0$, then $q\in\{(0,0),(0,-\ii)\}$.
\item The set $\beta_p^m\circ\gamma_1(\fH)$ does not have $(0,0)$ as a base point.
\item If $q$ is an $n$-fold base point, then $\alpha_q^n\circ\alpha_p^m\circ\gamma_1(\fH)$ has no base points.
\item If $q$ is an $n$-fold base point, then $(0,0)$ is not a base point of $\beta_q^n\circ\alpha_p^m\circ\gamma_1(\fH)$.
\end{enumerate}
\end{lemma}

\begin{proof}
Since $\fH$ consist of polynomials with real coefficients, the multiplicities at conjugate base points are equal, so that
we only need to check~$p:=(0,\ii)$. The remaining statements are a straightforward
consequence of the definitions and left to the reader.
It may be instructive to verify that if $\fH$ is as defined as in \EXM{hbp},
then each of the eight conditions hold.
\end{proof}

We now proceed with showing that the hypothesis of \PRP{pseudo} holds,
namely that indeed all calligraphs are centric.

\section{Sufficient conditions for centricity of calligraphs}
\label{sec:op}

In \LEM{centric} of the previous section, we gave eight sufficient conditions for
a calligraph to be centric. We conclude this section with equivalent conditions
by instead performing operators on the quadratic polynomials that are associated to edges.
Since we know the quadratic polynomials,
the resulting conditions can be shown to be valid for every calligraph.

The methods of this and the remaining sections should be accessible to the non-expert,
but are rather technical due to coordinate dependence.
For this reason let us start with giving an informal explanation.
If $\cG$ is a calligraph with associated series $\Gamma_\fH$,
then in \LEM{centric} we gave sufficient conditions for~$\cG$ to be centric
in terms of~$\fH$.
In particular, it is required that the base points of~$\alpha_p^m\circ\gamma_1(\fH)$
are either $(0,0)$ or $(0,-\ii)$ for some $m>0$ with $p=(0,\ii)$.
However, $\fH$ is not directly accessible from the combinatorial data of~$\cG$
and thus will not lead to a proof which states that all calligraphs are centric.
Instead, we assign to each edge $e=\{i,j\}$ of the calligraph~$\cG$ a quadratic polynomial
$(x_i-x_j)^2+(y_i-y_j)^2-\ell_e^2$, where $(x_i,y_i)$ is the coordinate of
a realization of vertex~$i$ and $\ell_e$ is the edge length of $e$.
The polynomial representing~$\fH$ is obtained by first homogenizing and then eliminating
from the set of quadratic polynomials the variables $x_i$ and $y_i$ for all vertices~$i>0$, except vertex~$0$.
Geometrically, this means that the coupler
curves of~$\cG$ are linear projections of curves defined by the zero set
of quadratic polynomials.
We introduce operators such as $H$, $M$ and $T$, that take as input a finite set of polynomials
and output the set of polynomials that are the result of substitutions and removing certain factors.
We apply compositions of these maps to the set of quadratic polynomials and show
in particular that the above condition for~$\alpha_p^m\circ\gamma_1(\fH)$
is equivalent to the condition~$B \circ G\circ T\circ M\circ H\circ\mu(\E)\subseteq\{(0,0),~(0,-\ii)\}$,
where $\mu(\E)$ is the set of quadratic polynomials assigned to edges~$\E$
and the expression~$B\circ G$ informally means ``the set of base points after elimination''.
This will be the fifth of the eight conditions at \PRP{cases}.
Let us now proceed with formally defining the notation.

Suppose that $\cG$ is a calligraph with $\V:=\vv(\cG)\setminus\{1,2\}$
and $\E:=\ee(\cG)\setminus\{\{1,2\}\}$.
We define $R:=\C[x_i,y_i,\l_e: i\in \V,~e\in \E]$ and
$S:=\C[x_0,y_0,\l_e:e\in \E]$ to be polynomial rings.
We denote the set of variables of these rings by
\[
\gR := \set{x_i,y_i,\l_e}{i\in \V,~e\in \E}
\qquad\text{and}\qquad
\gS := \set{x_0,y_0,\l_e}{e\in \E}.
\]
The map $\mu\c \E\to R$ is defined as
\[
\mu(e):=(x_i-x_j)^2+(y_i-y_j)^2-\l_e^2,
\]
where $e=\{i,j\}$, $(x_1,y_1):=(0,0)$ and $(x_2,y_2):=(1,0)$.

Let $f\in R$ be a polynomial.
We denote by $f\sub{a_0\to b_0,\ldots}$
the result of substituting the variable $a_i\in \gR$ in $f$
with the rational function $b_i\in \C(\gR\cup\{z_0\})$ for all~$i$.
The variable~$z_0$ will serve a homogenization variable for polynomials
that belong to~$R$.

\begin{definition}
\label{def:ophat}
We define the following \df{operators} $R\to R$,
where $f\in R$, $r\in \gR$, $c\in \C$,
$\deg f$ is the degree of $f$ when considered as a polynomial in~$x_0$, $y_0$,
and $\delta(r,f)$
is the maximal integer \st $r^{\delta(r,f)}$ is a monomial factor of~$f$:
\[
\begin{array}{rcl}
\oH_r(f)     & := & \left(z_0^{\deg f}\cdot f
                \sub{x_0\to \frac{x_0}{z_0},~y_0\to \frac{y_0}{z_0}}\right)
                \sub{r\to 1, z_0\to r},\\
\oF_r(f)   & := & r^{-{\delta(r,f)}}\cdot f,\\
\oM_c(f)   & := & f\sub{y_0\to x_0y_0 + c},\\
\oN_c(f)   & := & f\sub{y_0\to y_0 + c,~ x_0\to x_0y_0}.\\
\end{array}
\]
We shall denote
$\oH_{x_0}$, $\oM_\ii$, $\oN_\ii$
by
$\oH$, $\oM$, $\oN$, \resp.
\END
\end{definition}

The polynomial $\oH_r(f)$ is obtained by first homogenizing a polynomial $f$
and then dehomogenizing with respect to a variable determined by~$r$.
This was used in \SEC{bp} to determine base points at infinity
for algebraic series associated to calligraphs.
In order to determine the infinitely near base points,
we apply the compositions $\oF_r\circ\oM_c$ and $\oF_r\circ\oN_c$,
where $\oF_r$ removes a monomial factor (see \RMK{blowup}
for the algebro geometric interpretation of these factors).
We make this more explicit in the following example.

\begin{example}
\label{exm:ophat}
Suppose that $\cG$ is defined by the calligraph $\cR$ in \FIG{cal} with associated series~$\Gamma_\fH$.
Let $p:=(0,\ii)$ and $e:=\{0,2\}$.
We identify $x$ and $y$ in \SEC{bp} with $x_0$ and $y_0$, \resp.
Recall from \EXM{R}, together with \EXMS{exm:bp,exm:hbp}, that
\begin{gather*}
\gamma_i(\fH)=\set{f_i\sub{\ell_e\to\lambda}}{\lambda\in\C},\\
\alpha_p^1\circ \gamma_1(\fH)=\set{g\sub{\ell_e\to\lambda}}{\lambda\in\C},\qquad
\beta_p^1\circ \gamma_1(\fH)=\set{h\sub{\ell_e\to\lambda}}{\lambda\in\C},
\end{gather*}
where
\begin{gather*}
f_0:=(x_0-1)^2+y_0^2-\ell_e^2,\qquad
f_1:=(1-x_0)^2+y_0^2-\ell_e^2\,x_0^2,
\\
f_2:=(x_0-y_0)^2+1-\ell_e^2\,x_0^2,
\\
g:=-2 + x_0 + 2\ii\,y_0 + x_0\,y_0^2 - \l_e^2\,x_0
\quad\text{and}\quad
h:=2\ii - 2x_0 + y_0 + x_0^2y_0 - \l_e^2\,x_0^2\,y_0.
\end{gather*}
We now verify that
\[
\mu(e)=f_0,\quad
\oH(f_0)=f_1,\quad
\oH_{y_0}(f_0)=f_2,\quad
\oF_{x_0}\circ\oM\circ\oH(f_0)=g,\quad
\oF_{y_0}\circ\oN\circ\oH(f_0)=h.
\]
In particular, $\oH(f_0)$ is equal to $f_1=f_{\text{hom}}(x_0,1,y_0)$,
where
\[
f_{\text{hom}}\in \C[z_0,z_1,z_2,\l_e:e\in\E]
\]
is the homogenization of~$f_0$
with respect to the variables~$x_0$ and~$y_0$.
Indeed, this corresponds to $\gamma_1(\fH)$.
Similarly, $\oM(f_1)=\oM\circ\oH(f_0)=f_1(x_0,x_0y_0+\ii)$.
Since $p$ is a base point of $\gamma_1(\fH)$ with multiplicity $m=1$,
we find that $f_1(x_0,x_0y_0+\ii)$ has $x_0^m$ as a factor.
We remove this monomial factor using the operator~$\oF_{x_0}$
so that $\oF_{x_0}\circ\oM\circ\oH(f_0)$ corresponds to $\alpha_p^1\circ \gamma_1(\fH)$.
Notice that we can choose for~$m$ in \DEF{ab} the maximal integer
\st the remainders of the polynomial quotients vanish identically.
The reason is that the polynomial~$f_0$ has no $(x_0)$-factor or $(y_0)$-factor,
as this would mean that two general coupler curves of~$\cG$ have a common component.
\END
\end{example}

We recover the base points of coupler curves by
eliminating certain variables from an ideal generated by quadratic polynomials.
For that purpose we consider the lexicographic monomial ordering on $R$ that is induced by the following ordering on the
variables
for all $i,j\in \V$ and $\{a,b\},\{a',b'\}\in\E$
\st
$i<j$, $a>b$, $a'>b'$ and either ($a<a'$) or ($a=a'$ and $b<b'$):
\[
\l_{\{a,b\}}<\l_{\{a',b'\}},\quad \l_{\{a,b\}}<x_i,\quad x_i<x_j,\quad x_i<y_i,\quad y_i<x_j\quad\text{and}\quad y_i<y_j.
\]
If $\A=\{z_0,\ldots,z_n\}$ and $\B=\{z_0,\ldots,z_m\}$
are ordered sets of variables \st $\B\subset \A$ with an ordering $z_0<\ldots<z_n$,
then we denote by $\A\wr\B$ the linear projection
$\C^{|\A|}\to\C^{|\B|}$ that sends
$(z_0,\ldots,z_n)$ to $(z_0,\ldots, z_m)$.
We set $\pi:=\Upsilon_R\wr\Upsilon_S$.

\begin{example}
If $\cG$ is equal to the calligraph $\cC_3$,
then
\[
\V=\{0,3\},~~~~
\E=\{\{3,0\},\{3,1\},\{3,2\}\},~~~~
R=\C[x_0,y_0,x_3,y_3,\l_{30}, \l_{31}, \l_{32}]
\]
and $\l_{30}<\l_{31}<\l_{32}<x_0<y_0<x_3<y_3$.
In this case $\pi=\gR\wr\gS$ is the map $\C^7\to\C^5$
that sends $(\l_{30},\l_{31},\l_{32},x_0,y_0,x_3,y_3)$
to $(\l_{30},\l_{31},\l_{32},x_0,y_0)$.
\END
\end{example}

Let $\wp(R)$ denote the set of all finite subsets of~$R$
and let $\bas{P}$ denote the \df{ideal} in~$R$ generated by~$P\in\wp(R)$.
The map~$\oG\c\wp(R)\to\wp(R)$ assigns to~$P$
the reduced Gr\"obner basis~$\oG(P)$ for the
elimination ideal $\bas{P}\cap S$
\wrt the above lexicographic ordering on $R$.
The map $V\c\wp(R)\to \operatorname{powerset}\left(\C^{|\gR|}\right)$ assigns to~$P$ its \df{zero set}~$V(P)$.
We shall denote $V(\{f_1,\ldots, f_r\})$ by $V(f_1,\ldots,f_r)$.

\begin{notation}
\label{ntn:o}
By abuse of notation we consider the operator $\oH\c R\to R$
also as the map $\oH\c \wp(R)\to\wp(R)$ so that $\oH(P)=\set{\oH(p)}{p\in P}$;
we use the same notation for the other operators in \DEF{ophat}
and thus the composition of~$\oG$ with an operator is defined.
Similarly, we consider
$\mu\c \E\to R$,
as the map
$\mu\c \wp(\E)\to \wp(R)$
so that for example the composition
$\oH\circ\mu$ is defined.
We shall denote by
$\pi\circ V(P)$
the Zariski closure of
$\set{\pi(p)}{p\in V(P)})$, where we
recall that $\pi$ denotes the linear projection~$\gR\wr\gS$.
\END
\end{notation}

We now introduce a map, which recovers base points from~$\oG(P)$.
For all polynomials~$f\in S$ there
exist polynomials $c_\alpha\in\C[x_0,y_0]$ \st
\[
f=
\sum_{\alpha\in\Z_{\geq 0}^\E}c_{\alpha}(x_0,y_0)
\prod_{e\in \E}\l_e^{\alpha_e}.
\]
Notice that $\alpha\c \E\to \Z_{\geq 0}$ is a map and that $\alpha_e=\alpha(e)$ is its evaluation at~$e$.
We call $p$ a \df{base point of} $f$ if $c_\alpha(p)=0$ for all $\alpha\in\Z_{\geq 0}^\E$.
The map $B_t\c\wp(R)\to \operatorname{powerset}(\C^2)$ with $t\in \gS\cup\{0\}$ is defined as
\[
B_t(P):=\set{p\in\C^2\cap V(t)}{p \text{ is a base point of some } f\in P\cap S}.
\]
We denote $B_{x_0}$ by $B$.
Notice that in the definition of~$B_0$ we have that $\C^2\cap V(0)=\C^2$.

\begin{remark}
\label{rmk:B}
If $p\in\C^2$ is a base point of a polynomial $f\in S$, then it is a base point
as defined in \SEC{bp} for the following subset:
\[
\set{f\sub{\l_e\to \lambda_e:e\in \E}}{\lambda\in\C^\E}\subset\C[x_0,y_0],
\]
where $\C^\E$ defines the set of maps with domain~$\E$ and codomain~$\C$.
Geometrically this means that
the base point $p$ is
for all edge length assignments~$\lambda\in\C^\E$
contained in the coupler curve defined by~$V(f\sub{\l_e\to \lambda_e:e\in \E})$.
\END
\end{remark}

\begin{example}
\label{exm:C3}
Suppose that $\cG$ is defined by the calligraph $\cC_3$.
We show that $\cG$ is centric.
First notice that
\[
\mu(\E)=\{(x_0-x_3)^2+(y_0-y_3)^2-\l_{30}^2,\quad
x_3^2+y_3^3-\l_{31}^2,\quad
(x_3-1)^2+y_3^2-\l_{32}^2
\}.
\]
We have $\oG\circ\mu(\E):=\{f\}$, where we consider
$f\in S$ as a polynomial in $x_0$ and $y_0$ with coefficients
depending on the $\{\l_e\}_{e\in \E}$.
Thus, for each choice of edge length assignment~$\lambda\in\C^{\E}$,
the zero set of $f\sub{\l_e\to \lambda_e: e\in\E}$ defines a coupler curve of~$\cG$.
We can factor the polynomial~$f$ as follows:
\[
f=\bigl((x_0-c_x)^2+ (y_0-c_y)^2-\l_{30}^2\bigr)\cdot
\bigl((x_0-c_x)^2+ (y_0+c_y)^2-\l_{30}^2\bigr),
\]
where $c_y:=\left(\l_{31}^2-c_x^2\right)^{\frac{1}{2}}$
and $c_x:=\frac{1}{2}\left(\l_{31}^2-\l_{32}^2+1\right)$.
We verify using the Pythagorean theorem and the law of cosines
that indeed $(c_x,\pm c_y)$ are the centers of the two circles
as depicted in \FIG{intro}.
Similarly to \EXM{ophat},
the set~$\gamma_1(\fH)$ corresponds to~$\oH\circ\oG\circ\mu(\E)=\{f_1\}$, where
\[
f_1:=\bigl((1-c_x x_0)^2+ (y_0-c_y x_0)^2-\l_{30}^2x_0^2\bigr)\cdot
\bigl((1-c_x x_0)^2+ (y_0+c_y x_0)^2-\l_{30}^2x_0^2\bigr).
\]
The 1-fold base points of~$\gamma_1(\fH)$ are $p:=(0,\ii)$ and $\overline{p}=(0,-\ii)$
and indeed we find that $B\circ\oH\circ\oG\circ\mu(\E)=B(\{f_1\})=\{p,\overline{p}\}$.
It is straightforward to verify also the relations in \TAB{C3} for this example.
\begin{table}[!ht]
\caption{See \EXM{C3}. We denote by
$\bp(\fF)$ the base points of the subset $\fF\subset \C[x,y]$.}
\label{tab:C3}
\centering
\vspace{-7mm}
$
\begin{array}{rrcl}
                                   & (\oH\circ \oG\circ\mu(\E))\sub{x_0\to 0} &\supseteq &\{1+y_0^2\}, \\
\bp(\gamma_0(\fH))                =& B_0\circ \oG\circ\mu(\E)                                  & =          & \varnothing,           \\
\bp(\gamma_1(\fH))                =& B\circ \oH\circ \oG\circ\mu(\E)                           & =          & \{(0,\ii),(0,-\ii)\},\\
\bp(\gamma_2(\fH))                =& B\circ \oH_{y_0}\circ \oG\circ\mu(\E)                     & \nsupseteq & \{(0,0)\},           \\
\bp(\alpha_p^1\circ\gamma_1(\fH)) =& B\circ \oF_{x_0}\circ \oM\circ \oH\circ \oG\circ\mu(\E)   & =          & \varnothing,           \\
\bp(\beta_p^1\circ\gamma_1(\fH))  =& B\circ \oF_{y_0}\circ \oN\circ \oH\circ \oG\circ\mu(\E)   & \nsupseteq & \{(0,0)\}.           \\
\end{array}
$
\end{table}
Therefore, it follows from \LEM{centric} that $\cG$ is indeed centric.
Notice that the first condition in~\TAB{C3} is equivalent to Condition~\ref{lem:centric:V} in \LEM{centric}
and that it follows from the fact that $(\oH\circ\mu(\E))\sub{x_0\to 0}\supseteq \{1+y_0^2\}$.
Hence, the first condition can be shown without computing a Gr\"obner basis.
\END
\end{example}

The approach in \EXM{C3} does not directly lead to a general proof that all calligraphs are centric,
as for most calligraphs it is not feasible to compute the Gr\"obner basis~$\oG\circ \mu(\E)$.
To overcome this obstacle we shall modify the operators so that instead
of for example $B\circ \oF_{x_0}\circ \oM\circ \oH\circ \oG\circ\mu(\E)$
we consider $B\circ\oG(P)$, where $P$ is obtained from $\oF_{x_0}\circ \oM\circ \oH\circ\mu(\E)$
using substitutions and removing factors.
In other words, we want the operators $\oH$, $\oM$, $\oN$
to commute with the map~$\oG$ up to removing some factors.
This is the content of \LEM{op} below, but first we account for these factors
in the following definition.

\begin{definition}
\label{def:op}
We follow \NTN{o} and consider the following compositions of
operators in \DEF{ophat} as either $R\to R$ or $\wp(R)\to\wp(R)$:
\[
\begin{array}{rl@{\hspace{1cm}}rl}
G:=   & \oF_{y_0}\circ \oF_{x_0} \circ \oG,  & H_r:= & \oF_{y_0}\circ \oF_{x_0} \circ \oH_r, \\
M_c:= & \oF_{y_0}\circ \oF_{x_0} \circ \oM_c, &
N_c(\cdot):= & (\oF_{y_0}\circ \oF_{x_0} \circ \oN_c(\cdot))\sub{x_0\to y_0, y_0\to x_0},
\end{array}
\]
where $r\in\{x_0,y_0\}$ and $c\in \C$.
We shall denote $H_{x_0}$, $M_\ii$, $N_\ii$ by
$H$, $M$, $N$, \resp.
In addition, we define the operator~$T\c \wp(R)\to\wp(R)$ as
\[
T(P):=\set{f\sub{x_0\to x_0\cdot s,~\l_e\to \l_e \cdot s: e\in \E}}{f\in P},
\]
where $s:=x_0\prod_{u\in U}(y_0-u)$ and $U:=\set{u}{(0,u)\in B\circ G(P)}$.
\END
\end{definition}

Notice that in the operator $N_c$ we interchange $x_0$ and $y_0$ so that always the $x_0$ coordinate of infinitely near base points vanishes identically.
The motivations for the new operator $T$ are clarified in \LEM{T} and \EXM{H}.
Before stating \LEM{op} we first assert in \LEM{G}
that the elimination of variables in an ideal corresponds geometrically to
a linear projection.

\begin{lemma}
\label{lem:G}
Suppose that $\A$ is a set of variables and $\B\subset\A$.
If $\kappa=\A\wr \B$
and $Z\subset \C^{|\A|}$ is a variety with ideal $\bas{P}\subset\C[\A]$,
then the ideal of $\kappa(Z)$ is $\bas{P}\cap \C[\B]$.
In particular, for all $P\in\wp(R)$ we have
\[
V\circ \oG(P)=\pi\circ V(P).
\]
\end{lemma}

\begin{proof}
See \citep[Theorem~3 in Section~3.2]{cox}.
\end{proof}

\begin{lemma}
\label{lem:op}
If $P,Q\in\wp(R)$ \st $|G(P)|=|G(Q)|=1$ and $V(P)=V(Q)$, then
for all $c\in\C$, $r\in\{x_0,y_0\}$ and $s\in S$ we have
\begin{gather*}
\begin{array}{r@{\,}l@{\hspace{1cm}}r@{\,}l}
V\circ M_c\circ G(P) & = V\circ G\circ M_c(P), & V\circ H_r \circ G(P)  & = V\circ G\circ H_r(P),\\
V\circ N_c\circ G(P) & = V\circ G\circ N_c(P), & V\circ T \circ G(P) & = V\circ G\circ T(P),
\end{array}
\end{gather*}
and $B\circ G(P)=B\circ G(Q)$.
\end{lemma}

\begin{proof}
Recall that $\pi=\gR\wr\gS$
and let $m=|\gR|-|\gS|$ and $n=|\gS|$ \st $\pi\c \C^{|\gR|}\to \C^{|\gS|}$
can be restated as the projection
$\pi\c \C^m\times \C^n\to \C^n$ to the second component.
Let $\eta\c \C^m\times \C^n\dto \C^m\times \C^n$
be a map of the form $\id\times \nu$
for some rational map $\nu\c \C^n\dto \C^n$,
where
$\id\c \C^m\to \C^m$ is the identity map
so that the following diagram is commutative.
\begin{center}
  \begin{tikzpicture}[map/.style={->,dashed},proj/.style={->}]
    \node[] (a) at (0,0) {$\C^m\times \C^n$};
    \node[] (b) at (3,0) {$\C^m\times \C^n$};
    \node[] (c) at (0,-1.5) {$\C^n$};
    \node[] (d) at (3,-1.5) {$\C^n$};
    \draw[map] (a) to node[labelsty,above] {$\eta$} (b);
    \draw[proj] (a) to node[labelsty,left] {$\pi$} (c);
    \draw[proj] (b) to node[labelsty,right] {$\pi$} (d);
    \draw[map] (c) to node[labelsty,below] {$\nu$} (d);
    \node[] at ($(a)!0.5!(d)$) {$\circlearrowleft$};
  \end{tikzpicture}
\end{center}
\vspace{-5mm}
Same as $\pi\circ V(\cdot)$ in \NTN{o}, we let $\eta\circ V(\cdot)$
and $\nu\circ V(\cdot)$ denote the
Zariski closures of~$\set{\eta(p)}{p\in V(\cdot)}$ and $\set{\nu(p)}{p\in V(\cdot)}$, \resp.
Since $\eta$ sends a fiber of $\pi$ to a fiber
we find that
\begin{equation}
\label{eqn:1}
\nu\circ \pi \circ V(P)=\pi\circ \eta\circ V(Q).
\end{equation}
We know from \LEM{G} that
\begin{equation}
\label{eqn:2}
V\circ \oG=\pi\circ V.
\end{equation}
Suppose that $\oO\in\{\oM_c,\oN_c, T\}$
and
notice that if $\eta$ is birational with a polynomial inverse,
then $\eta\circ V(A)=V(\set{f\circ\eta^{-1}}{f\in A})$ for all $A\in \wp(R)$.
Thus, there exist birational maps $\nu_{_\oO}$ and $\eta_{_\oO}$ \st $\eta_{_\oO}=\id\times \nu_{_\oO}$
and
\begin{equation}
\label{eqn:3}
V\circ\oO=\eta_{_\oO}\circ V
\quad\text{and}\quad
V\circ\oO\circ\oG=\nu_{_\oO}\circ V\circ \oG
.
\end{equation}
For example, if $\oO=\oM_c$, then $\nu_{_\oO}^{-1}$ and $\nu_{_\oO}$
maps $(x_0,y_0,\ldots)$ to $(x_0,y_0x_0+c,\ldots)$ and $(x_0,(y_0-c)x_0^{-1},\ldots)$, \resp.
We apply \EQN{1}, \EQN{2} and \EQN{3} for all $\oO\in\{\oM_c,\oN_c,T\}$
and obtain the following sequence of equalities:
\begin{multline*}
V\circ\oO\circ\oG(P)
=\nu_{_\oO}\circ V\circ\oG(P)
=\nu_{_\oO}\circ \pi\circ V(P)
\\
=\pi\circ \eta_{_\oO}\circ V(Q)
=\pi\circ V\circ\oO(Q)
= V\circ\oG\circ\oO(Q)
.
\end{multline*}
We set $O:=\oF_{y_0}\circ\oF_{x_0}\circ\oO$ and $D:=V(x_0y_0)$, and we observe that
\begin{align}
V\circ G(\cdot)\setminus D  &= \pi\circ V(\cdot)\setminus D,\nonumber\\
V\circ O(\cdot)\setminus D  &= \eta_{_\oO}\circ V(\cdot)\setminus D.
\label{eqn:4}
\end{align}
We apply \EQN{1} and the identities at \EQN{4} and deduce that
\begin{multline*}
V\circ O\circ G(P)\setminus D
=\nu_{_\oO}\circ V\circ G(P)\setminus D
=\nu_{_\oO}\circ \pi\circ V(P)\setminus D
\\
=\pi\circ\eta_{_\oO}\circ V(Q)\setminus D
=\pi\circ V\circ O(Q)\setminus D
= V\circ G\circ O(Q)\setminus D
.
\end{multline*}
But this implies that for all $O\in\{M_c, N_c\}$ we have
\[
V\circ O\circ G(P)=V\circ G\circ O(Q).
\]
\newpage
Let us now consider the case where $\oO=\oH_r$ and $O= H_r$.
Recall the embeddings $\gamma_i^*\c \C^2\hookrightarrow\P^2$ for $i\in\{0,1,2\}$ as defined in \SEC{bp}
and let $\eta_r:=\id\times\nu_r$ with $\nu_r:=(\gamma_i^*)^{-1}\circ\gamma_0^*$ and
$(r,i)\in\{(x_0,1),(y_0,2)\}$.
For example, if $r=x_0$, then both $\nu_r$ and $\nu_r^{-1}$ send $(x_0,y_0,\ldots)$
to $(x_0^{-1},y_0x_0^{-1},\ldots)$.
If $r\neq 0$, then $\eta_r\circ V(A)=V(\set{r^{\deg f}\cdot (f\circ\eta_r^{-1})}{f\in A})$
for all $A\in \wp(R)$.
This implies that $V\circ\oH_r(\cdot)\setminus D=\eta_r\circ V(\cdot)\setminus D$ and thus
\[
V\circ H_r(\cdot)\setminus D=\eta_r\circ V(\cdot)\setminus D.
\]
Applying the same arguments as before
we confirm the following assertion
\[
V\circ H_r \circ G(P)=V\circ G\circ H_r(Q).
\]
Since $G(P)=\{f\}$ and a base point of $f\in S$
is expressed in \RMK{B} in terms of zero sets,
we find that $V(P)=V(Q)$
implies that $B\circ G(P)=B\circ G(Q)$.
We verified all assertions and concluded the proof.
\end{proof}

The following lemma shows that
the map $T$ modifies a set of polynomials
\st the base point candidates associated to the input remain candidates for the output.
The motivation for this map is that only after
the transformation we can characterize the base points
for the cases we need to consider.
We clarify this in more detail in \EXM{H}.

\begin{lemma}
\label{lem:T}
If $|G(P)|=1$, then
\[
B\circ G(P)\subseteq B\circ G\circ T(P)
\]
and
$\oG\circ T(P)=\{g\cdot s\}$ for some $g\in S$,
where $s$ is as defined in \DEF{op}.
\end{lemma}

\begin{proof}
Let $G(P)=\{f\}$, where
\[
f=
\sum_{\alpha\in\Z_{\geq 0}^\E}c_{\alpha}(x_0,y_0)
\prod_{e\in \E}\l_e^{\alpha_e}.
\]
We know from \LEM{op} that
$\bas{T\circ G(P)}=\bas{G\circ T(P)}$
and thus there exists a monomial $\lambda:=c\cdot x_0^a\cdot y_0^b$
with $c\in \C\setminus\{0\}$ and $a,b\in\Z_{\geq0}$ \st
\[
\oG\circ T(P)=\left\{\lambda
\sum_{\alpha\in\Z_{\geq 0}^\E}
c_{\alpha}(x_0\cdot s,y_0)
\prod_{e\in \E}(\l_e\cdot s)^{\alpha_e}\right\}.
\]
We notice that $c_0(x_0\cdot s,y_0)=h\cdot s$ for some polynomial~$h\in S$,
since $c_0(0,u)=0$ and $s\sub{y_0\to u}=0$ for all $(0,u)\in B(\{f\})$.
From this we deduce that $\oG\circ T(P)=\{g\cdot s\}$ for some $g\in S$.
The assertion
$B\circ G(P)\subseteq B\circ G\circ T(P)$
is now a direct consequence of the definitions.
\end{proof}

The purpose of the following example is to clarify the definitions and lemmas,
and to prepare the reader for the proof strategy in the remaining text.
We shall use \TAB{eqn} which is obtained via straightforward calculations.

\begin{table}[!ht]
\caption{
We list polynomials that
are obtained after applying certain compositions of operators to~$\mu(e)$,
where $e:=\{i,j\}$ is an edge in~$\E$.
Let $c\in\C$ and suppose that $s=x_0\prod_{u\in U}(y_0-u)$ as in \DEF{op}.
For example, we read from this table that if $1\in e$, then
$T\circ M\circ H\circ\mu(e)=h_2$, where $h_2=x_i^2+y_i^2-\l_e^2s^2$.}
\label{tab:eqn}
\centering
$
\begin{array}{rcccc}
                                    & 0\in e & 0,1,2\notin e & 1\in e & 2\in e   \\\hline
\mu(e):                             & f_1  & g_1 & h_1  & k_1    \\
H\circ\mu(e):                       & f_2  & g_1 & h_1  & k_1    \\
H_{y_0}\circ\mu(e):                     & f_3  & g_1 & h_1  & k_1    \\
T\circ M\circ H\circ\mu(e):         & f_4  & g_2 & h_2  & k_2    \\
T\circ N\circ H\circ\mu(e):         & f_5  & g_2 & h_2  & k_2    \\
T\circ M_c\circ M\circ H\circ\mu(e):& f_6  & g_2 & h_2  & k_2    \\
T\circ N_c\circ M\circ H\circ\mu(e):& f_7  & g_2 & h_2  & k_2    \\\hline
\end{array}
$
\\[4mm]
The polynomials in the above table are defined as follows:
\\[2mm]
$
\begin{array}{l}
f_1:=(x_0-x_i)^2+(y_0-y_i)^2-\l_e^2,\\
f_2:=(1-x_ix_0)^2+(y_0-y_ix_0)^2-\l_e^2x_0^2,\\
f_3:=(x_0-x_iy_0)^2+(1-y_iy_0)^2-\l_e^2y_0^2,\\
f_4:=(x_i - \ii y_0 + \ii y_i)(-2  + x_0s(x_i + \ii y_0 - \ii y_i)) - x_0\l_e^2s^3,\\
f_5:=(1 + \ii x_i y_0 - y_i y_0)(2\ii + x_0s(1 - \ii x_i y_0 - y_i y_0))-y_0^2x_0\l_e^2s^3,\\
f_6:=\scale{0.91}{2 \ii (c + \ii x_i - y_i) +
x_0 s( c^2 + x_i^2 + 2 \ii y_0 - 2 c y_i + y_i^2 + x_0y_0 s (2 c  - 2 y_i  + y_0 x_0 s) )
-x_0 \l_e^2 s^3}\\
f_7:=
(c + \ii x_i + x_0 s - y_i)(2 \ii  + y_0 x_0 s (c - \ii x_i + x_0 s - y_i))
- y_0 x_0 \l_e^2 s^3,
\end{array}
$
\\
$
\begin{array}{@{}ll}
g_1:=(x_i-x_j)^2+(y_i-y_j)^2-\l_e^2,    &
g_2:=(x_i-x_j)^2+(y_i-y_j)^2-\l_e^2s^2, \\
h_1:=x_i^2+y_i^2-\l_e^2,                &
h_2:= x_i^2+y_i^2-\l_e^2s^2,            \\
k_1:=(x_i-1)^2+y_i^2-\l_e^2,            &
k_2:=(x_i-1)^2+y_i^2-\l_e^2s^2.         \\
\end{array}
$
\end{table}

\begin{example}
\label{exm:H}
Suppose that $\cG$ is equal to the calligraph $\cH$ in \FIG{HH}.
It follows from \TAB{eqn} that
\begin{gather*}
\begin{array}{r@{\,}c@{\,}l}
\mu(\{0,3\})&=&(x_0-x_3)^2+(y_0-y_3)^2-\l_{30}^2,\\
\mu(\{0,4\})&=&(x_0-x_4)^2+(y_0-y_4)^2-\l_{40}^2,\\
\mu(\{3,4\})&=&(x_3-x_4)^2+(y_3-y_4)^2-\l_{43}^2,\\
\mu(\{3,1\})&=&x_3^2+y_3^2-\l_{31}^2,\\
\mu(\{4,2\})&=&(x_4-1)^2+y_4^2-\l_{42}^2.
\end{array}
\end{gather*}
Suppose that $\Gamma_\fH$ is the series associated to $\cG$
and that $(0,u)$ is a base point of $\alpha_p^m\circ\gamma_1(\fH)$,
where $p=(0,\ii)$ is a base point of~$\gamma_1(\fH)$ with multiplicity~$m$.
If $\cG$ is centric, then $u\in\{0,-\ii\}$ and thus
\[
Z:=B\circ\oF\circ\oM\circ\oH\circ\oG\circ\mu(\E)\subseteq \{(0,0),~(0,-\ii)\}.
\]
Since two general curves in $\Gamma_\fH$
do not contain a common component, we can replace
$\oF\circ\oM$, $\oH$, $\oG$ by $M$, $H$, $G$
so that $B\circ M\circ H\circ G\circ\mu(\E)=Z$.
Let $P:=M\circ H\circ\mu(\E)$, where
\begin{gather*}
\begin{array}{r@{\,}c@{\,}l}
M\circ H\circ\mu(\{0,3\})&=&(x_3-\ii y_0+\ii y_3)(-2+ x_0(x_3+\ii y_0-\ii y_3))- x_0\l_{30}^2,\\
M\circ H\circ\mu(\{0,4\})&=&(x_4-\ii y_0+\ii y_4)(-2+ x_0(x_4+\ii y_0-\ii y_4))- x_0\l_{40}^2,\\
M\circ H\circ\mu(\{3,4\})&=&(x_3-x_4)^2+(y_3-y_4)^2-\l_{43}^2,\\
M\circ H\circ\mu(\{3,1\})&=&x_3^2+y_3^2-\l_{31}^2,\\
M\circ H\circ\mu(\{4,2\})&=&(x_4-1)^2+y_4^2-\l_{42}^2.
\end{array}
\end{gather*}
By applying \LEM{op} two times we deduce that
$B\circ G(P)=Z$.
In this example we prepare the reader for our general strategy
to show that
\[
B\circ G(P)\subseteq \{(0,0),~(0,-\ii)\}.
\]
In order to motivate this strategy, we first investigate a more straightforward approach
that leads to a problem.
Let $\kappa\c\C^{11}\dto\C^2$ be defined as the linear projection~$\Upsilon_R\wr\{x_0,y_0\}$.
We set $L:=V(\set{\l_e-\lambda_e}{e\in \E})$, where $\lambda\in\C^\E$ is a general edge length assignment.
We define $C_\lambda$ as the Zariski closure of~$(V(P)\cap L)\setminus V(x_0)$.
Thus the section $C_\lambda\subset \C^{11}$ of~$V(P)$ is a curve consisting of one or more irreducible components
and its linear projection $\kappa(C_\lambda)\subset \C^2$
passes through the base points in $B\circ G(P)$.
These base points lie by definition on the line~$V(x_0)\subset \C^2$,
which is a component of the linear projection~$\kappa(V(P)\cap L)$.
Thus $V(P)\cap L$ contains aside $C_\lambda$, additional components
in the hyperplane $V(x_0)\subset\C^{11}$ that project to the line~$V(x_0)\subset \C^2$.
We want to show that a base point
has no preimage \wrt the restricted projection $\kappa|_{C_\lambda}$
and thus it is hopeless to recover base points directly from the zero set $V(P)\cap V(x_0)\subset \C^{11}$
with the ideal
\begin{gather*}
\bas{P\cup\{x_0\}}=\langle
x_3-\ii y_0+\ii y_3,~
x_4-\ii y_0+\ii y_4,~\\
(x_3-x_4)^2+(y_3-y_4)^2-\l_{43}^2,~
x_3^2+y_3^2-\l_{31}^2,~
(x_4-1)^2+y_4^2-\l_{42}^2
\rangle.
\end{gather*}
Let $U:=\set{u}{(0,u)\in B\circ G(P)}$ be the set of $y_0$-coordinates
of the base points.
We define the map $\rho\c \C^4\to \C^3$ as
\[
(x_3,y_3,x_4,y_4)
\mapsto
\bigl(
(x_3-x_4)^2+(y_3-y_4)^2,~
x_3^2+y_3^2,~
(x_4-1)^2+y_4^2
\bigr).
\]
Thus $\rho$ is dominant and sends vertex coordinates to the squares of the corresponding edge lengths.
Since $\lambda$ was chosen general we have
\[
(\lambda_{43}^2,\lambda_{31}^2,\lambda_{42}^2)
\notin
\rho(V(x_3-\ii u+\ii y_3,x_4-\ii u+\ii y_4)),
\]
for all $u\in U$,
and thus there does not exist a $q\in C_\lambda$
\st $\kappa(q)=(0,u)$.
This means that the projection $\kappa|_{C_\lambda}$ is not proper.
For an example of an improper map, we may think of
the projection of the hyperbola $\set{(x,y)\in \C^2}{xy=1}$
to the $x$-axis in which case the origin does not have a preimage
(see \citep[Example~II.4.6.1]{har}).

To avoid the above problem of base points not having preimages,
we consider~$T(P)$ instead of~$P$.
This modification does not exclude any candidates for base points since
we know from \LEM{T} that
$B\circ G(P)\subseteq B\circ G\circ T(P)$.
Let the birational map~$\eta\c \C^{11}\dto \C^{11}$ be defined as
\[
(x_0,y_0,x_3,y_3,x_4,y_4,\l_{43},\l_{31},\l_{42})
\mapsto
(x_0 s^{-1},y_0,x_3,y_3,x_4,y_4,\l_{43}s^{-1},\l_{31}s^{-1},\l_{42}s^{-1}).
\]
Notice that $\eta$ sends $V(P)$ to $V\circ T(P)$, but is not defined at the zero set~$V(s)$,
where $s=x_0\prod_{u\in U}(y_0-u)$.
Thus, $V\circ T(P)\cap L$ consists of the image~$\eta(C_\lambda)$
and additional components in~$V(s)$.
The preimages \wrt the linear projection $\kappa\c\C^{11}\dto\C^2$ of base points in~$B\circ G\circ T(P)$
are contained in the zero set~$V\circ T(P)\cap V(x_0)\subset \C^{11}$ with the ideal:
\begin{multline*}
\bas{T(P)\cup\{x_0\}}=\langle
x_0,~x_3-\ii y_0+\ii y_3,~
x_4-\ii y_0+\ii y_4,~\\
(x_3-x_4)^2+(y_3-y_4)^2,~
x_3^2+y_3^2,~
(x_4-1)^2+y_4^2
\rangle.
\end{multline*}
The generators of this ideal
factor into linear factors and thus
there exist linear spaces $W_1,\ldots, W_8$ \st
\[
V\circ T(P)\cap V(x_0)=W_1\cup\cdots\cup W_8.
\]
In the proof of \LEM{case5678} in \SEC{proof} we show
that for each base point $(0,u)$ in $B\circ T\circ G(P)$,
there exists $1\leq i\leq 8$ \st $\kappa(W_i)=(0,u)$.
We deduce from \PRP{vert} in \SEC{proof} that $\kappa(W_i)$
equals either the line $V(x_0)$ or a point in $\{(0,0),(0,-\ii)\}$.
Thus, we conclude that
\[
B\circ G(P)\subseteq B\circ T\circ G(P)\subseteq \{(0,0),~(0,-\ii)\}.
\]
We remark that in \PRP{vert} and \LEM{case5678}
we consider the linear projection $\pi=\Upsilon_R\wr\Upsilon_S$ instead of $\kappa=\Upsilon_R\wr\{x_0,y_0\}$,
but the translation to the setting of this example is straightforward.
\END
\end{example}

\begin{proposition}
\label{prp:cases}
A calligraph $\cG$ is centric if the following 8 conditions hold for all~$c\in\{0,\ii\}$:
\\
$
\begin{array}{rrcl}
1. & (H\circ\mu(\E))\sub{x_0\to 0}                        &\supseteq  & \{1+y_0^2\} ,\\
2. & B_0\circ G\circ \mu(\E)                              &=          & \varnothing,             \\
3. & B  \circ G\circ H\circ\mu(\E)                        &\subseteq  & \{(0,\ii),~(0,-\ii)\}, \\
4. & B  \circ G\circ H_{y_0}\circ\mu(\E)                  &\nsupseteq & \{(0,0)\},  \\
5. & B  \circ G\circ T\circ M\circ H\circ\mu(\E)          &\subseteq  & \{(0,0),~(0,-\ii)\}, \\
6. & B  \circ G\circ T\circ N\circ H\circ\mu(\E)          &\nsupseteq & \{(0,0)\},  \\
7. & B  \circ G\circ T\circ M_c\circ M\circ H\circ\mu(\E) &=          & \varnothing,  \\
8. & B  \circ G\circ T\circ N_c\circ M\circ H\circ\mu(\E) &\nsupseteq & \{(0,0)\}.  \\
\end{array}
$
\end{proposition}

\begin{proof}
Let $\Gamma_\fH$ be the series associated to $\cG$ as characterized by \LEM{dim}\ref{lem:dim:c}.
It is straightforward to see that Condition~1 implies that
the curves in $\Gamma_\fH$ meet the line at infinity
only at the cyclic points (see \EXM{C3}).
The first equality in each row of \TAB{cases} is a direct consequence
of the definitions (see \EXM{ophat}).
Thus, if the remaining set relations in \TAB{cases} hold, then the assertions of this proposition
follow from \LEM{centric}.
Two general curves in~$\Gamma_\fH$ do not have a common component,
and thus their defining polynomials
do not contain a $(x_0)$-factor or $(y_0)$-factor.
Hence, we can replace $\oG$, $\oH$, $\oF_{x_0}\circ \oM$, $\oF_{y_0}\circ\oN$
by $G$, $H$, $M$, $N$, \resp.
Moreover, we can replace $B_{y_0}$ with $B$,
since the operator~$N_c$ interchanges $x_0$ and $y_0$.
The set relations in \TAB{cases} are now a consequence of \LEMS{lem:op,lem:T},
and thus we concluded the proof.
\end{proof}

\begin{table}[!ht]
\caption{See the proof of \PRP{cases}. Let $p:=(0,\ii)$, $q:=(0,c)$ and $m,n\in\Z_{\geq 0}$.
Let $\bp(\fF)$ denote the base points of the subset $\fF\subset \C[x,y]$.
The set relation~$\A\propto\B$ for $\A,\B\subset\C^2$ indicates that $\A\subseteq \set{(y,x)}{(x,y)\in\B}$.
}
\label{tab:cases}
\centering
\scalebox{0.80}{$
\begin{array}{rrcl}
\bp(\gamma_0(\fH))                              =& B_0\circ \oG\circ \mu(\E)                                                 &=        & B_0\circ G\circ\mu(\E)                                \\
\bp(\gamma_1(\fH))                              =& B\circ \oH\circ \oG\circ \mu(\E)                                          &=        & B\circ G\circ H\circ \mu(\E)                          \\
\bp(\gamma_2(\fH))                              =& B\circ \oH_{y_0}\circ \oG\circ \mu(\E)                                        &=        & B\circ G\circ H_{y_0}\circ \mu(\E)                        \\
\bp(\alpha_p^m\circ\gamma_1(\fH))               =& B\circ \oF_{x_0}\circ \oM\circ \oH\circ \oG\circ\mu(\E)                         &\subseteq& B\circ G\circ T\circ M\circ H\circ\mu(\E)             \\
\bp(\beta_p^m\circ\gamma_1(\fH))                =& B_{y_0}\circ \oF_{y_0}\circ \oN\circ \oH\circ \oG\circ\mu(\E)                     &\propto  & B\circ G\circ T\circ N\circ H\circ\mu(\E)           \\
\bp(\alpha_q^n\circ\alpha_p^m\circ\gamma_1(\fH))=& B\circ \oF_{x_0}\circ \oM_c\circ \oF_{x_0}\circ \oM\circ \oH\circ \oG\circ\mu(\E)     &\subseteq& B\circ G\circ T\circ M_c\circ M\circ H\circ\mu(\E)    \\
\bp(\beta_q^n\circ\alpha_p^m\circ\gamma_1(\fH))= & B_{y_0}\circ \oF_{y_0}\circ \oN_c\circ \oF_{x_0}\circ \oM\circ \oH\circ \oG\circ\mu(\E) &\propto  & B\circ G\circ T\circ N_c\circ M\circ H\circ\mu(\E)
\end{array}
$}
\end{table}

\section{All calligraphs are centric}
\label{sec:proof}

In this section we show that the eight sufficient conditions for centricity in \PRP{cases}
hold for all calligraphs.
We then conclude the proof of the main results \THM{class} and \COR{class}
by referring to \PRP{pseudo} in \SEC{bp}.
The eight conditions of \PRP{cases} are proven in \LEMS{lem:case2,lem:case134,lem:case5678}.
\LEM{case5678} depends on \PRP{vert}, which is proven in \APP{lines}.

We assume the notation of~\SEC{op}.
The following lemma shows that the base points of series associated to calligraphs
lie on the line at infinity.

\begin{lemma}
\label{lem:case2}
$B_0\circ G\circ \mu(\E)=\varnothing$.
\end{lemma}

\begin{proof}
We suppose by contradiction that $(\alpha,\beta)\in B_0\circ G\circ\mu(\E)$ is a base point.

Let
$C_{\lambda}:=V(\{\mu(e)\sub{\l_e\to\lambda_e}:e\in \E\})$
for all $\lambda\in \C^\E$ and
let $\kappa\c\C^{2|\V|}\to\C^2$ be defined as the linear projection $\set{x_i,y_i}{i\in \V}\wr\{x_0,y_0\}$.
We denote the Zariski closure of the linear projection~$\kappa(C_\lambda)$ by $C'_\lambda$.
Suppose that $\tla\in\C^\E$ is a general choice of edge length assignment
and let $L_\delta:=\set{\lambda\in\C^\E}{|\lambda_e-\tla_e|\leq\delta}$
with $\delta\in \R_{>0}$.
Recall from \RMK{B} that $C'_\tla\subset\C^2$ is a coupler curve that passes through the
base point $(\alpha,\beta)\in\C^2$.

\textbf{Claim 1.}
\textit{There exists a $\delta\in\R_{>0}$ \st
for all $\lambda\in L_\delta$
and general $q\in C'_\lambda$,
we have
$|\kappa^{-1}(q)\cap C_\lambda|>0$
and
$|\kappa^{-1}(\alpha,\beta)\cap C_\lambda|=0$.}
\\
Since $q$ is general in the linear projection~$C'_\lambda$, it must have a preimage in $C_\lambda$ and thus $|\kappa^{-1}(q)\cap C_\lambda|>0$.
For the remaining assertion, we consider the map~$\rho\c \C^{2|\V|}\to\C^\E$
that sends $(x_i,y_i)_{i\in\V}$ to $((x_i-x_j)^2+(y_i-y_j)^2))_{\{i,j\}\in\E}$.
Notice that $C_\lambda\subset \C^{2|\V|}$ corresponds to the fiber $\rho^{-1}(\lambda)$.
It follows from the definition of calligraphs that $|\E|=2|\V|-1$
and we know from \LEM{dim}\ref{lem:dim:a} that $\dim C_\lambda=1$.
We deduce that $\rho$ is dominant
and thus the image via $\rho$ of the codimension two set~$\set{(x,y)\in\C^{2|\V|}}{x_0=\alpha,~y_0=\beta}$
has codimension at least one in~$\C^\E$.
In other words, the set~$W:=\set{\rho(x,y)}{(x_0,y_0)\neq (\alpha,\beta)}$ is Zariski dense in~$\C^\E$.
This implies that $|\kappa^{-1}(\alpha,\beta)\cap C_\lambda|=0$ for all $\lambda\in W$.
There exists a radius~$\delta$ \st $L_\delta\subset W$
and thus we conclude that Claim~1 holds.

Suppose that $p:=\left(\tx_i,\ty_i\right)_{i\in \V}$ is a point in~$C_\tla$
\st $|\tx_0-\alpha|<\e$ and $|\ty_0-\beta|<\e$ for some $\e\in\R_{>0}$.
Recall that for all $e\in \E$ \st $0\in e$, we have:
\[
(\tx_0-\tx_i)^2+(\ty_0-\ty_i)^2=\tla_e^2.
\]
It follows from Claim~1 with $q=\kappa(p)$ that
we can choose $\e>0$ arbitrary small
and that
$|\kappa^{-1}(\alpha,\beta)\cap C_\lambda|=0$
for all $\lambda\in L_\delta$, where $\delta>0$ is small enough.
On the other hand, there exists a $\lambda'\in \C^\E$ \st
for all $e\in \E$ we have
\[
\lambda'_e=0
~~\text{ if } 0\notin e,
\quad\text{and}\quad
(\alpha-\tx_i)^2+(\beta-\ty_i)^2=(\tla_e+\lambda'_e)^2
~~\text{ if } 0\in e.
\]
By choosing $\e$ very small we can ensure that $\tla+\lambda'\in L_\delta$.
We arrived at a contradiction with Claim~1, since $|\kappa^{-1}(\alpha,\beta)\cap C_{\tla+\lambda'}|\geq 1$
by construction.
Hence, the base point $(\alpha,\beta)$ cannot exist and we concluded the proof
of the main assertion.
\end{proof}

\begin{lemma}
\label{lem:case134}
~
\begin{Menum}
\item $(H\circ\mu(\E))\sub{x_0\to 0}\supseteq \{1+y_0^2\}$.\label{it:case134:H}
\item $B\circ G\circ H\circ \mu(\E)\subseteq \{(0,\ii),(0,-\ii)\}$.\label{it:case134:BGH}
\item $B\circ G\circ H_{y_0}\circ\mu(\E) \nsupseteq \{(0,0)\}$.\label{it:case134:BGHy}
\end{Menum}
\end{lemma}

\begin{proof}
\ref{it:case134:H}
From \TAB{eqn} we see that $1+y_0^2=f_2\sub{x_0\to 0}\in (H\circ\mu(\E))\sub{x_0\to 0}$.

\ref{it:case134:BGH} It follows from \ref{it:case134:H} that the elimination
ideal $\bas{H\circ\mu(\E)\cup\{x_0\}}\cap S$
contains $1+y_0^2$ as well.
Hence, $1+y_0^2$ is in the ideal $\bas{G\circ H\circ\mu(\E)\cup\{x_0\}}\subset S$
so that if $(0,u)\in B\circ G\circ H\circ \mu(\E)$, then $u=\pm\ii$.

\ref{it:case134:BGHy}
Using \TAB{eqn} we find that $1+x_0^2\in (H_{y_0}\circ\mu(\E))\sub{y_0\to 0}$.
Similarly as in the proof of \ref{it:case134:BGH} we deduce that
if $(u,0)\in B_0\circ G\circ H_{y_0}\circ\mu(\E)$, then $u=\pm\ii\neq 0$.
\end{proof}

\begin{proposition}
\label{prp:vert}
If $P:=T\circ M\circ H\circ\mu(\E)$, then
the zero set $V(P\cup \{x_0\})$ is a union of
linear spaces $W_1\cup\cdots\cup W_r$
and the linear projection~$\pi(W_i)$ is equal to either
$V(x_0)$, $V(x_0,y_0)$ or $V(x_0,y_0+\ii)$, for all $1\leq i\leq r$.
\end{proposition}

\begin{proof}
See \APP{lines}.
\end{proof}

\newpage
\begin{lemma}
\label{lem:case5678}
~
\begin{Menum}
\item\label{it:case5678:BGTMH}
$B\circ G\circ T\circ M\circ H\circ\mu(\E)\subseteq\{(0,0),~(0,-\ii)\}$.

\item\label{it:case5678:BGTNH}
$B\circ G\circ T\circ N\circ H\circ\mu(\E)\nsupseteq\{(0,0)\}$.

\item\label{it:case5678:BGTMMH}
$B\circ G\circ T\circ M_c\circ M\circ H\circ\mu(\E)=\varnothing$.

\item\label{it:case5678:BGTNMH}
$B\circ G\circ T\circ N_c\circ M\circ H\circ\mu(\E)\nsupseteq\{(0,0)\}$.
\end{Menum}
\end{lemma}

\begin{proof}
\ref{it:case5678:BGTMH}
Suppose that $P:=T\circ M\circ H\circ\mu(\E)$
is characterized as in \TAB{eqn}.
Let $s=x_0\prod_{u\in U}(y_0-u)$ and $U=\set{u}{(0,u)\in B\circ G(P)}$ be as in \DEF{op}.
By \LEM{T} there exists $g\in S$ \st
\[
\oG(P)=\{ g \cdot x_0\cdot  \prod_{u\in U}(y_0-u) \}.
\]
By \LEM{G} this means geometrically that
\[
\pi\circ V(P)=V(g)\cup V(x_0)\cup\left(\bigcup_{u\in U} V(y_0-u)\right).
\]
We know from \PRP{vert} that $V(P\cup\{x_0\})=W_1\cup\cdots\cup W_r$,
where $W_i$ is a linear space for all $1\leq i\leq r$.
Notice that the restriction $\pi|_{W_i}$ is surjective for all $1\leq i\leq r$.
The preimage of
the component $V(x_0)\subset\pi\circ V(P)$
consists of the union of those linear spaces~$W_i$ satisfying $\pi(W_i)=V(x_0)$.
The preimage of the intersection~$V(y_0-u)\cap V(x_0)$ in $\pi\circ V(P)$ consists of the union of
those linear spaces~$W_i$ for which $\pi(W_i)=V(x_0,y_0-u)$.
We conclude from \PRP{vert} that $u\in\{0,-\ii\}$.
Hence, $U\subseteq \{0,-\ii\}$ and therefore
$B\circ G(P)\subseteq \{(0,0),~(0,-\ii)\}$ as was to be shown.

The proofs for the assertions \ref{it:case5678:BGTNH}, \ref{it:case5678:BGTMMH} and \ref{it:case5678:BGTNMH}
are similar to the proof of \ref{it:case5678:BGTMH} except that $P$
is equal to $T\circ N\circ H\circ\mu(\E)$,
$T\circ M_c\circ M\circ H\circ\mu(\E)$ and
$T\circ N_c\circ M\circ H\circ\mu(\E)$, \resp.
See \TAB{eqn} for a characterization of the
the polynomials in these sets.
Since each polynomial in $P\sub{x_0\to 0}$
is either linear or a product of two linear factors,
it follows that $V(P\cup\{x_0\})$ is a union of linear
subspaces $W_1\cup\cdots \cup W_r$.
Thus, $\pi|_{W_i}$ is surjective and for all $u\in U$
there exists $1\leq i\leq r$ \st $\pi(W_i)=V(x_0,y_0-u)$.
For assertion \ref{it:case5678:BGTNH} we observe that $V(P\cup \{x_0,y_0\})=\varnothing$
and thus $0\notin U$ so that $B\circ G(P)\nsupseteq\{(0,0)\}$.
For \ref{it:case5678:BGTMMH} and \ref{it:case5678:BGTNMH} we have
$\pi(W_i)=V(x_0)$ for all $1\leq i\leq r$, because no polynomial
in~$P\sub{x_0\to 0}$ depends on $y_0$.
Therefore, $B\circ G(P)=\varnothing$ for assertions \ref{it:case5678:BGTMMH} and \ref{it:case5678:BGTNMH} so that we
concluded the proof.
\end{proof}

\begin{proof}[Proof of \THM{class} and \COR{class}]
It follows from \PRP{cases}
in combination with \LEMS{lem:case2,lem:case134,lem:case5678}
that all calligraphs are centric.
Thus \THM{class} and \COR{class} are now a direct consequence of \PRP{pseudo}.
\end{proof}

\appendix
\clearpage
\section{Overview notation}
\label{sec:nota}
Below we list an overview of notation that is used across more than one section.
See the assigned sections \SEC{main}, \SEC{bp} and \SEC{op} for the precise definitions.
\begin{center}
\footnotesize
\begin{tabular}{l@{~}l}
\SEC{main}&\\\hline
$\vv(\cG)$, $\ee(\cG)$ & Vertices $\vv(\cG)\subset \Z_{\geq0}$ and edges of a graph~$\cG$. \\
$\cL$, $\cR$, $\cC_v$ & The calligraphs defined in \FIG{cal} with $v\in \Z_{\geq 3}$. \\
$\Omega_\cG$, $\Xi_\cG^\omega$ & Sets of edge length assignments and realizations \\
                               & of a marked graph $\cG$ with edge length assignment $\omega\in \Omega_\cG$. \\
$\cc(\cG)$            & Number of realizations of a graph $\cG$. \\
$\sng C$, $\deg C$, $g(C)$ & Singular locus, degree and geometric genus of a curve~$C$. \\
$[\cG]$, $\tt_\omega(\cG)$, $\mm(\cG)$ & Class, coupler curve and coupler multiplicity of a calligraph~$\cG$. \\[2ex]
\SEC{bp}&\\\hline
$\alpha^m_p$, $\beta^m_p$ & Maps $\C[x,y]\to\C[x,y]$ that perform substitutions and quotients \\
                          & for base point analysis. \\
$\ii$                     & Imaginary unit.\\
$\gamma_0$, $\gamma_1$, $\gamma_2$ & Maps $\C[z_0,z_1,z_2]\to\C[x,y]$ that send a homogeneous polynomial $h$ to \\
                                   & $h(1,x,y)$, $h(x,1,y)$ and $h(y,x,1)$, \resp.\\
$\Gamma_\fH$              & The series associated to a calligraph~$\cG$ with $\fH$ the set of homogeneous  \\
                          & polynomials, whose zero sets are projectivized coupler curves of $\cG$.\\[2ex]
\SEC{op} & \\\hline
$\V$, $\E$ & $\V=\vv(\cG)\setminus\{1,2\}$ and $\E=\ee(\cG)\setminus\{\{1,2\}\}$. \\
$S$, $\Upsilon_S$, $R$, $\Upsilon_R$ & Polynomial rings $S=\C[\Upsilon_S]$ and $R=\C[\Upsilon_R]$ \\
                                     & with $\Upsilon_S\subset \Upsilon_R$ being sets of variables. \\
$\pi=\Upsilon_R\wr\Upsilon_S$ & The linear projection $\C^{|\Upsilon_R|}\to \C^{|\Upsilon_S|}$ that forgets coordinates.\\
$f\sub{a_0\to b_0,\ldots}$ & Substitution operator for polynomials $f\in R$.\\
$\wp(R)$                & The set of all finite subsets of the polynomial ring~$R$. \\
$\mu$                   & Map $\mu\c\wp(\E)\to\wp(R)$. See \NTN{o} as $\mu$ is also the map~$\E\to R$\\
                        & that assigns a quadratic polynomial to an edge!\\
$H_r$, $M_c$, $N_c$, $T$, $G$ & Maps $\wp(R)\to\wp(R)$ that are defined via operators in \DEF{ophat}\\
                        & and \DEF{op}. We abbreviate $H=H_{x_0}$, $M=M_\ii$ and $N=N_\ii$.\\
$B$                     & Map $\wp(R)\to\operatorname{powerset}(\C^2)$ for assigning base points.\\
$V$                     & Map $\wp(R)\to\operatorname{powerset}\left(\C^{|\Upsilon_R|}\right)$ for assigning the zero set.\\
$\pi\circ V(P)$         & The Zariski closure of $\set{\pi(p)}{p\in V(P)}$, see \NTN{o}.

\end{tabular}
\end{center}

\vspace{1cm}
\begin{remark}
\label{rmk:generic}
Our definition of ``general'' in \SEC{class} is closely related
to the notions of \df{generic} in rigidity theory (see for example~\citep[\textsection2]{rigid-jack-2018})
and \df{generic point} in scheme theory (see \citep[Example~II.2.3.4]{har}).
For example, a point $p\in\C^2$ is general with respect to some property,
if $f(p)\neq 0$ for some polynomial $f\in\C[u,v]$ that depends on this property.
In comparison, $p\in\C^2$ is generic in the rigidity theoretic sense
if $f(p)\neq 0$ for all~$f\in\Z[u,v]$.
\END
\end{remark}

\clearpage
\section{Case study for the getNoR algorithm}
\label{sec:tree}

We compute the number of realizations~$\cc(\cG)$ using \ALG{nor},
where the minimally rigid graph $\cG$ is defined in \FIG{large}.
We depict the recursive execution tree of \ALG{nor} in \FIG{tree}
together with references to the figures that depict the corresponding
minimally rigid graphs and calligraphs.

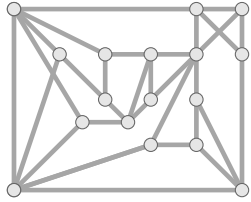
\begin{figure}[!ht]
\centering
\setlength{\tabcolsep}{10mm}
\begin{tikzpicture}[scale=0.3,rotate=-90]
\BUedge \Paf \Pbf;
\BUedge \Paf \Pbe;
\BUedge \Paf \Pae;
\BUedge \Pbf \Pae;
\BUedge \Pbf \Pef;
\BUedge \Pce \Pbe;
\BUedge \Pce \Pef;
\BUedge \Pce \Pde;
\BUedge \Pae \Paa;
\BUedge \Pae \Pbe;
\BUedge \Pbe \Pdd;
\BUedge \Pef \Pea;
\BUedge \Pef \Pde;
\BUedge \Pea \Pdd;
\BUedge \Pea \Pdd;
\BUedge \Pdd \Pde;
\BUedge \Pbe \Pbd;
\BUedge \Pbe \Pcd;
\BUedge \Paa \Pbb;
\BUedge \Paa \Pbc;
\BUedge \Paa \Pkk;
\BUedge \Pea \Pbb;
\BUedge \Pea \Pkk;
\BUedge \Pbc \Pbd;
\BUedge \Pbc \Pcc;
\BUedge \Pkk \Pjj;
\BUedge \Pbb \Pcc;
\BUedge \Pcc \Pjj;
\BUedge \Pjj \Pbd;
\BUedge \Pjj \Pcd;
\BUedge \Pcd \Pbd;
\BUedge \Paa \Pea;
\BIvertW \Paa \
\BIvertN \Pea \
\BIvertW \Paf \
\BIvertN \Pbf \
\BIvertN \Pce \
\BIvertW \Pae \
\BIvertS \Pef \
\BIvertE \Pde \
\BIvertE \Pdd \
\BIvertW \Pbe \
\BIvertW \Paa \
\BIvertN \Pea \
\BIvertE \Pbc \
\BIvertE \Pkk \
\BIvertW \Pbb \
\BIvertN \Pcc \
\BIvertE \Pjj \
\BIvertS \Pbd \
\BIvertS \Pcd \
\end{tikzpicture}
\caption{The minimally rigid graph $\cG$ has 17 vertices.}
\label{fig:large}
\end{figure}

We notice that $\cG$ does not have degree two vertices
and that $(\cG_1,\cG_2)$ is a non-trivial calligraphic split for $\cG$
(see \FIG{level1}).
This step corresponds to the first vertical ``split'' separator in \FIG{tree}.
We have $\cc(\cG)=[\cG_1]\cdot [\cG_2]$ by Axiom~\AXM{2}
and thus we would like to compute the classes
$[\cG_1]$ and $[\cG_2]$ using \ALG{class}.

In order to compute $[\cG_1]$
we first determine the number of realizations
for the minimally rigid graphs $\cG_{1L}=\cG_1\cup\cL$,
$\cG_{1R}=\cG_1\cup\cR$ and $\cG_{1C}=\cG_1\cup\cC_v$ for some vertex $v\notin\vv(\cG_1)$
as is depicted in \FIG{level2a}.
This step is associated to the first vertical ``glue'' separator in \FIG{tree}.

In order to compute $\cc(\cG_{1R})$, $\cc(\cG_{1L})$ and
$\cc(\cG_{1C})$ we call \ALG{nor} three times
and arrive at the
second vertical ``split'' separator in \FIG{tree}.
We continue recursively and thus the remaining
steps are analogous. The captions of the corresponding
figures are self-explanatory.

A minimally rigid graph that corresponds
to one of the 36 leaves of the execution tree
has at most 10 vertices instead of 17.
We either continue recursively or resort to the fall-back algorithm for
computing the number of realizations, namely
Algorithm~\cite{rigid-alg}.
We conclude that $\cc(\cG)=200192$.

\begin{figure}[!hp]
\begin{tikzpicture}[scale=0.9,
		edge/.style={line width=2pt},
		sedge/.style={edge,colG},
		gedge/.style={edge,colB},
		gnode/.style={draw=black,fill=white,minimum width=1cm,font=\scriptsize},
		elabel/.style={font=\scriptsize,black!50!white},
		sgline/.style={line width=0.5pt,dotted,black!50!white},
		ref/.style={black!4!white},
		rlabel/.style={font=\scriptsize,black}
  ]
  \coordinate (levelsep1) at (0,6);
  \coordinate (levelsep2) at (0,4);
  \coordinate (levelsep3) at (0,1);
  \coordinate (levelsep4) at (0,0.65);
  \coordinate (leveldist1) at (3,0);
  \coordinate (leveldist2) at (3,0);
  \coordinate (leveldist3) at (3,0);
  \coordinate (leveldist4) at (3,0);
  \coordinate (up) at (0,12);
  \coordinate (figdown) at (0,-12.5);
  \coordinate (figup) at (0,12.5);

  \draw[sgline] ($0.5*(leveldist1)+(up)$) node[above,elabel] {split} -- ++ ($-2*(up)$);
  \draw[sgline] ($(leveldist1)+0.5*(leveldist2)+(up)$) node[above,elabel] {glue} -- ++ ($-2*(up)$);
  \draw[sgline] ($(leveldist1)+(leveldist2)+0.5*(leveldist3)+(up)$) node[above,elabel] {split} -- ++ ($-2*(up)$);
  \draw[sgline] ($(leveldist1)+(leveldist2)+(leveldist3)+0.5*(leveldist4)+(up)$) node[above,elabel] {glue} -- ++ ($-2*(up)$);
  \node[gnode] (g) at (0,0) {$\mathcal G$};
  \foreach \i in {1,2}
  {
		\node[gnode] (g\i) at ($(leveldist1)+{(-1)^\i}*(levelsep1)$) {$\mathcal G_\i$};
		\draw[sedge] (g)edge(g\i);
		\foreach \j [count=\jc] in {L,R,C}
		{
			\node[gnode] (g\i\j) at ($(g\i)+(leveldist2)-2*(levelsep2)+\jc*(levelsep2)$) {$\mathcal G_{\i\j}$};
			\draw[gedge] (g\i)to node[pos=0.3,above,elabel] {$\mathcal \j$} (g\i\j);
			\foreach \k in {1,2}
			{
				\node[gnode] (g\i\j\k) at ($(g\i\j)+(leveldist3)+{(-1)^\k}*(levelsep3)$) {$\mathcal G_{\i\j\k}$};
				\draw[sedge] (g\i\j)edge(g\i\j\k);
				\foreach \l [count=\lc] in {L,R,C}
				{
					\node[gnode] (g\i\j\k\l) at ($(g\i\j\k)+(leveldist4)-2*(levelsep4)+\lc*(levelsep4)$) {$\mathcal G_{\i\j\k\l}$};
					\draw[gedge] (g\i\j\k)to node[pos=0.8,above=-1pt,elabel] {$\mathcal \l$}(g\i\j\k\l);
				}
			}
		}
  }
  \begin{scope}[on background layer] 
		\fill[ref] ($(g1)-(0.8,0.8)$) rectangle ($(g2)+(0.8,0.8)$);
		\node[rlabel] at ($(leveldist1)+(figdown)$) {Figure~\ref{fig:level1}};
		\fill[ref] ($(g1L)-(0.8,0.8)$) rectangle ($(g1C)+(0.8,0.8)$);
		\node[rlabel] at ($(leveldist1)+(leveldist2)+(figdown)$) {Figure~\ref{fig:level2a}};
		\fill[ref] ($(g1L1)-(0.8,0.8)$) rectangle ($(g1C2)+(0.8,0.8)$);
		\node[rlabel] at ($(leveldist1)+(leveldist2)+(leveldist3)+(figdown)$) {Figure~\ref{fig:level3a}};
		\foreach \i in {L,R,C}
		{
			\fill[ref] ($(g1\i1L)-(0.8,0.35)$) rectangle ($(g1\i1C)+(0.8,0.35)$);
		}
		\node[rlabel] at ($(leveldist1)+(leveldist2)+(leveldist3)+(leveldist4)+(figdown)$) {Figure~\ref{fig:level4a}};

		\fill[ref] ($(g2L)-(0.8,0.8)$) rectangle ($(g2C)+(0.8,0.8)$);
		\node[rlabel] at ($(leveldist1)+(leveldist2)+(figup)$) {Figure~\ref{fig:level2b}};
		\fill[ref] ($(g2L1)-(0.8,0.8)$) rectangle ($(g2C2)+(0.8,0.8)$);
		\node[rlabel] at ($(leveldist1)+(leveldist2)+(leveldist3)+(figup)$) {Figure~\ref{fig:level3b}};
		\foreach \i in {L,R,C}
		{
			\fill[ref] ($(g2\i1L)-(0.8,0.35)$) rectangle ($(g2\i1C)+(0.8,0.35)$);
		}
		\fill[ref] ($(g2L2L)-(0.8,0.35)$) rectangle ($(g2L2C)+(0.8,0.35)$);
		\node[rlabel] at ($(leveldist1)+(leveldist2)+(leveldist3)+(leveldist4)+(figup)$) {Figure~\ref{fig:level4b}};
  \end{scope}
\end{tikzpicture}
\caption{Execution tree for \ALG{nor} with input $\cG$ and output $\cc(\cG)=200192$.}
\label{fig:tree}
\end{figure}

\begin{figure}[!ht]
\centering
\begin{tikzpicture}[scale=0.4,rotate=-90]
\BVedge \Paf \Pbf;
\BVedge \Paf \Pbe;
\BVedge \Paf \Pae;
\BVedge \Pbf \Pae;
\BVedge \Pbf \Pef;
\BVedge \Pce \Pbe;
\BVedge \Pce \Pef;
\BVedge \Pce \Pde;
\BVedge \Pae \Paa;
\BVedge \Pae \Pbe;
\BVedge \Pbe \Pdd;
\BVedge \Pef \Pea;
\BVedge \Pef \Pde;
\BVedge \Pea \Pdd;
\BVedge \Pea \Pdd;
\BVedge \Pdd \Pde;
\BUedge \Pbe \Pbd;
\BUedge \Pbe \Pcd;
\BUedge \Paa \Pbb;
\BUedge \Paa \Pbc;
\BUedge \Paa \Pkk;
\BUedge \Pea \Pbb;
\BUedge \Pea \Pkk;
\BUedge \Pbc \Pbd;
\BUedge \Pbc \Pcc;
\BUedge \Pkk \Pjj;
\BUedge \Pbb \Pcc;
\BUedge \Pcc \Pjj;
\BUedge \Pjj \Pbd;
\BUedge \Pjj \Pcd;
\BUedge \Pcd \Pbd;
\Aedge \Paa \Pea;
\BMvertW \Pbe \
\BMvertW \Pbe \
\BAvertW \Paa \
\BAvertN \Pea \
\BBvertW \Paf \
\BBvertN \Pbf \
\BBvertN \Pce \
\BBvertW \Pae \
\BBvertS \Pef \
\BBvertE \Pde \
\BBvertE \Pdd \
\BMvertW \Pbe \
\BAvertW \Paa \
\BAvertN \Pea \
\BIvertE \Pbc \
\BIvertE \Pkk \
\BIvertW \Pbb \
\BIvertN \Pcc \
\BIvertE \Pjj \
\BIvertS \Pbd \
\BIvertS \Pcd \
\end{tikzpicture}
~\\
$\cG=\cG_1\cup\cG_2$
\qquad and\qquad
$c(\cG)=200192$
\vspace{-2mm}
\caption{
A non-trivial calligraphic split $(\cG_1,\cG_2)$ for the minimally rigid graph~$\cG$.}
\label{fig:level1}
\end{figure}

\begin{figure}[!ht]
\centering
\setlength{\tabcolsep}{1mm}
\begin{tabular}{ccc}
\\
&$[\cG_1]=(368, 96, 176)$&
\\
\begin{tikzpicture}[scale=0.3,rotate=-90]
\BVedge \Paf \Pbf;
\BVedge \Paf \Pbe;
\BVedge \Paf \Pae;
\BVedge \Pbf \Pae;
\BVedge \Pbf \Pef;
\BVedge \Pce \Pbe;
\BVedge \Pce \Pef;
\BVedge \Pce \Pde;
\BVedge \Pae \Paa;
\BVedge \Pae \Pbe;
\BVedge \Pbe \Pdd;
\BVedge \Pef \Pea;
\BVedge \Pef \Pde;
\BVedge \Pea \Pdd;
\BVedge \Pea \Pdd;
\BVedge \Pdd \Pde;
\BDedge \Paa \Pbe;
\Aedge \Paa \Pea;
\BMvertW \Pbe \
\BAvertW \Paa \
\BAvertN \Pea \
\BBvertW \Paf \
\BBvertN \Pbf \
\BBvertN \Pce \
\BBvertW \Pae \
\BBvertS \Pef \
\BBvertE \Pde \
\BBvertE \Pdd \
\end{tikzpicture}
&
\begin{tikzpicture}[scale=0.3,rotate=-90]
\BVedge \Paf \Pbf;
\BVedge \Paf \Pbe;
\BVedge \Paf \Pae;
\BVedge \Pbf \Pae;
\BVedge \Pbf \Pef;
\BVedge \Pce \Pbe;
\BVedge \Pce \Pef;
\BVedge \Pce \Pde;
\BVedge \Pae \Paa;
\BVedge \Pae \Pbe;
\BVedge \Pbe \Pdd;
\BVedge \Pef \Pea;
\BVedge \Pef \Pde;
\BVedge \Pea \Pdd;
\BVedge \Pea \Pdd;
\BVedge \Pdd \Pde;
\BDedge \Pea \Pbe;
\Aedge \Paa \Pea;
\BMvertW \Pbe \
\BAvertW \Paa \
\BAvertN \Pea \
\BBvertW \Paf \
\BBvertN \Pbf \
\BBvertN \Pce \
\BBvertW \Pae \
\BBvertS \Pef \
\BBvertE \Pde \
\BBvertE \Pdd \
\end{tikzpicture}
&
\begin{tikzpicture}[scale=0.3,rotate=-90]
\BVedge \Paf \Pbf;
\BVedge \Paf \Pbe;
\BVedge \Paf \Pae;
\BVedge \Pbf \Pae;
\BVedge \Pbf \Pef;
\BVedge \Pce \Pbe;
\BVedge \Pce \Pef;
\BVedge \Pce \Pde;
\BVedge \Pae \Paa;
\BVedge \Pae \Pbe;
\BVedge \Pbe \Pdd;
\BVedge \Pef \Pea;
\BVedge \Pef \Pde;
\BVedge \Pea \Pdd;
\BVedge \Pea \Pdd;
\BVedge \Pdd \Pde;
\BDedge \Pii \Pbe;
\BDedge \Pii \Paa;
\BDedge \Pii \Pea;
\Aedge \Paa \Pea;
\BIvertS \Pii v;
\BMvertW \Pbe \
\BAvertW \Paa \
\BAvertN \Pea \
\BBvertW \Paf \
\BBvertN \Pbf \
\BBvertN \Pce \
\BBvertW \Pae \
\BBvertS \Pef \
\BBvertE \Pde \
\BBvertE \Pdd \
\end{tikzpicture}
\\
$\cG_{1L}:=\cG_1\cup \cL$ &
$\cG_{1R}:=\cG_1\cup \cR$ &
$\cG_{1C}:=\cG_1\cup \cC_v$
\\
$c(\cG_{1L})=544$ &
$c(\cG_{1R})=384$ &
$c(\cG_{1L})=1472$
\end{tabular}
\caption{The class of the calligraph $\cG_1$ in
\FIG{level1}.}
\label{fig:level2a}
\end{figure}

\begin{figure}[!ht]
\centering
\setlength{\tabcolsep}{1mm}
\begin{tabular}{ccc}
\begin{tikzpicture}[scale=0.3,rotate=-90]
\BUedge \Paf \Pbf;
\BUedge \Paf \Pbe;
\BUedge \Paf \Pae;
\BUedge \Pbf \Pae;
\BUedge \Pbf \Pef;
\BVedge \Pce \Pbe;
\BVedge \Pce \Pef;
\BVedge \Pce \Pde;
\BUedge \Pae \Paa;
\BUedge \Pae \Pbe;
\BUedge \Pef \Pea;
\BVedge \Pef \Pde;
\BUedge \Paa \Pea;
\BUedge \Pea \Pdd;
\BUedge \Pea \Pdd;
\BVedge \Pdd \Pde;
\BUedge \Paa \Pbe;
\Aedge \Pbe \Pdd;
\BAvertW \Pbe \
\BIvertW \Paa \
\BIvertN \Pea \
\BIvertW \Paf \
\BIvertN \Pbf \
\BBvertN \Pce \
\BIvertW \Pae \
\BMvertS \Pef \
\BBvertE \Pde \
\BAvertE \Pdd \
\end{tikzpicture}
&
\begin{tikzpicture}[scale=0.3,rotate=-90]
\BUedge \Paf \Pbf;
\BUedge \Paf \Pbe;
\BUedge \Paf \Pae;
\BUedge \Pbf \Pae;
\BUedge \Pbf \Pef;
\BVedge \Pce \Pbe;
\BVedge \Pce \Pef;
\BVedge \Pce \Pde;
\BUedge \Pae \Paa;
\BUedge \Pae \Pbe;
\BUedge \Pef \Pea;
\BVedge \Pef \Pde;
\BUedge \Paa \Pea;
\BUedge \Pea \Pdd;
\BUedge \Pea \Pdd;
\BVedge \Pdd \Pde;
\BUedge \Pea \Pbe;
\Aedge \Pbe \Pdd;
\BAvertW \Pbe \
\BIvertW \Paa \
\BIvertN \Pea \
\BIvertW \Paf \
\BIvertN \Pbf \
\BBvertN \Pce \
\BIvertW \Pae \
\BMvertS \Pef \
\BBvertE \Pde \
\BAvertE \Pdd \
\end{tikzpicture}
&
\begin{tikzpicture}[scale=0.3,rotate=-90]
\BUedge \Paf \Pbf;
\BUedge \Paf \Pbe;
\BUedge \Paf \Pae;
\BUedge \Pbf \Pae;
\BUedge \Pbf \Pef;
\BVedge \Pce \Pbe;
\BVedge \Pce \Pef;
\BVedge \Pce \Pde;
\BUedge \Pae \Paa;
\BUedge \Pae \Pbe;
\Aedge \Pbe \Pdd;
\BUedge \Pef \Pea;
\BVedge \Pef \Pde;
\BUedge \Paa \Pea;
\BUedge \Pea \Pdd;
\BUedge \Pea \Pdd;
\BVedge \Pdd \Pde;
\BUedge \Pii \Pbe;
\BUedge \Pii \Paa;
\BUedge \Pii \Pea;
\Aedge \Pbe \Pdd;
\BAvertW \Pbe \
\BIvertW \Paa \
\BIvertN \Pea \
\BIvertW \Paf \
\BIvertN \Pbf \
\BBvertN \Pce \
\BIvertW \Pae \
\BMvertS \Pef \
\BBvertE \Pde \
\BAvertE \Pdd \
\BIvertS \Pii \
\end{tikzpicture}
\\
$\cG_{1L}=\cG_{1L1}\cup\cG_{1L2}$&
$\cG_{1R}=\cR_{1R1}\cup\cG_{1R2}$&
$\cG_{1C}=\cG_{1C1}\cup\cG_{1C2}$
\end{tabular}
\\
\caption{
Non-trivial calligraphic splits for the minimally rigid graphs in \FIG{level2a}.
}
\label{fig:level3a}
\end{figure}

\begin{figure}[!ht]
\centering
\setlength{\tabcolsep}{1mm}
\begin{tabular}{ccc}
&$[\cG_2]=(272,0,0)$&
\\
\begin{tikzpicture}[scale=0.31,rotate=-90]
\BUedge \Pbe \Pbd;
\BUedge \Pbe \Pcd;
\BUedge \Paa \Pbb;
\BUedge \Paa \Pbc;
\BUedge \Paa \Pkk;
\BUedge \Pea \Pbb;
\BUedge \Pea \Pkk;
\BUedge \Pbc \Pbd;
\BUedge \Pbc \Pcc;
\BUedge \Pkk \Pjj;
\BUedge \Pbb \Pcc;
\BUedge \Pcc \Pjj;
\BUedge \Pjj \Pbd;
\BUedge \Pjj \Pcd;
\BUedge \Pcd \Pbd;
\Aedge \Paa \Pea;
\BEedge \Paa \Pbe;
\BMvertW \Pbe \
\BAvertW \Paa \
\BAvertN \Pea \
\BIvertE \Pbc \
\BIvertE \Pkk \
\BIvertW \Pbb \
\BIvertN \Pcc \
\BIvertE \Pjj \
\BIvertS \Pbd \
\BIvertS \Pcd \
\end{tikzpicture}
&
\begin{tikzpicture}[scale=0.31,rotate=-90]
\BUedge \Pbe \Pbd;
\BUedge \Pbe \Pcd;
\BUedge \Paa \Pbb;
\BUedge \Paa \Pbc;
\BUedge \Paa \Pkk;
\BUedge \Pea \Pbb;
\BUedge \Pea \Pkk;
\BUedge \Pbc \Pbd;
\BUedge \Pbc \Pcc;
\BUedge \Pkk \Pjj;
\BUedge \Pbb \Pcc;
\BUedge \Pcc \Pjj;
\BUedge \Pjj \Pbd;
\BUedge \Pjj \Pcd;
\BUedge \Pcd \Pbd;
\BEedge \Pea \Pbe;
\Aedge \Paa \Pea;
\BMvertW \Pbe \
\BAvertW \Paa \
\BAvertN \Pea \
\BIvertE \Pbc \
\BIvertE \Pkk \
\BIvertW \Pbb \
\BIvertN \Pcc \
\BIvertE \Pjj \
\BIvertS \Pbd \
\BIvertS \Pcd \
\end{tikzpicture}
&
\begin{tikzpicture}[scale=0.31,rotate=-90]
\BUedge \Pbe \Pbd;
\BUedge \Pbe \Pcd;
\BUedge \Paa \Pbb;
\BUedge \Paa \Pbc;
\BUedge \Paa \Pkk;
\BUedge \Pea \Pbb;
\BUedge \Pea \Pkk;
\BUedge \Pbc \Pbd;
\BUedge \Pbc \Pcc;
\BUedge \Pkk \Pjj;
\BUedge \Pbb \Pcc;
\BUedge \Pcc \Pjj;
\BUedge \Pjj \Pbd;
\BUedge \Pjj \Pcd;
\BUedge \Pcd \Pbd;
\Aedge \Paa \Pea;
\BEedge \Pac \Pbe;
\BEedge \Pac \Paa;
\BEedge \Pac \Pea;
\BMvertW \Pbe \
\BAvertW \Paa \
\BAvertN \Pea \
\BIvertE \Pbc \
\BIvertE \Pkk \
\BIvertW \Pbb \
\BIvertN \Pcc \
\BIvertE \Pjj \
\BIvertS \Pbd \
\BIvertS \Pcd \
\BBvertS \Pac \
\end{tikzpicture}
\\
$\cG_{2L}:=\cG_2\cup \cL$ &
$\cG_{2R}:=\cG_2\cup \cR$ &
$\cG_{2C}:=\cG_2\cup \cC_v$
\\
$c(\cG_{2L})=544$ & $c(\cG_{2R})=544$ & $c(\cG_{2C})=1088$
\end{tabular}
\caption{The class of the calligraph graph $\cG_2$ in \FIG{level1}.}
\label{fig:level2b}
\end{figure}

\begin{figure}[!ht]
\centering
\setlength{\tabcolsep}{1mm}
\begin{tabular}{ccc}
\begin{tikzpicture}[scale=0.31,rotate=-90]
\BUedge \Pbe \Pbd;
\BUedge \Pbe \Pcd;
\BVedge \Paa \Pea;
\BVedge \Paa \Pbb;
\BUedge \Paa \Pbc;
\BVedge \Paa \Pkk;
\BVedge \Pea \Pbb;
\BVedge \Pea \Pkk;
\BUedge \Pbc \Pbd;
\BUedge \Pbc \Pcc;
\BVedge \Pkk \Pjj;
\BVedge \Pbb \Pcc;
\BUedge \Pjj \Pbd;
\BUedge \Pjj \Pcd;
\BUedge \Pcd \Pbd;
\BUedge \Paa \Pbe;
\Aedge \Pcc \Pjj;
\BIvertW \Pbe \
\BMvertW \Paa \
\BBvertN \Pea \
\BIvertE \Pbc \
\BBvertE \Pkk \
\BBvertW \Pbb \
\BAvertN \Pcc \
\BAvertE \Pjj \
\BIvertS \Pbd \
\BIvertS \Pcd \
\end{tikzpicture}
&
\begin{tikzpicture}[scale=0.31,rotate=-90]
\BVedge \Pbe \Pbd;
\BVedge \Pbe \Pcd;
\BUedge \Paa \Pea;
\BUedge \Paa \Pbb;
\BUedge \Paa \Pbc;
\BUedge \Paa \Pkk;
\BUedge \Pea \Pbb;
\BUedge \Pea \Pkk;
\BUedge \Pbc \Pbd;
\BUedge \Pbc \Pcc;
\BUedge \Pkk \Pjj;
\BUedge \Pbb \Pcc;
\BUedge \Pcc \Pjj;
\BVedge \Pjj \Pcd;
\BVedge \Pcd \Pbd;
\BVedge \Pea \Pbe;
\Aedge \Pjj \Pbd;
\BBvertW \Pbe \
\BIvertW \Paa \
\BMvertN \Pea \
\BIvertE \Pbc \
\BIvertE \Pkk \
\BIvertW \Pbb \
\BIvertN \Pcc \
\BAvertE \Pjj \
\BAvertS \Pbd \
\BBvertS \Pcd \
\end{tikzpicture}
&
\begin{tikzpicture}[scale=0.31,rotate=-90]
\BVedge \Pbe \Pbd;
\BVedge \Pbe \Pcd;
\BUedge \Paa \Pea;
\BUedge \Paa \Pbb;
\BUedge \Paa \Pbc;
\BUedge \Paa \Pkk;
\BUedge \Pea \Pbb;
\BUedge \Pea \Pkk;
\BUedge \Pbc \Pbd;
\BUedge \Pbc \Pcc;
\BUedge \Pkk \Pjj;
\BUedge \Pbb \Pcc;
\BUedge \Pcc \Pjj;
\BVedge \Pjj \Pcd;
\BVedge \Pcd \Pbd;
\BVedge \Pac \Pbe;
\BUedge \Pac \Paa;
\BUedge \Pac \Pea;
\Aedge \Pjj \Pbd;
\BMvertS \Pac \
\BBvertW \Pbe \
\BIvertW \Paa \
\BIvertN \Pea \
\BIvertE \Pbc \
\BIvertE \Pkk \
\BIvertW \Pbb \
\BIvertN \Pcc \
\BAvertE \Pjj \
\BAvertS \Pbd \
\BBvertS \Pcd \
\end{tikzpicture}
\\
$\cG_{2L}=\cG_{2L1}\cup\cG_{2L2}$ &
$\cG_{2R}=\cG_{2R1}\cup\cG_{2R2}$ &
$\cG_{2C}=\cG_{2C1}\cup\cG_{2C2}$
\end{tabular}
\caption{
Non-trivial calligraphic splits for the minimally rigid graphs in \FIG{level2b}.
}
\label{fig:level3b}
\end{figure}

\begin{figure}[!ht]
\centering
\setlength{\tabcolsep}{1mm}
\begin{tabular}{ccc}
&$[\cG_{1L1}]=(56,8,24)$&
\\
\begin{tikzpicture}[scale=0.3,rotate=-90]
\BUedge \Paf \Pbf;
\BUedge \Paf \Pbe;
\BUedge \Paf \Pae;
\BUedge \Pbf \Pae;
\BUedge \Pbf \Pef;
\BUedge \Pae \Paa;
\BUedge \Pae \Pbe;
\BUedge \Pef \Pea;
\BUedge \Paa \Pea;
\BUedge \Pea \Pdd;
\BUedge \Pea \Pdd;
\BEedge \Pbe \Pef;
\BUedge \Paa \Pbe;
\Aedge \Pbe \Pdd;
\BAvertW \Pbe \
\BIvertW \Paa \
\BIvertN \Pea \
\BIvertW \Paf \
\BIvertN \Pbf \
\BIvertW \Pae \
\BMvertS \Pef \
\BAvertE \Pdd \
\end{tikzpicture}
&
\begin{tikzpicture}[scale=0.3,rotate=-90]
\BUedge \Paf \Pbf;
\BUedge \Paf \Pbe;
\BUedge \Paf \Pae;
\BUedge \Pbf \Pae;
\BUedge \Pbf \Pef;
\BUedge \Pae \Paa;
\BUedge \Pae \Pbe;
\BUedge \Pef \Pea;
\BUedge \Paa \Pea;
\BUedge \Pea \Pdd;
\BUedge \Pea \Pdd;
\BUedge \Paa \Pbe;
\BEedge \Pdd \Pef;
\Aedge \Pbe \Pdd;
\BAvertW \Pbe \
\BIvertW \Paa \
\BIvertN \Pea \
\BIvertW \Paf \
\BIvertN \Pbf \
\BIvertW \Pae \
\BMvertS \Pef \
\BAvertE \Pdd \
\end{tikzpicture}
&
\begin{tikzpicture}[scale=0.3,rotate=-90]
\BUedge \Paf \Pbf;
\BUedge \Paf \Pbe;
\BUedge \Paf \Pae;
\BUedge \Pbf \Pae;
\BUedge \Pbf \Pef;
\BUedge \Pae \Paa;
\BUedge \Pae \Pbe;
\BUedge \Pef \Pea;
\BUedge \Paa \Pea;
\BUedge \Pea \Pdd;
\BUedge \Pea \Pdd;
\BUedge \Paa \Pbe;
\BEedge \Pmm \Pef;
\BEedge \Pmm \Pbe;
\BEedge \Pmm \Pdd;
\Aedge \Pbe \Pdd;
\BAvertW \Pbe \
\BIvertW \Paa \
\BIvertN \Pea \
\BIvertW \Paf \
\BIvertN \Pbf \
\BIvertW \Pae \
\BMvertS \Pef \
\BAvertE \Pdd \
\BBvertW \Pmm \
\end{tikzpicture}
\\
$\cG_{1L1L}:=\cG_{1L1}\cup\cL$ & $\cG_{1L1R}:=\cG_{1L1}\cup\cR$ & $\cG_{1L1C}:=\cG_{1L1}\cup\cC_v$
\\
$\cc(\cG_{1L1L})=64$ & $\cc(\cG_{1L1R})=96$ & $\cc(\cG_{1L1C})=224$
\\\hline
&$[\cG_{1R1}]=(32,0,0)$&
\\
\begin{tikzpicture}[scale=0.3,rotate=-90]
\BUedge \Paf \Pbf;
\BUedge \Paf \Pbe;
\BUedge \Paf \Pae;
\BUedge \Pbf \Pae;
\BUedge \Pbf \Pef;
\BUedge \Pae \Paa;
\BUedge \Pae \Pbe;
\BUedge \Pef \Pea;
\BUedge \Paa \Pea;
\BUedge \Pea \Pdd;
\BUedge \Pea \Pdd;
\BEedge \Pbe \Pef;
\BUedge \Pea \Pbe;
\Aedge \Pbe \Pdd;
\BAvertW \Pbe \
\BIvertW \Paa \
\BIvertN \Pea \
\BIvertW \Paf \
\BIvertN \Pbf \
\BIvertW \Pae \
\BMvertS \Pef \
\BAvertE \Pdd \
\end{tikzpicture}
&
\begin{tikzpicture}[scale=0.3,rotate=-90]
\BUedge \Paf \Pbf;
\BUedge \Paf \Pbe;
\BUedge \Paf \Pae;
\BUedge \Pbf \Pae;
\BUedge \Pbf \Pef;
\BUedge \Pae \Paa;
\BUedge \Pae \Pbe;
\BUedge \Pef \Pea;
\BUedge \Paa \Pea;
\BUedge \Pea \Pdd;
\BUedge \Pea \Pdd;
\BUedge \Pea \Pbe;
\BEedge \Pdd \Pef;
\Aedge \Pbe \Pdd;
\BAvertW \Pbe \
\BIvertW \Paa \
\BIvertN \Pea \
\BIvertW \Paf \
\BIvertN \Pbf \
\BIvertW \Pae \
\BMvertS \Pef \
\BAvertE \Pdd \
\end{tikzpicture}
&
\begin{tikzpicture}[scale=0.3,rotate=-90]
\BUedge \Paf \Pbf;
\BUedge \Paf \Pbe;
\BUedge \Paf \Pae;
\BUedge \Pbf \Pae;
\BUedge \Pbf \Pef;
\BUedge \Pae \Paa;
\BUedge \Pae \Pbe;
\BUedge \Pef \Pea;
\BUedge \Paa \Pea;
\BUedge \Pea \Pdd;
\BUedge \Pea \Pdd;
\BUedge \Pea \Pbe;
\BEedge \Pmm \Pef;
\BEedge \Pmm \Pbe;
\BEedge \Pmm \Pdd;
\Aedge \Pbe \Pdd;
\BAvertW \Pbe \
\BIvertW \Paa \
\BIvertN \Pea \
\BIvertW \Paf \
\BIvertN \Pbf \
\BIvertW \Pae \
\BMvertS \Pef \
\BAvertE \Pdd \
\BBvertW \Pmm \
\end{tikzpicture}
\\
$\cG_{1R1L}:=\cG_{1R1}\cup\cL$ & $\cG_{1R1R}:=\cG_{1R1}\cup\cR$ & $\cG_{1R1C}:=\cG_{1R1}\cup\cC_v$
\\
$\cc(\cG_{1R1L})=64$ & $\cc(\cG_{1R1R})=64$ & $\cc(\cG_{1R1C})=128$
\\\hline
&$[\cG_{1C1}]=(144, 16, 48)$&
\\
\begin{tikzpicture}[scale=0.3,rotate=-90]
\BUedge \Paf \Pbf;
\BUedge \Paf \Pbe;
\BUedge \Paf \Pae;
\BUedge \Pbf \Pae;
\BUedge \Pbf \Pef;
\BUedge \Pae \Paa;
\BUedge \Pae \Pbe;
\BUedge \Pef \Pea;
\BUedge \Paa \Pea;
\BUedge \Pea \Pdd;
\BUedge \Pea \Pdd;
\BEedge \Pbe \Pef;
\BUedge \Pii \Pbe;
\BUedge \Pii \Paa;
\BUedge \Pii \Pea;
\Aedge \Pbe \Pdd;
\BAvertW \Pbe \
\BIvertW \Paa \
\BIvertN \Pea \
\BIvertW \Paf \
\BIvertN \Pbf \
\BIvertW \Pae \
\BMvertS \Pef \
\BAvertE \Pdd \
\BIvertW \Pii \
\end{tikzpicture}
&
\begin{tikzpicture}[scale=0.3,rotate=-90]
\BUedge \Paf \Pbf;
\BUedge \Paf \Pbe;
\BUedge \Paf \Pae;
\BUedge \Pbf \Pae;
\BUedge \Pbf \Pef;
\BUedge \Pae \Paa;
\BUedge \Pae \Pbe;
\BUedge \Pef \Pea;
\BUedge \Paa \Pea;
\BUedge \Pea \Pdd;
\BUedge \Pea \Pdd;
\BEedge \Pdd \Pef;
\BUedge \Pii \Pbe;
\BUedge \Pii \Paa;
\BUedge \Pii \Pea;
\Aedge \Pbe \Pdd;
\BAvertW \Pbe \
\BIvertW \Paa \
\BIvertN \Pea \
\BIvertW \Paf \
\BIvertN \Pbf \
\BIvertW \Pae \
\BMvertS \Pef \
\BAvertE \Pdd \
\BIvertW \Pii \
\end{tikzpicture}
&
\begin{tikzpicture}[scale=0.3,rotate=-90]
\BUedge \Paf \Pbf;
\BUedge \Paf \Pbe;
\BUedge \Paf \Pae;
\BUedge \Pbf \Pae;
\BUedge \Pbf \Pef;
\BUedge \Pae \Paa;
\BUedge \Pae \Pbe;
\BUedge \Pef \Pea;
\BUedge \Paa \Pea;
\BUedge \Pea \Pdd;
\BUedge \Pea \Pdd;
\BEedge \Pmm \Pef;
\BEedge \Pmm \Pbe;
\BEedge \Pmm \Pdd;
\BUedge \Pii \Pbe;
\BUedge \Pii \Paa;
\BUedge \Pii \Pea;
\Aedge \Pbe \Pdd;
\BAvertW \Pbe \
\BIvertW \Paa \
\BIvertN \Pea \
\BIvertW \Paf \
\BIvertN \Pbf \
\BIvertW \Pae \
\BMvertS \Pef \
\BAvertE \Pdd \
\BIvertW \Pii \
\BBvertW \Pmm \
\end{tikzpicture}
\\
$\cG_{1C1L}:=\cG_{1C1}\cup\cL$ & $\cG_{1C1R}:=\cG_{1C1}\cup\cR$ & $\cG_{1C1C}:=\cG_{1C1}\cup\cC_v$
\\
$\cc(\cG_{1C1L})=192$ & $\cc(\cG_{1C1R})=256$ & $\cc(\cG_{1C1C})=576$
\end{tabular}
\vspace{-1mm} 
\caption{
The classes for three of the six calligraphs in \FIG{level3a}.
For the remaining calligraphs we have
$[\cG_{1L2}]=[\cG_{1R2}]=[\cG_{1C2}]=[\cH]=(6,2,2)$ (see \SEC{H}).
}
\label{fig:level4a}
\end{figure}

\begin{figure}[!ht]
\centering
\setlength{\tabcolsep}{1mm}
\begin{tabular}{ccc}
&$[\cG_{2L1}]=(28,4,12)$&
\\
\begin{tikzpicture}[scale=0.31,rotate=-90]
\BUedge \Pbe \Pbd ;
\BUedge \Pbe \Pcd ;
\BUedge \Pab \Pbc ;
\BUedge \Pbc \Pbd ;
\BUedge \Pbc \Pcc ;
\BUedge \Pjj \Pbd ;
\BUedge \Pjj \Pcd ;
\BUedge \Pcd \Pbd ;
\BUedge \Pab \Pbe ;
\BEedge \Pcc \Pab ;
\Aedge \Pcc \Pjj ;
\BIvertW \Pbe \
\BMvertW \Pab \
\BIvertE \Pbc \
\BAvertN \Pcc \
\BAvertE \Pjj \
\BIvertS \Pbd \
\BIvertS \Pcd \
\end{tikzpicture}
&
\begin{tikzpicture}[scale=0.31,rotate=-90]
\BUedge \Pbe \Pbd ;
\BUedge \Pbe \Pcd ;
\BUedge \Pab \Pbc ;
\BUedge \Pbc \Pbd ;
\BUedge \Pbc \Pcc ;
\BUedge \Pjj \Pbd ;
\BUedge \Pjj \Pcd ;
\BUedge \Pcd \Pbd ;
\BUedge \Pab \Pbe ;
\BEedge \Pjj \Pab ;
\Aedge \Pcc \Pjj ;
\BIvertW \Pbe \
\BMvertW \Pab \
\BIvertE \Pbc \
\BAvertN \Pcc \
\BAvertE \Pjj \
\BIvertS \Pbd \
\BIvertS \Pcd \
\end{tikzpicture}
&
\begin{tikzpicture}[scale=0.31,rotate=-90]
\BUedge \Pbe \Pbd ;
\BUedge \Pbe \Pcd ;
\BUedge \Pab \Pbc ;
\BUedge \Pbc \Pbd ;
\BUedge \Pbc \Pcc ;
\BUedge \Pjj \Pbd ;
\BUedge \Pjj \Pcd ;
\BUedge \Pcd \Pbd ;
\BUedge \Pab \Pbe ;
\BEedge \Pdb \Pab ;
\BEedge \Pdb \Pcc ;
\BEedge \Pdb \Pjj ;
\Aedge \Pcc \Pjj ;
\BBvertW \Pdb \
\BIvertW \Pbe \
\BMvertW \Pab \
\BIvertE \Pbc \
\BAvertN \Pcc \
\BAvertE \Pjj \
\BIvertS \Pbd \
\BIvertS \Pcd \
\end{tikzpicture}
\\
$\cG_{2L1L}:=\cG_{2L1}\cup\cL$ &
$\cG_{2L1R}:=\cG_{2L1}\cup\cR$ &
$\cG_{2L1C}:=\cG_{2L1}\cup\cC_v$
\\
$\cc(\cG_{2L1L})=48$ &
$\cc(\cG_{2L1R})=32$ &
$\cc(\cG_{2L1C})=112$
\\\hline
&$[\cG_{2L2}]=2\cdot[\cH]=(12,4,4)$&
\\
\begin{tikzpicture}[scale=0.31,rotate=-90]
\BVedge \Pba \Pea ;
\BVedge \Pba \Pbb ;
\BVedge \Pba \Pkk ;
\BVedge \Pea \Pbb ;
\BVedge \Pea \Pkk ;
\BVedge \Pkk \Pjj ;
\BVedge \Pbb \Pcd ;
\BDedge \Pcd \Pba ;
\Aedge \Pcd \Pjj ;
\BMvertW \Pba \
\BBvertN \Pea \
\BBvertE \Pkk \
\BBvertW \Pbb \
\BAvertN \Pcd \
\BAvertE \Pjj \
\end{tikzpicture}
&
\begin{tikzpicture}[scale=0.31,rotate=-90]
\BVedge \Pba \Pea ;
\BVedge \Pba \Pbb ;
\BVedge \Pba \Pkk ;
\BVedge \Pea \Pbb ;
\BVedge \Pea \Pkk ;
\BVedge \Pkk \Pjj ;
\BVedge \Pbb \Pcd ;
\BDedge \Pjj \Pba ;
\Aedge \Pcd \Pjj ;
\BMvertW \Pba \
\BBvertN \Pea \
\BBvertE \Pkk \
\BBvertW \Pbb \
\BAvertN \Pcd \
\BAvertE \Pjj \
\end{tikzpicture}
&
\begin{tikzpicture}[scale=0.31,rotate=-90]
\BVedge \Pba \Pea ;
\BVedge \Pba \Pbb ;
\BVedge \Pba \Pkk ;
\BVedge \Pea \Pbb ;
\BVedge \Pea \Pkk ;
\BVedge \Pkk \Pjj ;
\BVedge \Pbb \Pcd ;
\BDedge \Pcc \Pba ;
\BDedge \Pcc \Pcd ;
\BDedge \Pcc \Pjj ;
\Aedge \Pcd \Pjj ;
\BIvertW \Pcc \
\BMvertW \Pba \
\BBvertN \Pea \
\BBvertE \Pkk \
\BBvertW \Pbb \
\BAvertN \Pcd \
\BAvertE \Pjj \
\end{tikzpicture}
\\
$\cG_{2L2L}:=\cG_{2L2}\cup\cL$ &
$\cG_{2L2R}:=\cG_{2L2}\cup\cR$ &
$\cG_{2L2C}:=\cG_{2L2}\cup\cC_v$
\\
$\cc(\cG_{2L2L})=16$ &
$\cc(\cG_{2L2R})=16$ &
$\cc(\cG_{2L2C})=48$
\\\hline
&$[\cG_{2R1}]=(68, 12, 36)$&
\\
\begin{tikzpicture}[scale=0.31,rotate=-90]
\BUedge \Paa \Pca ;
\BUedge \Paa \Pbb ;
\BUedge \Paa \Pbc ;
\BUedge \Paa \Pdb ;
\BUedge \Pca \Pbb ;
\BUedge \Pca \Pdb ;
\BUedge \Pbc \Pbd ;
\BUedge \Pbc \Pcc ;
\BUedge \Pdb \Pjj ;
\BUedge \Pbb \Pcc ;
\BUedge \Pcc \Pjj ;
\BEedge \Pbd \Pca ;
\Aedge \Pjj \Pbd ;
\BIvertW \Pdb \
\BIvertW \Paa \
\BMvertN \Pca \
\BIvertE \Pbc \
\BIvertW \Pbb \
\BIvertN \Pcc \
\BAvertE \Pjj \
\BAvertS \Pbd \
\end{tikzpicture}
&
\begin{tikzpicture}[scale=0.31,rotate=-90]
\BUedge \Paa \Pca ;
\BUedge \Paa \Pbb ;
\BUedge \Paa \Pbc ;
\BUedge \Paa \Pdb ;
\BUedge \Pca \Pbb ;
\BUedge \Pca \Pdb ;
\BUedge \Pbc \Pbd ;
\BUedge \Pbc \Pcc ;
\BUedge \Pdb \Pjj ;
\BUedge \Pbb \Pcc ;
\BUedge \Pcc \Pjj ;
\BEedge \Pjj \Pca ;
\Aedge \Pjj \Pbd ;
\BIvertW \Pdb \
\BIvertW \Paa \
\BMvertN \Pca \
\BIvertE \Pbc \
\BIvertW \Pbb \
\BIvertN \Pcc \
\BAvertE \Pjj \
\BAvertS \Pbd \
\end{tikzpicture}
&
\begin{tikzpicture}[scale=0.31,rotate=-90]
\BUedge \Paa \Pca ;
\BUedge \Paa \Pbb ;
\BUedge \Paa \Pbc ;
\BUedge \Paa \Pdb ;
\BUedge \Pca \Pbb ;
\BUedge \Pca \Pdb ;
\BUedge \Pbc \Pbd ;
\BUedge \Pbc \Pcc ;
\BUedge \Pbd \Pjj ;
\BUedge \Pbb \Pcc ;
\BUedge \Pcc \Pjj ;
\BUedge \Pdb \Pjj ;
\BEedge \Pcb \Pca ;
\BEedge \Pcb \Pbd ;
\BEedge \Pcb \Pjj ;
\Aedge \Pjj \Pbd ;
\BIvertW \Pdb \
\BBvertS \Pcb \
\BIvertW \Paa \
\BMvertN \Pca \
\BIvertE \Pbc \
\BIvertW \Pbb \
\BIvertN \Pcc \
\BAvertE \Pjj \
\BAvertS \Pbd \
\end{tikzpicture}
\\
$\cG_{2R1L}:=\cG_{2R1}\cup\cL$ &
$\cG_{2R1R}:=\cG_{2R1}\cup\cR$ &
$\cG_{2R1C}:=\cG_{2R1}\cup\cC_v$
\\
$\cc(\cG_{2R1L})=112$ &
$\cc(\cG_{2R1R})=64$ &
$\cc(\cG_{2R1C})=272$
\\\hline
&$[\cG_{2C1}]=(136, 24, 72)$&
\\
\begin{tikzpicture}[scale=0.31,rotate=-90]
\BUedge \Paa \Pea ;
\BUedge \Paa \Pbb ;
\BUedge \Paa \Pbc ;
\BUedge \Paa \Pkk ;
\BUedge \Pea \Pbb ;
\BUedge \Pea \Pkk ;
\BUedge \Pbc \Pbd ;
\BUedge \Pbc \Pcc ;
\BUedge \Pkk \Pjj ;
\BUedge \Pbb \Pcc ;
\BUedge \Pcc \Pjj ;
\BUedge \Pac \Paa ;
\BUedge \Pac \Pea ;
\BEedge \Pac \Pbd ;
\Aedge \Pjj \Pbd ;
\BMvertS \Pac \
\BIvertW \Paa \
\BIvertN \Pea \
\BIvertE \Pbc \
\BIvertE \Pkk \
\BIvertW \Pbb \
\BIvertN \Pcc \
\BAvertE \Pjj \
\BAvertS \Pbd \
\end{tikzpicture}
&
\begin{tikzpicture}[scale=0.31,rotate=-90]
\BUedge \Paa \Pea ;
\BUedge \Paa \Pbb ;
\BUedge \Paa \Pbc ;
\BUedge \Paa \Pkk ;
\BUedge \Pea \Pbb ;
\BUedge \Pea \Pkk ;
\BUedge \Pbc \Pbd ;
\BUedge \Pbc \Pcc ;
\BUedge \Pkk \Pjj ;
\BUedge \Pbb \Pcc ;
\BUedge \Pcc \Pjj ;
\BUedge \Pac \Paa ;
\BUedge \Pac \Pea ;
\BEedge \Pac \Pjj ;
\Aedge \Pjj \Pbd ;
\BMvertS \Pac \
\BIvertW \Paa \
\BIvertN \Pea \
\BIvertE \Pbc \
\BIvertE \Pkk \
\BIvertW \Pbb \
\BIvertN \Pcc \
\BAvertE \Pjj \
\BAvertS \Pbd \
\end{tikzpicture}
&
\begin{tikzpicture}[scale=0.31,rotate=-90]
\BUedge \Paa \Pea ;
\BUedge \Paa \Pbb ;
\BUedge \Paa \Pbc ;
\BUedge \Paa \Pkk ;
\BUedge \Pea \Pbb ;
\BUedge \Pea \Pkk ;
\BUedge \Pbc \Pbd ;
\BUedge \Pbc \Pcc ;
\BUedge \Pkk \Pjj ;
\BUedge \Pbb \Pcc ;
\BUedge \Pcc \Pjj ;
\BUedge \Pac \Paa ;
\BUedge \Pac \Pea ;
\BEedge \Pad \Pac ;
\BEedge \Pad \Pbd ;
\BEedge \Pad \Pjj ;
\Aedge \Pjj \Pbd ;
\BBvertS \Pad \
\BMvertS \Pac \
\BIvertW \Paa \
\BIvertN \Pea \
\BIvertE \Pbc \
\BIvertE \Pkk \
\BIvertW \Pbb \
\BIvertN \Pcc \
\BAvertE \Pjj \
\BAvertS \Pbd \
\end{tikzpicture}
\\
$\cG_{2C1L}:=\cG_{2C1}\cup\cL$ &
$\cG_{2C1R}:=\cG_{2C1}\cup\cR$ &
$\cG_{2C1C}:=\cG_{2C1}\cup\cC_v$
\\
$\cc(\cG_{2C1L})=224$ &
$\cc(\cG_{2C1R})=128$ &
$\cc(\cG_{2C1C})=544$
\end{tabular}
\caption{
The classes of four of the six calligraphs in \FIG{level3b}.
For the remaining two calligraphs we have
$[\cG_{2R2}]=[\cG_{2C2}]=2\cdot [\cC_v] = (4,0,0)$.}
\label{fig:level4b}
\end{figure}

\clearpage
\section{The proof for \PRP{vert}}
\label{sec:lines}

The goal of this section is to prove \PRP{vert} in \SEC{proof}.
We associate to each edge of a calligraph a product of two linear polynomials.
Thus, the zero set of these polynomials form a union of linear spaces.
\PRP{vert} lists all possible linear projections of such a linear space.
Algebraically, this projection corresponds to Gaussian elimination,
which in turn can be described in terms of a procedure that
depends on the combinatorics of the calligraph.
We clarify this procedure with examples, but for this
we need to introduce some graph theoretic terminology.

Suppose that $\cG$ is a calligraph and recall that
$\V=\vv(\cG)\setminus\{1,2\}$ and $\E=\ee(\cG)\setminus\{\{1,2\}\}$.

A \df{walk} is defined as a sequence of vertices $\rho=(\rho_0,...,\rho_r)$
\st $\{\rho_i,\rho_{i+1}\}\in \E$ for all $0\leq i<r$.
We define $\start(\rho):=\rho_0$ and $\wend(\rho):=\rho_r$.
If all vertices are single digits, then we write $\rho$ as $\rho_0\rho_1\rho_2\cdots\rho_r$.
For example, the walk $(0,3,4)$ shall be written as~$034$.
We call the walk~$\rho$ \df{western}, \df{eastern} or \df{round}
if $\bigl(\start(\rho),\wend(\rho)\bigr)$
is equal to $(0,1)$, $(0,2)$ and $(0,0)$, \resp.
We call $\rho$ a \df{route} if $\{\rho_0,\ldots,\rho_r\}=\V$ and $\start(\rho)=0$.
For example, if $\cG$ is the calligraph in \FIG{route}, then both
$(0,3,4,5,6)$ and $(0,6,5,4,6,3)$ are routes.

Suppose that $\Lambda$ is a finite set of walks
and let $\Lambda_v:=\set{\rho\in \Lambda}{v\in\{\start(\rho),\wend(\rho)\}}$ for~$v\in\vv(\cG)$.
We call $\Lambda$ \df{initial} if
$|\Lambda|=|\E|$ and for all $\{i,j\}\in \E$ either $(i,j)\in \Lambda$ or $(j,i)\in \Lambda$.
For example, if $\cG$ is the calligraph in \FIG{route}, then
the set of walks $\{03,06,32,34,36,41,45,46,56\}$ is initial.

Next, we define a \df{concatenation} of two walks.
If $\rho=(\rho_0,\ldots,\rho_r)$ and $\rho'=(\rho'_0,...,\rho'_s)$ are walks,
then $\rho\odot\rho':=(\rho_0,\ldots,\rho_{r-1},\rho'_0,...,\rho'_s)$
and $-\rho:=(\rho_r,\ldots,\rho_0)$.
If $\rho,\rho'\in \Lambda_v$,
then $\rho+\rho'$ is defined as the first element in the following tuple that is a walk:
\[
(\rho\odot \rho',~\rho\odot-\rho',~-\rho\odot\rho',~-\rho\odot-\rho').
\]
For example, $034+45=034+54=430+45=430+54=0345$.

We define $\rho>\rho'$ if $r>s$ or if $r=s$, then
$(\rho_0,\ldots,\rho_r)>_{\text{lex}}(\rho'_0,\ldots,\rho'_r)$ according to the lexicographic ordering.
Let $\max(\Lambda)$
be the unique maximal element in the finite set of walks~$\Lambda$ \wrt the total strict order~$>$.
We define
\[
\Delta_v(\Lambda):=(\Lambda\setminus \Lambda_v)\cup\set{\rho+\rho'}{
\rho=\max(\Lambda_v\cap \Lambda_0) \text{ and } \rho'\in \Lambda_v\setminus\{\rho\}}.
\]
For example, if $\cG$ is as in \FIG{route}
and $\Lambda=\{06,032,036,0341,0346,03456\}$,
then
$\Delta_6(\Lambda)=\{032,0341\} \cup \{034560,0345630, 03456430\}$
is obtained as follows:
$\Lambda\setminus \Lambda_6=\{032,0341\}$,  $03456=\max(\Lambda_6\cap \Lambda_0)$ and
$\Lambda_6\setminus\{03456\}=\{06,036,0346\}$
so that $03456+06=034560$, $03456+036=0345630$ and $03456+0346=03456430$
are elements of~$\Delta_6(\Lambda)$.

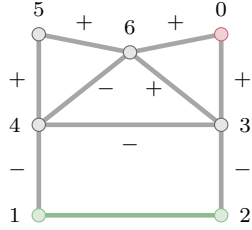
\begin{figure}[!ht]
\centering
\setlength{\tabcolsep}{8mm}
\begin{tabular}{ccc}
\begin{tikzpicture}[xscale=0.8,yscale=1.2]
\MedgeW \pl \pL;
\PedgeW \pL \pA;
\PedgeN \pA \pM;
\PedgeN \pM \pB;
\PedgeE \pB \pR;
\MedgeE \pR \pr;
\MedgeS \pL \pR;
\MedgeE \pL \pM;
\PedgeW \pR \pM;
\Aedge \pl \pr;
\AvertW \pl 1;
\AvertE \pr 2;
\MvertN \pB 0;
\IvertE \pR 3;
\IvertW \pL 4;
\IvertN \pA 5;
\IvertN \pM 6;
\end{tikzpicture}
\end{tabular}
\caption{A calligraph together with a sign labeling.}
\label{fig:route}
\end{figure}
\begin{example}
\label{exm:route}
Suppose that $\cG$ is defined as the calligraph in \FIG{route}
(in this example we may ignore the sign labeling for the edges)
and let
\[
\Lambda^0:=\{03,06,32,34,36,41,45,46,56\}
\]
be an initial set of walks.
We consider the following sets in terms of a union that
comes from the definition of~$\Delta_v$ for $v$ in the route $(0,3,4,5,6)$ \st $v\neq 0$:
\[
\setlength{\arraycolsep}{1mm}
\begin{array}{lclcrcl}
\Lambda^3 &:=&\Delta_3(\Lambda^0)  & = & \{06,41,45,46,56\} & \cup & \{032,034,036\}. \\
\Lambda^4 &:=&\Delta_4(\Lambda^3)  & = & \{06,56,032,036\} & \cup & \{0341,0345,0346\}. \\
\Lambda^5 &:=&\Delta_5(\Lambda^4)  & = & \{06,032,036,0341,0346\} & \cup & \{03456\}. \\
\Lambda^6 &:=&\Delta_6(\Lambda^5)  & = & \{032,0341\} & \cup & \{034560,0345630, 03456430\}. \\
\end{array}
\]
We observe that each walk in $\Lambda^6$ is either western, eastern or round.
\END
\end{example}

\begin{lemma}
\label{lem:wed}
If $\Lambda$ is an initial set of walks and $(v_0,\ldots,v_n)$ is a route,
then each walk in
$\Delta_{v_n} \circ\ldots\circ \Delta_{v_2} \circ\Delta_{v_1}(\Lambda)$
is either western, eastern or round.
\end{lemma}

\begin{proof}
We verified the assertion in \EXM{route} by
computing $\Delta_6\circ\Delta_5\circ\Delta_4\circ\Delta_3(\Lambda)$
via a deterministic procedure that halts.
This procedure generalizes to any calligraph and route, and
it is straightforward to see that each walk in the output must
be either western, eastern or round.
\end{proof}

In what follows we assign to each walk a linear polynomial.
The concatenation of walks corresponds to linear combinations of these polynomials
\st common variables cancel out.
We consider the ideal generated by the polynomials associated to walks in an initial set.
The zero set of this ideal corresponds a linear space $W_i$ in \PRP{vert} for some $1\leq i\leq r$.
In \PRP{b0} we list all possible ideals that are obtained
after eliminating all but one variable.
The zero sets of the resulting ideals corresponds to the linear projection~$\pi(W_i)$.
Thus, \PRP{b0} translates to a proof of \PRP{vert}.

Let
$\bA:=\C[a_i,b_i:i\in \V]$ and $\bB:=\C[b_i:i\in \V]$.
We denote the set of variables for these rings by
$\Upsilon_\bA:=\set{a_i,b_i}{i\in \V}$
and
$\Upsilon_\bB:=\set{b_i}{i\in \V}$.
If $I\subset \V$, then
\[
\bhA(I):=\C[\Upsilon_\bA\setminus\{a_i\}_{i\in I}].
\]
We use the following short hand notation:
\[
a_{ij}:=a_i-a_j \qquad\text{and}\qquad b_{ij}:=b_i-b_j.
\]
A \df{sign labeling} is defined as a
map $\tau\c \E\to \{1,-1\}$ \st $\tau(e)=1$ for all edges~$e\in \E$ \st $0\in e$.
Let $\bar \E := \{(i,j): \{i,j\}\in\E\}$ and let $\varphi_\tau\c \bar \E\to \bA$ be defined as
\[
\varphi_\tau(e):=a_{ij}+\tau(e)\cdot b_{ij},
\]
where $a_0:=0$, $(a_1,b_1):=(0,0)$ and $(a_2,b_2):=(-1,0)$.
The \df{$\tau$-polynomial} of a walk $\rho=(\rho_0,\rho_1,...,\rho_r)$  is defined as
\[
\tau_\rho:=\varphi_\tau(\{\rho_0,\rho_1\})+\varphi_\tau(\{\rho_1,\rho_2\})+\cdots +\varphi_\tau(\{\rho_{r-1},\rho_r\}),
\]
Notice that $\tau_\rho\in\bhA(\rho_1,\ldots,\rho_{r-1})$
by construction.
We call $\tau_\rho$ \df{western}, \df{eastern} or \df{round} if $\rho$
is as such.
The \df{$\tau$-ideal} in the ring $\bA$ of a finite set of walks~$\Lambda$ is defined as
\[
\bI_\tau(\Lambda):=\bas{\tau_\rho: \rho\in \Lambda}.
\]

\begin{example}
\label{exm:ideal}
Suppose that $\cG$, $\Lambda^0$, $\Lambda^3$, $\Lambda^4$, $\Lambda^5$ and $\Lambda^6$ are as in \EXM{route}
and that the sign labeling $\tau$ is defined as in \FIG{route},
where $+/-$ stands for $1/-1$.
Let us consider the $\tau$-ideal
$\bI_\tau(\Lambda^3)=
\bas{
\tau_{06},~
\tau_{41},~
\tau_{45},~
\tau_{46},~
\tau_{56},~
\tau_{032},~
\tau_{034},~
\tau_{036}
}$, where
\begin{gather*}
\tau_{06}=a_{06}+b_{06},\qquad
\tau_{41}=a_{14}-b_{14},\qquad
\tau_{45}=a_{45}+b_{45},\qquad
\\
\tau_{46}=a_{46}-b_{46},\qquad
\tau_{56}=a_{56}+b_{56},
\\
\setlength{\arraycolsep}{1mm}
\begin{array}{rclcl}
\tau_{032}&=&(a_{03}+b_{03})+(a_{32}-b_{32}) &=& b_0-2b_3+1,\\
\tau_{034}&=&(a_{03}+b_{03})+(a_{34}-b_{34}) &=& -a_4+b_0-2b_3+b_4,\\
\tau_{036}&=&(a_{03}+b_{03})+(a_{36}+b_{36}) &=& -a_6+b_0-b_6.\\
\end{array}
\end{gather*}
By applying Gaussian elimination we find that
\[
\bI_\tau(\Lambda^3)\cap \bhA(4)
=
\bas{
\tau_{06},~\tau_{56},~\tau_{032},~\tau_{036},~\tau_{0341},~\tau_{0345},~\tau_{0346}},
\]
where
\[
\tau_{0341}=\tau_{034}+\tau_{41},\qquad
\tau_{0345}=\tau_{034}+\tau_{45}\quad\text{and}\quad
\tau_{0346}=\tau_{034}+\tau_{46}.
\]
Since $034\in \max(\Lambda^3_4\cap \Lambda^3_0)$
and $\tau_{034v}=\tau_{034+4v}$ for $v\in\{1,5,6\}$ it
follows from the definition of $\Delta_4$ that
$\bI_\tau(\Lambda^4)=\bI_\tau(\Lambda^3)\cap \bhA(4)$.
In fact, we find that
$\bI_\tau(\Lambda^i)=\bI_\tau(\Lambda^j)\cap \bhA(i)$ for all $(i,j)\in\{(0,3),(3,4),(4,5),(5,6)\}$
and thus
\[
\bI_\tau(\Lambda^6)=\bI(\Lambda)\cap\bhA(3,4,5,6)=\bI(\Lambda)\cap\bB.
\]
Let us compute the following $\tau$-polynomials in $\bI_\tau(\Lambda^5)$ and $\bI_\tau(\Lambda^6)$:
\begin{gather*}
\setlength{\arraycolsep}{1mm}
\begin{array}{rclcl}
\tau_{0346}    &=&\tau_{034}+\tau_{46}          &=& -a_6+b_0-2b_3+b_6,\\
\tau_{03456}   &=&\tau_{034}+\tau_{45}+\tau_{56} &=& -a_6+b_0-2b_3+2b_4-b_6, \\
\tau_{03456430}&=&\tau_{03456}-\tau_{0346}       &=& 2b_4-2b_6.\\
\end{array}
\end{gather*}
Notice that $03456,0346 \in \Lambda^5$ and $03456430=03456+0346 \in \Lambda^6$.
We choose the walk~$641$
so that $03456+641$ and $0346+641$ are both western and
\[
\tau_{03456430}=
\tau_{03456}-\tau_{0346}=
\tau_{03456+641}-\tau_{0346+641}.
\]
Thus, the round polynomial $\tau_{03456430}$ can be written as a difference of
western polynomials in the ideal~$\bI(\Lambda)$.
Notice that we can replace $641$
by any walk $\rho$ \st $\start(\rho)=6$ and $\wend(\rho)=1$.
The remaining two round walks in $\Lambda^6$ can also be
written as a difference of western polynomials:
\[
\tau_{034560}=\tau_{03456+641}-\tau_{06+641}\quad\text{and}\quad
\tau_{0345630}=\tau_{03456+641}-\tau_{036+641}.
\]
Since $\bI(\Lambda^6)=\bI(\Lambda)\cap\bB$, we
conclude that $\bI(\Lambda^6)$ is generated by eastern and western polynomials:
\[
\bI(\Lambda^6)=\bas{
\tau_{032},~
\tau_{0341},~
\tau_{06+641},~
\tau_{036+641},~
\tau_{0346+641},~
\tau_{03456+641}
}.
\]
The constructions in this example generalize to any calligraph, route
and sign labeling.
\END
\end{example}

\begin{lemma}
\label{lem:Delta}
If $\Lambda$ is a set of walks and $\tau$ a sign labeling, then
\[
\bI_\tau(\Delta_v(\Lambda))=\bI_\tau(\Lambda)\cap\bhA(v)
\quad\text{for all}\quad v\in \V.
\]
\end{lemma}

\begin{proof}
Direct application of Gaussian elimination as
is explained in \EXM{ideal}.
\end{proof}

\begin{lemma}
\label{lem:dan}
If $f$ is a round $\tau$-polynomial for some sign labeling~$\tau$,
then there exist two western $\tau$-polynomials $g$ and $h$ \st $f=g-h$.
\end{lemma}

\begin{proof}
Recall that if $f=\tau_{03456430}$ as in \EXM{route},
then $g=\tau_{0345641}$ and $h=\tau_{034641}$.
We conclude the proof as this construction directly
generalizes to any calligraph and sign labeling.
\end{proof}

\begin{proposition}
\label{prp:B}
If $\Lambda$ is an initial set of walks
and $\tau$ is a sign labeling, then the
elimination ideal $\bI_\tau(\Lambda)\cap\bB$ is generated by the
set of all $\tau$-polynomials that are either western or eastern.
\end{proposition}

\begin{proof}
Suppose that $(v_0,\ldots,v_n)$ is a route and
let $Q=\Delta_{v_n}\circ\cdots\circ\Delta_{v_1}\circ\Delta_{v_0}(\Lambda)$.
We know from \LEM{wed} that each walk in $Q$ is either western, eastern
or round.
We deduce from \LEM{Delta} that $\bI_\tau(Q)=\bI_\tau(\Lambda)\cap \bB$.
By \LEM{dan} all polynomials in $\set{\tau_\rho}{\rho\in Q}$
are generated by western and eastern polynomials
and thus we concluded the proof.
\end{proof}

If $\tau$ is a sign labeling and $\rho=(\rho_0,\ldots,\rho_r)$ an eastern or western walk,
then we define the function $\chi_{\tau,\rho}\c \rho\to\{-2,0,2\}$ as
\begin{align*}
\chi_{\tau,\rho}(\rho_i):=
\begin{cases}
\tau(\{\rho_i,\rho_{i+1}\})-\tau(\{\rho_{i-1},\rho_i\}) & \text{if } i\notin\{0,r\},\\
0                 & \text{otherwise}.
\end{cases}
\end{align*}
Notice that $\chi_{\tau,\rho}$ is well defined
and only attains values in $\{-2,0,2\}$.

\begin{example}
\label{exm:flip}
Suppose that $\cG$ with sign labeling $\tau$ is defined as in \FIG{route}
and let $\rho:=03645641$ be a western walk.
We simplify the $\tau$-polynomial of $\rho$ as follows,
where we used that
$a_{ij}+a_{jk}=a_{ik}$
and
$b_{ij}+b_{jk}=b_{ik}$:
\[
\begin{array}{rl}
\tau_\rho=& (a_{03}+b_{03})+(a_{36}+b_{36})+(a_{64}-b_{64})+(a_{45}+b_{45})\\
                &+(a_{56}+b_{56})+(a_{64}-b_{64})+(a_{41}-b_{41})\\
               =& a_{03}+a_{36}+a_{64}+a_{45}+a_{56}+a_{64}+a_{41}\\
                & +b_{03}+b_{36} -b_{64} +b_{45}+b_{56} -b_{64}-b_{41}
              ~=~  b_{06}        -b_{64} +b_{46}        -b_{61}\\
               =& b_0-2b_6+2b_4-2b_6.\\
\end{array}
\]
Notice that the places where $\chi_{\tau,\rho}$ attains $-2$ and $2$
are underlined at
\[
03\underline{6}45\underline{6}41
\quad\text{and}\quad
036\underline{4}5641,
\quad\text{\resp}.
\]
Thus the coefficients of $\tau_\rho$ are determined by $\chi_{\tau,\rho}$.
The next lemma states that this holds for any eastern or western walk.
\END
\end{example}

\begin{lemma}
\label{lem:flip}
If $\tau$ is a sign labeling and
$\rho$ is an either eastern or western walk,
then there exists $\epsilon\in\{0,1\}$ \st
\[
\tau_\rho=b_0+\sum_{v\in\rho} \chi_{\tau,\rho}(v)\,b_v+\epsilon.
\]
\end{lemma}

\begin{proof}
Straightforward consequence of the definitions (see \EXM{flip}).
Notice that $\epsilon=0$ or $\epsilon=1$ if $\wend(\rho)=1$ and $\wend(\rho)=2$, \resp.
\end{proof}

\begin{proposition}
\label{prp:b0}
If $\Lambda$ is an initial set of walks and
$\tau$ a sign labeling,
then $\bI_\tau(\Lambda)\cap \C[b_0]$ is equal to either $\bas{0}$, $\bas{1}$, $\bas{b_0}$
or $\bas{b_0+1}$.
\end{proposition}

\begin{proof}
We know from \PRP{B} that $\bI_\tau(\Lambda)\cap \bB=\bas{Q}$,
where $Q$ is a set of $\tau$-polynomials that are either eastern or western.
If there is an eastern or western walk~$\rho$ whose edges all have sign~1, then \LEM{flip} yields
that $\tau_\rho=b_0+\epsilon\in Q\subset \bI_\tau(\Lambda)\cap \bB$ for
some~$\epsilon\in\{0,1\}$
and thus $\bI_\tau(\Lambda)\cap \C[b_0]\in\{\bas{1},~\bas{b_0},~\bas{b_0+1}\}$.
Now let us assume that no such walk exists.
Let $K$ be the largest connected subgraph of~$\cG$ which contains vertex $0$ and only edges with positive sign.
Notice that $1,2\notin K$ by assumption.
Let $f\c\V\to \bB$ be defined as
\begin{align*}
	f(v):=
	\begin{cases}
		b_v+\tfrac{1}{2}b_0 & \text{if } v\in K\setminus\{0\},\\
		b_v                 & \text{otherwise}.
	\end{cases}
\end{align*}
The linear isomorphism $\lambda\c\C^n\to\C^n$
sends $(b_{v_1},\ldots,b_{v_n})$ to $(f(v_1),\ldots,f(v_n))$,
where $\V=\{v_1,\ldots,v_n\}$ and $v_1<\ldots<v_n$ with $v_1=0$
(here we consider $\bB$ as a set of polynomial functions).
By \LEM{flip} there exists for all~$\tau_\rho\in Q$ an $\epsilon\in\{0,1\}$ \st
\begin{align*}
\tau_\rho\circ\lambda
&=
b_0
+
\sum_{v\in \rho\cap K} \chi_{\tau,\rho}(v)\,\bigl(b_v+\tfrac{1}{2}b_0\bigr)
+
\sum_{v\in \rho\setminus K} \chi_{\tau,\rho}(v)\,b_v
+
\epsilon
\\
&=
b_0 - 2(b_{\rho_\iota} + \tfrac{1}{2}b_0) + \sum_{v\in \rho\setminus\{\rho_\iota\}}\chi_{\tau,\rho}(v)\,b_v+\epsilon,
\end{align*}
where $\iota$ is the first index in $\rho$ \st $\rho_\iota\in K$
and $\rho_{\iota+1}\not\in K$.
Indeed, if a walk goes back into $K$ it also needs to go out again, so all further transformations of $\lambda$ cancel out:
$|\set{v\in \rho\cap K}{\chi_{\tau,\rho}(v)=-2}|=|\set{v\in \rho\cap K}{\chi_{\tau,\rho}(v)=2}|+1$.
We deduce that $\tau_\rho\circ\lambda\in \C[\Upsilon_\bB\setminus\{b_0\}]$
for all~$\tau_\rho\in Q$.
The ideal of $\lambda^{-1}(V_Q)$ is generated by
$\set{\tau_\rho\circ\lambda}{\tau_\rho\in Q}$, where $V_Q$ denotes the zero set of~$Q$.
Hence, we observe that $\kappa\circ \lambda^{-1}(V_Q)=\C$,
where the linear projection~$\kappa\c \C^n\to \C$ sends $(b_{v_1},\ldots,b_{v_n})$ to $b_{v_1}=b_0$.
Since $\kappa=\kappa\circ \lambda^{-1}$ we find that $\kappa(V_Q)=\C$ as well.
The ideal of the projection~$\kappa(V_Q)$ is equal
to the elimination ideal~$\bI_\tau(\Lambda)\cap \C[b_0]$ by \LEM{G}
and thus $\bI_\tau(\Lambda)\cap \C[b_0]=\bas{0}$.
Since we considered all cases we concluded the proof.
\end{proof}

\begin{proof}[Proof of \PRP{vert}.]
For all $f\in P=T\circ M\circ H\circ \mu(\E)$ as listed in \TAB{eqn}
we consider the substitution~$f\sub{x_0\to 0}$.
Next we make the identification
$a_0=0$,
$a_i=-x_i$ and
$b_i=-\ii\, y_i$ for $i>0$,
where $(x_1,y_1):=(0,0)$, $(x_2,y_2):=(1,0)$
and $x_i, y_i\in \Upsilon_R$ for $i\in\Z_{\geq 0}\setminus\{1,2\}$.
We find that the ideal of~$P\cup \{x_0\}$
is generated by $\set{\varphi(e)}{e\in \E}\cup\{x_0\}$,
where
\[
\varphi(\{i,j\}):=
\begin{cases}
a_{0i}+b_{0i}                        & \text{ if } 0\in \{i,j\},\\
(a_{ij}+b_{ij})\cdot(a_{ij}- b_{ij}) & \text{ if } 0\notin \{i,j\} \text{ with } i<j.\\
\end{cases}
\]
As each generator in $\set{\varphi(e)}{e\in \E}\cup\{x_0\}$ splits into linear factors,
we find that $V(P\cup \{x_0\})$
is equal to a union of linear spaces~$W_1\cup\cdots\cup W_r$.
Let us consider the ideal $\bI_\tau(\Lambda)\subset \C[\Upsilon_\bA]$
via the inclusion $\Upsilon_\bA\subset \Upsilon_R$
as an ideal in the ring~$\C[\Upsilon_R]$.
Under this identification there exist
for all~$1\leq i\leq r$ a sign labeling~$\tau$
\st $\bI_\tau(\Lambda)+\bas{x_0}$ is the ideal of~$W_i$.
Hence, $\left(\bI_\tau(\Lambda)+\bas{x_0}\right)\cap \C[b_0]$
is the ideal of~$\pi(W_i)$ by \LEM{G}.
We have $\left(\bI_\tau(\Lambda)+\bas{x_0}\right)\cap \C[b_0]=\left(\bI_\tau(\Lambda)\cap \C[b_0]\right)+\bas{x_0}$
and thus the proof is concluded by \PRP{b0}.
\end{proof}

\clearpage
\bibliography{graph-splitting}

Georg Grasegger
\\
Johann Radon Institute for Computational and Applied Mathematics (RICAM),
\\Austrian Academy of Sciences
\\
georg.grasegger@ricam.oeaw.ac.at

Boulos El Hilany
\\
Institut f\"ur Analysis und Algebra, TU Braunschweig
\\
{b.el-hilany@tu-braunschweig.de}
\\
\href{https://boulos-elhilany.com}{boulos-elhilany.com}

Niels Lubbes
\\
Johann Radon Institute for Computational and Applied Mathematics (RICAM),
\\Austrian Academy of Sciences
\\
info@nielslubbes.com

\end{document}